%% file: sn-article.tex
\documentclass[pdflatex,sn-mathphys-ay]{sn-jnl}

\usepackage{graphicx}%
\usepackage{multirow}%
\usepackage{amsmath,amssymb,amsfonts}%
\usepackage{amsthm}%
\usepackage{mathrsfs}%
\usepackage[title]{appendix}%
\usepackage{xcolor}%
\usepackage{textcomp}%
\usepackage{manyfoot}%
\usepackage{booktabs}%
\usepackage{algorithm}%
\usepackage{algorithmicx}%
\usepackage{algpseudocode}%
\usepackage{listings}%

\numberwithin{table}{section}
\theoremstyle{thmstyleone}
\newtheorem{lemma}{Lemma}[section]
\newtheorem{theorem}[lemma]{Theorem}
\newtheorem{definition}[lemma]{Definition}

\newtheorem{proposition}[lemma]{Proposition}
\newtheorem{remark}[lemma]{Remark}

\usepackage{tikz}
\usetikzlibrary{decorations.pathreplacing}
\usepackage{xspace}
\usepackage{adjustbox}
\usepackage[all]{xy}
\usepackage{centernot}
\renewcommand{\d}{\operatorname{d}}
\newcommand{\Z}{\mathbb{Z}}
\newcommand{\del}{\partial}

\newcommand{\N}{\mathbb{N}}
\newcommand{\R}{\mathbb{R}}
\newcommand{\C}{\mathbb{C}}

\newcommand{\im}{\operatorname{im}}
\newcommand{\Vol}{\operatorname{Vol}}
\renewcommand{\O}{\mathcal{O}}

\renewcommand{\H}{\mathbb{H}}
\newcommand{\D}{\centernot{D}}
\newcommand{\Tf}{\operatorname{Tf}}
\newcommand{\Mat}{\operatorname{Mat}}

\newcommand{\Sym}{\operatorname{Sym}}

\newcommand{\hk}{hyperk\"ahler\xspace}

\newcommand{\mycite}[1]{\cite{#1}}

\begin{document}

\title[Article Title]{Construction of gravitational instantons with non-maximal volume growth via codimension-1 collapse}

\author*[1]{\fnm{Willem Adriaan} \sur{Salm}}\email{willem.adriaan.salm@ulb.be}

\affil*[1]{\orgdiv{Department of Mathematics}, \orgname{Université Libre de Bruxelles}, \orgaddress{\street{Boulevard du Triomphe}, \city{Brussels}, \postcode{1050}, \country{Belgium}}}

\abstract{In this paper, we construct families of gravitational instantons of type ALG, ALG*, ALH and ALH* using a gluing construction. Away from a finite set of exceptional points, the metric collapses with bounded curvature to a quotient of $\R^3$ by $\Z_2$ and a lattice of rank one or two. Depending on whether the gravitational instantons are of type ALG/ALG* or ALH/ALH*, there are either two or four exceptional points respectively that are modelled on the Atiyah-Hitchin manifold. The other exceptional points are modelled on the Taub-NUT metric. There are at most four, respectively eight, of these points in each case.}


\maketitle
{
\section{Introduction}
Gravitational instantons are complete, non-compact \hk manifolds of dimension four with $L^2$ bounded curvature. In the late 70's, \mycite{Eguchi1978} gave the first non-trivial examples of these spaces. 
Over the years, many other constructions of gravitational instantons were found: \mycite{Kronheimer1989a} constructed gravitational instantons using \hk quotients. \mycite{Atiyah1988} constructed a gravitational instanton using gauge theory. 
\mycite{Ivanov1996}, \mycite{Lindstroem1988}, \mycite{Cherkis1999} and \mycite{Cherkis2005} constructed gravitational instantons using twistor methods. 
\cite{Biquard2011} constructed examples using a gluing construction.
\cite{Tian1990} and \cite{Hein2010b,Hein2010a} created examples by solving the complex Monge–Ampère equation.
Finally, \mycite{Sun2021} classified all gravitational instantons.

According to \mycite{Chen2019}, there are four classes of gravitational instantons with $r^{-2 - \epsilon}$ curvature decay. These classes are called ALE (Asymptotically Locally Euclidean), ALF (Asymptotically Locally Flat), ALG and ALH and they are distinguished by their \textit{volume growth}. That is, in a gravitational instanton of type ALE, ALF, ALG or ALH, a ball of radius $r\gg 1$ will have volume of order $r^4$, $r^3$, $r^2$ or $r$ respectively. 
If the curvature of a gravitational instanton does not decay with $\O(r^{-2 - \epsilon})$ rate, it can only be in one of two classes: ALG* or ALH*. In this case, the volume growth will be of order $r^2$ or $r^{4/3}$ respectively.

Although there are now many different constructions, each construction is tailored to a specific class of gravitational instantons. In this paper we show that the gluing construction carried out by \mycite{Schroers2020} for ALF gravitational instantons, can be extended to gravitational instantons of type ALG, ALG*, ALH and ALH*. Moreover, we show that this can be done in an explicit, systematic and uniform way, even if, from a geometric and analytic viewpoint, these gravitational instantons behave quite differently.

This construction will also be useful in understanding the boundary of the moduli space of gravitational instantons. As this moduli space is not compact, it is possible that families of gravitational instantons degenerate. Like in the work of \mycite{Foscolo2016}, we will focus on the limit where a family of gravitational instantons collapse to a flat 3-dimensional space. In order to describe this process explicitly, we will construct a 4-manifold which has the structure of a  circle bundle almost everywhere. We will equip it with a metric that is approximately \hk and we will show that when the circle fibres are sufficiently small, this metric can be perturbed into a genuine \hk metric.

\subsection{Results}
\label{sec:results}
The main result of this paper is as follows:
\begin{theorem}
	\label{main-theorem}
	Let $L \subset \R^3$ be a lattice of rank one or two and consider the $\Z_2$ action on $\R^3 / L$ that is induced by the antipodal map on $\R^3$. Let $\{p_i\}$ be a configuration of $n$ distinct points in $(\R^3/L - \operatorname{Fix}(\Z_2))/ \Z_2$. Suppose that $n \le 4$ when $\R^3/L \simeq \R^2 \times S^1$ and $n \le 8$ when $\R^3/L \simeq \R \times T^2$. Then, there exists an $\epsilon_0 > 0$, such that for all $0<\epsilon< \epsilon_0$ there exists a gravitational instanton $(M_{\R^3/L, n}, g_\epsilon)$ with the following properties:
	\begin{enumerate}
		\item For each fixed point of the $\Z_2$ action on $\R^3 / L$, there is a compact set $K \subset M_{\R^3 / L, n}$, such that $\epsilon^{-2} g_\epsilon$ approximates the Atiyah-Hitchin metric on $K$ as $\epsilon \to 0$.
		\item For each $i \in \{1, \ldots, n\}$, there is a compact set $K_i \subset M_{\R^3 / L, n}$ such that $\epsilon^{-2}g_\epsilon$ approximates the Taub-NUT metric on $K_i$ as $\epsilon \to 0$.
		\item Away from the singularities, the manifold collapses to $(\R^3 / L)/\Z_2$ with bounded curvature as $\epsilon$ converges to zero.
	\end{enumerate} 
\end{theorem}

We show this theorem in two steps. In Section \ref{sec:approximate-solution}, we will construct a 1-parameter family of Riemannian 4-manifolds using the data specified in Theorem \ref{main-theorem}. 
We construct this family such that conditions 1 to 3 are satisfied. We show that it is \hk outside some small annular regions and for each annular region, we give an explicit error estimate in Theorem \ref{thm:almost-hk:global-symplectic-triple}.

In Section \ref{sec:deformation-problem}, we set up the deformation problem and show that, using the analysis done in \cite{Salm001}, the approximate solution can be perturbed into a gravitational instanton. For this, we have to set up the inverse function theorem: the largest part of this section will be proving that the linearized operator has a uniformly bounded inverse.

Finally, in Section \ref{sec:setup:topology}, we study the global properties of our gravitational instantons. Namely, we will calculate the topology of our manifolds, and show
\begin{proposition}
	\label{prop:topology:R2xS1}
	The homology of $M_{\R^2 \times S^1, n}$ is given by
	$$
	H_k(M_{\R^2 \times S^1, n}) = \begin{cases}
		\Z &\text{if } k=0 \\
		\Z_2 &\text{if } k=1 \text{ and } n = 0 \\
		\Z^{n+1} &\text{if } k=2 \\
		0 & \text{otherwise}.
	\end{cases}
	$$
	For $n$ equal to $2$, $3$ or $4$, the intersection matrix is given by the negative Cartan matrix of the Dynkin diagram of type $A_1 + \tilde{A}_1$, $\tilde{A}_3$ or $\tilde{D}_4$ respectively.
\end{proposition}
\begin{proposition}
	\label{prop:topology:RxT2}
	The homology of $M_{\R \times T^2, n}$ is given by
	$$
	H_k(M_{\R \times T^2, n}) = \begin{cases}
		\Z &\text{if } k=0 \\
		\Z_2 &\text{if } k=1 \text{ and } n = 0 \\
		\Z^{n+3} &\text{if } k=2 \\
		0 & \text{otherwise}.
	\end{cases}
	$$
\end{proposition}

In the second part of Section \ref{sec:setup:topology}, we determine the type of our gravitational instantons. Using the overview in \mycite{Sun2021}, we conclude that
\begin{theorem}
	\label{main-theorem-2}
	Let $M_{\R^3 / L, n}$ be the gravitational instanton given in Theorem \ref{main-theorem}. Then, depending on the lattice and $n$, the asymptotic metric can be classified as
	\begin{itemize}
		\item ALG*-$I^*_{4-n}$ when $\dim L = 1$ and $n < 4$,
		\item ALG${}_{\frac{1}{2}}$ when $\dim L = 1$ and $n = 4$,
		\item ALH*-${I}_{8-n}$ when $\dim L = 2$ and $n < 8$,
		\item ALH when $\dim L = 2$ and $n = 8$.
	\end{itemize}
\end{theorem}

Finally, we count the degrees of freedom in our construction, and compare it with the dimension of the moduli spaces. We also compare our construction of ALH* gravitational instantons with the generalized Tian-Yau construction\footnote{I.e., the generalisation that is explained in \cite{Hein2021}}, where we conclude
\begin{proposition}
	\label{main-proposition-3}
	Let $L \subset \R^3$ be a lattice of rank two, $n \in \{ 0,1, \ldots, 7\}$ and let $M_{\R^3 / L, n}$ be the ALH*-gravitational instanton given in Theorem \ref{main-theorem}.
	\begin{enumerate}
		\item For $1 \le n < 8$, the space $M_{\R \times T^2, n}$ is diffeomorphic to the complement of a smooth anticanonical divisor of the blowup of $\C P^2$ at $8-n$ points.
		\item The space $M_{\R \times T^2, 0}$ is diffeomorphic to
		the complement of a smooth anticanonical divisor of $S^2 \times S^2$.
	\end{enumerate}
\end{proposition}

\noindent
\textbf{Acknowledgements.} This paper contains the geometric results found in my PhD thesis ``Construction of gravitational instantons with non-maximal volume growth'', funded by the Royal Society research grant RGF\textbackslash R1\textbackslash 180086. A special thanks goes to the author’s supervisor Lorenzo Foscolo, without whom his work would not have been possible.

\section{The approximate solution}
\label{sec:approximate-solution}
In this section, we will construct a 1-parameter family of Riemannian 4-manifolds that is approximately \hk by gluing known examples. First, we need to recall the construction due to \mycite{Gibbons1978}.
In modern language, we need a flat $3$-dimensional Riemannian manifold $(U, g_U)$ with a parallel orthonormal frame $\{e_i\}$, a principal $S^1$-bundle $P$ over $U$, a connection $\eta$ on $P$, and a harmonic function $h \colon U \to (0,\infty)$ satisfying $*^{g_U} \d h = \d \eta$.
The metric
$$
g^{GH} := h \cdot g_U + h^{-1} \eta^2
$$
with the $2$-forms
$$
\omega^{GH}_i := e_i \wedge \eta + h *^{g_U} e_i
$$
is \hk and we call this construction the Gibbons-Hawking ansatz. 

Some of the requirements above are redundant. Namely, if $H^2(U, \Z)$ has no torsion, then $P$ is fully determined by the cohomology class of $[\d \eta] = [* \d h]$. Moreover, every principal bundle admits a connection $\eta$ and this connection can be chosen such that $\d \eta = * \d h$. 

One important example of the Gibbons-Hawking ansatz is the Taub-NUT metric.
The Taub-NUT metric with mass $k$ is the 1-point completion of the Gibbons-Hawking metric for $h(x) = c + \frac{k}{2 |x|}$ on $U = \R^3 \setminus \{0\}$. It is a smooth \hk manifold if $k = 1$. (Hence, if the mass is not specified, one assumes that $k=1$.) 

Another important example of a gravitational instanton is due to \mycite{Atiyah1988}.
Topologically, this space can be viewed as a cohomogeneity-one manifold\footnote{A good exposition about this can be found in \cite{Schroers2020}.}. Namely, the Atiyah-Hitchin manifold has a natural $SO(3)$ action and its quotient with this $SO(3)$ action is the half line $[\pi,\infty)$. A generic orbit is of the form $SU(2)/ \langle i, j,k \rangle$, while the orbit over the endpoint of the half line can be identified with $SU(2)/ \langle e^{k \phi}, j \rangle \simeq \R P^2$. Hence, the Atiyah-Hitchin manifold retracts to $\R P^2$.

The Atiyah-Hitchin manifold has a metric that is \hk.
To understand its asymptotics, one has to consider its branched double cover. With the above identifications of the fibres, the $\Z_2$-action of this branched double cover is given by the $j$-multiplication. A generic fibre on this branched double cover is of the form $SU(2)/ \langle i \rangle$, which is a degree $-4$ circle bundle over $S^2$. When the radial parameter is sufficiently large, the metric of the Atiyah-Hitchin manifold on this branched double cover approximates the Taub-NUT metric with mass $-4$ with exponentially small error. Also, on this circle bundle over $S^2$, the $j$-action descends to the antipodal map on the base space and it identifies the fibres by reversing the orientation, i.e. for any $x$ in this circle bundle, $ e^{i \phi} \cdot j \cdot x = j \cdot e^{- i \phi} \cdot x$.

\subsection{The bulk space}
\label{sec:setup:bulk-space}
For our gluing construction we will follow a method proposed by \mycite{Sen1997}:
Namely, he started with the Gibbons-Hawking Ansatz on a punctured $\R^3$ with the harmonic function
$$
h(x) = 1 + \frac{-4}{2|x|} + \sum_i \frac{1}{2 |x + p_i|} + \frac{1}{2 |x - p_i|},
$$
where $p_i \in \R^3 \setminus \{0\}$ are distinct. Near each singularity $p_i$ the Gibbons-Hawking metric approximates the Taub-NUT metric and near the origin it approximates the Taub-NUT metric with mass $-4$.

Secondly, he considered the antipodal map on $\R^3$ and he lifted this involution such that near the origin it coincides with the $\Z_2$-action on the branched double cover of the Atiyah-Hitchin manifold. After taking the $\Z_2$-quotient, he claimed that this \hk space can be made complete by gluing in the Atiyah-Hitchin manifold near the origin and the Taub-NUT space near each $p_i$. Although Sen never did this gluing explicitly, \mycite{Schroers2020} formalised his argument in their quest of finding geometric models of matter \cite{Atiyah2011}.

We start our gluing construction with the construction of the bulk spaces.
That is, we generalize Sen's use of the Gibbons-Hawking ansatz on $\R^3$ to $\R^3$ modulo a non-maximal lattice. In Section \ref{sec:almost-hk} we make this bulk space complete and equip it with an almost \hk metric.

\begin{definition}
	\label{def:setup:basespace}
	Fix once and for all a non-maximal lattice $L$ on $\R^3$. 
	\begin{itemize}
		\item 	We call the quotient $B:= \R^3/L$, endowed with the flat metric, the base space. 
		\item 	We refer to the map $\tau\colon B \to B$ that is induced from the map $x \mapsto -x $ on $\R^3$ as the antipodal map.
		\item  We denote the fixed point set of $\tau$ by $\{q_j\}$ and we call $q_j$ a fixed point or a fixed point singularity. Unless specified otherwise, we call the action induced by $\tau$ on $B$, the $\Z_2$ action on $B$.
	\end{itemize}
\end{definition}

Because $L$ is non-maximal, the base space $B$ can only be diffeomorphic to $\R^3$, $\R^2 \times S^1$ or $\R\times T^2$. As explained in \cite{Salm001}, each case yields different kinds of gravitational instantons and will require different kinds of analysis. To distinguish these cases we often will write $B = \R^3$, $B = \R^2 \times S^1$ or $B = \R \times T^2$.

\noindent
When the lattice $L$ is trivial, the only fixed point is the origin. However, in the other cases we have two or four fixed points respectively. Just as in \cite{Sen1997}, we pick a certain number of points which we remove from $B$. It turns out we need to impose some restrictions on the number of singularities, which are given in the following definition. In Lemma \ref{lem:setup:harmonic-function} we see the necessity of these requirements. 

\begin{definition}
	\label{def:setup:non-fixed-points}
	Fix once and for all a finite set of points $\{p_i\} \in (B \setminus \{q_j\})/\Z_2$.  We call an element $p_i$ a non-fixed point or a non-fixed singularity. When $B = \R^2 \times S^1$ or $B = \R \times T^2$, the maximum number of non-fixed points must not exceed four or eight respectively.
\end{definition}
With abuse of notation, we also denote $ \pm p_i $ as the lift of the non-fixed points to $B$.

Like in the work of \cite{Sen1997}, we need to consider the punctured base space before we can consider the Gibbons--Hawking ansatz. Due to technical reasons we need to remove small balls around the fixed-point singularities. This reason will become apparent in Lemma \ref{lem:setup:harmonic-function}. 
\begin{definition}
	Let $\epsilon > 0$ and let $\overline{B}_{4 \epsilon}(q_j)$ be closed small balls of radius $4 \epsilon$ inside $B$ centred around $q_j$. Let $\cup_j \overline{B}_{4 \epsilon}(q_j)$ be their union inside $B$. Let $\cup_i \{\pm p_i\}$ be the union of the non-fixed points in $B$.
	We define the punctured base space $B' \subset B$ as the complement of $\cup_i \{\pm p_i\}$ and $\cup_j \overline{B}_{4 \epsilon}(q_j)$. Explicitly,
	\begin{equation}
	\label{eq:setup:Bprime}
	B' = B \setminus \left(\cup_i  \{\pm p_i\} \bigcup \cup_j \overline{B}_{4 \epsilon}(q_j) \right).
\end{equation}	
\end{definition}

\begin{remark}
	The exact choice of the radius for $\overline{B}_{4 \epsilon}(q_j)$ in $B'$ will be determined later when we study the gluing in more detail. For now it is sufficient that the radius is small enough such that the balls $\overline{B}_{4\epsilon}(q_j)$ are pairwise disjoint and $B'$ is connected.
\end{remark}

\subsubsection{Harmonic function}
Next, we need to construct a positive harmonic function $h$ with the correct asymptotics near the singularities. For the Taub-NUT metric this requires that the harmonic function must diverge as $\frac{1}{2|x - p_i|}$ at $\pm p_i$. 
For the Atiyah-Hitchin metric this requires that the harmonic function must diverge as $\frac{-2}{|x-q_j|}$ at $q_j$. Recall that $G(x,x') = \frac{1}{4 \pi |x - x'|}$ is the Green's function\footnote{In this paper we geometers Laplacian, i.e. $\Delta = \d ^* \d + \d \d ^*.$} on $\R^3$. Hence, viewing the harmonic function $h$ as a distribution we need that $\Delta h = 2 \pi \delta(x - p_i)$ near all $p_i$ and $\Delta h = -8 \pi \delta(x - q_j)$ near all $q_j$. A $\Z$-linear combination of Green's functions will satisfy all these conditions. 
It turns out that this is the only solution up to a constant. Before we prove this, we first revisit the Green's function on $\R^3 / L$.

\begin{lemma}
	\label{lem:setup:Greens-function}
	Let $p \in B$. There exists a smooth function $G$ on $B \setminus \{p\}$ that solves
	$$
	\Delta G = 2 \pi \delta(x - p) \text{ on } B,
	$$
	that is invariant under the involution centred at $p$, and has the following asymptotic expansion near infinity:
$$
		G(x,y,z) = \begin{cases}
				\frac{1}{2 \sqrt{x^2 + y^2 + z^2}} + \O((x^2 + y^2 + z^2)^{-1}) & B = \R^3 \\
				- \frac{1}{4 \pi \Vol(S^1)} \log(x^2 + y^2) + \O(e^{- 2 \pi  \sqrt{x^2 + y^2}/ \Vol(S^1)} ) & B = \R^2 \times S^1 \\
				- \frac{\pi}{\Vol(T^2)} |x| + \O(e^{-|x|}) & B = \R \times T^2.
			\end{cases}$$
\end{lemma}
\begin{proof}
%
\noindent
On $\R^3$, the function $G$ can be found explicitly and is given by $G(x,y,z) = \frac{1}{2 |x|}$. For the rest of this proof we work out the case $B = \R \times T^2$. The case $B = \R^2 \times S^1$ is similar and an alternative proof can be found in \cite{Gross2000}, Lemma 3.1.
\\
\raggedbottom

\noindent
We equip $\R \times T^2$ with coordinates $(x,y,z)$ such that the metric is given by $g = \d x^2 + g_{T^2}$. Without loss of generality we assume that $p$ is at $(0,0,0)$. Let $N \in \N$ and consider the function
$$
G^N(x,y,z) = \frac{\pi}{\Vol(T^2)} \left[
- |x| + \sum_{0 < |(m,n)| < N} \frac{1}{\sqrt{m^2 + n^2}} e^{- \sqrt{m^2 + n^2} |x|} e^{i (m y + n z)}
\right].
$$
The function $G^N + \frac{\pi}{\Vol(T^2)} |x|$ is an $L^2$ function on $\R \times T^2$, because the summand is exponentially decaying. By construction $\overline{G^N} = G^N$, and so $G^N$ is real valued. We claim that $G^N + \frac{\pi}{\Vol(T^2)} |x|$ is a Cauchy sequence in $L^2$. Indeed, for any $M > N$, 
	\begin{align*}
		\left\|G^M - G^N\right\|_{L^2(R \times T^2)}^2
		=& \frac{\pi^2}{\Vol(T^2)}\sum_{N \le |(m,n)| < M} \frac{1}{(m^2 + n^2)^{3/2}} = \O(N^{-1}).
	\end{align*}
	Define the limit of $G^N$ as $G$. We show that $G$ solves $\Delta G = 2 \pi \delta$ as a distribution. 
	For this, let $f \in C^{\infty}_c(\R \times T^2)$, $\delta > 0$ and $D = (\R \setminus B_\delta(0)) \times T^2$. We consider $\langle \Delta f, G\rangle_{L^2(D)}$ for some sufficiently small $\delta$. By integration by parts,
	\begin{align*}
		\langle \Delta f, G \rangle_{L^2(D)} 
		&= \lim\limits_{N \to \infty} \langle \Delta f, G^N \rangle_{L^2(D)} \\
		&= \lim\limits_{N \to \infty} \langle f, \Delta G^N \rangle_{L^2(D)} + \int_{S_\delta(0) \times T^2} \left(
		\overline{G^N} \frac{\del f}{\del x} - f \frac{\del \overline{G^N}}{\del x}
		\right)  \Vol_T^2.
	\end{align*}
	The function $G^N$ is defined as a sum of harmonic functions, and therefore $\langle f, \Delta G^N \rangle_{L^2(D)} = 0$. 	Next we consider the Fourier decompositions of $f$ and $G$, i.e. $f = \sum_{k,l \in \Z} \hat{f}_{kl}(x) e^{i(ky + lz)}$ and $G = \sum_{m,n \in \Z} \hat{G}_{mn} e^{i(my + nz)}$. This simplifies our calculation of $\langle \Delta f, G \rangle_{L^2(D)}$, as
	\begin{align*}
		\langle \Delta f, G \rangle_{L^2(D)} 
		=& \lim\limits_{N \to \infty} \sum_{\substack{k,l \in \Z \\ 0 < |(m,n)| < N}} \int_{S_\delta(0) \times T^2} \left(
		\overline{\hat{G}_{mn}} \frac{\del \hat{f}_{kl}}{\del x} - \hat{f}_{kl}\frac{\del \overline{\hat{G}_{mn}}}{\del x} e^{i((k-m)y + (l-n)z)} \Vol_T^2 
		\right) \\
		=& \lim\limits_{N \to \infty} \Vol(T^2) \sum_{|(m,n)| < N}  \left[
		\overline{\hat{G}_{mn}} \frac{\del \hat{f}_{mn}}{\del x} - \hat{f}_{mn}\frac{\del \overline{\hat{G}_{mn}}}{\del x}  
		\right]^\delta_{- \delta}.
	\end{align*}
	We evaluate the right hand side explicitly, which is
	\begin{align*}
		\langle \Delta f, G \rangle_{L^2(D)} 
		=&  
		- \delta  \pi \left(\frac{\del \hat{f}_{00}}{\del x}(\delta) 
		- \frac{\del \hat{f}_{00}}{\del x}(-\delta) \right)
		+ \pi \left(\hat{f}_{00}(\delta)
		+ \hat{f}_{00}(-\delta) \right)   \\
		&+ \pi \sum_{(m,n) \not = (0,0)} \frac{1}{\sqrt{m^2 + n^2}} e^{- \sqrt{m^2 + n^2} |\delta|}\left(\frac{\del \hat{f}_{mn}}{\del x}(\delta) - \frac{\del \hat{f}_{mn}}{\del x}(-\delta) \right) \\
		&+ \pi \sum_{(m,n) \not = (0,0)} e^{- \sqrt{m^2 + n^2} |\delta|}\left( 		+ \hat{f}_{mn}(\delta)
		+ \hat{f}_{mn}(-\delta) \right).
	\end{align*}
	When we take the limit $\delta \to 0$, we conclude
	\begin{align*}
		\langle \Delta f, G \rangle_{L^2(\R \times T^2)} 
		=&  
		2\pi \hat{f}_{00}(0) + 2\pi \sum_{(m,n) \not = (0,0)}  \hat{f}_{mn}(0) = 2 \pi f(0,0,0).
	\end{align*}
	Finally, we investigate whether $G$ is smooth outside $p$. For any open $U$ away from $p$, the function $G$ satisfies $\Delta G = 0$ as a distribution. According to \cite{Folland1995} Proposition 6.33, $G$ is an element of $W^{k,2}(U)$ for all $k \in \N$. By the Sobolev inequality, $G$ must be smooth.
\end{proof}

\noindent
Taking linear combinations of $G$ we are now able to construct the harmonic function that is required for the Gibbons-Hawking ansatz:

\begin{lemma}
	\label{lem:setup:harmonic-function}
	Write $\# \{p_i\}$ for the number of pairs $p_i$ in $(B - \{q_i\})/\Z_2$. Let $G$ be the Green's function defined on Lemma \ref{lem:setup:Greens-function} and consider
	$$
	h = - 4 \sum_j G(x - q_j) + \sum_i \left(G(x - p_i) + G(x+ p_i) \right).
	$$
	Let $r$ be the Euclidean distance from the origin on $\R^3$, $\R^2$ or $\R$ when $B = \R^3$, $B = \R^2 \times S^1$ or $B = \R \times T^2$ respectively.
	\begin{itemize}
		\item[(a)] Near infinity, $$
		h = \begin{cases}
			\frac{2 \cdot\#\{p_i\} - 4}{2 r} + \O(r^{-3}) &\text{if } B = \R^3 \\
			\beta \cdot (8 - 2 \cdot \#\{p_i\}) \cdot \log(r) + \O(r^{-2}) &\text{if } B = \R^2 \times S^1 \\
			\beta \cdot (16 - 2 \#\{p_i\}) \cdot r + \O(e^{- r}) &\text{if } B = \R \times T^2
		\end{cases}$$
		for some $\beta > 0$, and $\beta$ only depends on the lattice $L$.
		\item[(b)] Near the fixed points $q_j$, $h(x) = \alpha_j - \frac{2}{|x - q_j|} + \O(|x - q_j|^2)$ for some $\alpha_j  \in \R$. Near the non-fixed points $\pm p_i$, $h(x) = \alpha_i + \frac{1}{2 |x \mp p_i|} + \O(|x \mp p_i|)$ for some $\alpha_i \in \R$.
		\item[(c)] Denote the ball of radius $r$ centred at $x$ as $B_r(x)$. There exists an $\delta > 0$ such that $\epsilon^{-1} + h$ is a harmonic function on $B' = B \setminus \left(\cup_i  \{\pm p_i\} \bigcup \cup_j \overline{B}_{4 \epsilon}(q_j) \right)$ which is greater than $\frac{1}{2}$ for all $0 < \epsilon< \delta$.
		\item[(d)] The only maps that satisfy
		\begin{enumerate}
			\item $
			\Delta \tilde h = - 8 \pi \sum_j \delta(x - q_j) + 2 \pi \sum_i \delta(x - p_i) + \delta(x+ p_i)$,
			\item $\tilde h$ is bounded below on $B'$,
		\end{enumerate}
		are the maps $\tilde h = h + c$ for some constant $c \in \R$.
	\end{itemize}
\end{lemma}
\begin{remark}
	The choice of $B'$ can be explained from part (b) and (c) of this lemma. 
	According to part (b), the function $h$ diverges to $- \infty$ near $q_j$, degenerating the Gibbons-Hawking metric. In part (c) we show that this can be remedied by removing small balls around the fixed-point singularities.
\end{remark}
\begin{proof}
	Part (a): These estimates follow from the expansion given in Lemma \ref{lem:setup:Greens-function}. When $B = \R^3$ or $B = \R^2 \times S^1$, the leading error term in $\O(r^{-2})$ or $\O(r^{-1})$ respectively, disappears due to the $\Z_2$ invariance.
		
	Part (b): These estimates follow from the expansion of the Green's function in spherical harmonics. For the fixed-point singularities $q_j$ the linear term will vanish due to the $\Z_2$ invariance of $h$.
	
	Part (c): We consider $\epsilon^{-1} + h$. 
	By the maximum principle any harmonic function attains its minimum on the boundary, where we have explicit estimates. Near the point $p_i$, the function $h$ diverges to $+ \infty$ with rate $\frac{1}{r}$. On the boundary near the fixed point $q_j$ we have the estimate
	$
	\epsilon^{-1} + h = \alpha_j + \frac{1}{2} \epsilon^{-1} + \O(\epsilon^2),
	$
	which is greater than $\frac{1}{2}$ for $\epsilon$ sufficiently small. Lastly, we consider the boundary at infinity. At this boundary, the function $h$ can only attain $0$ or $\pm \infty$, and using the condition on $ \#\{p_i\}$, the case $h|_{\infty} = - \infty$ is discarded.
	
	Part (d): We only show uniqueness. Suppose that $\tilde h$ satisfies
	\begin{enumerate}[noitemsep]
		\item $
		\Delta \tilde h = - 8 \pi \sum_j \delta(x - q_j) + 2 \pi \sum_i \delta(x - p_i) + \delta(x+ p_i)$, and
		\item $\tilde h$ is bounded below on $B'$.
	\end{enumerate}
	Then, $u := \tilde h - h$ is a harmonic function on $B$ which can be lifted to a harmonic function on $\R^3$. We claim that $u = \O(r)$. Indeed, due to part (a) the lower bound of $u$ diverges at most linearly to $-\infty$ and hence we only need to study the upper bound. For this, fix $r > 0$ sufficiently large and consider the map $u^+_r(x) := u(x) + 1 - \inf_{y \in B_{2r}(0)} u(y)$.
	This is strictly positive on $B_{2r}(0)$ and hence the Harnack inequality implies for all $x \in B_r(0)$,
	$u^+_r(x) \le 6 \: u^+_r(0).$
	For sufficiently large $r$, this can be rewritten as $u(x) \le C + 5\sup_{y \in B_{2r}(0) \cap B'} h(y)$
	for some constant $C > 0$. 	
	This proves the claim. The only harmonic functions that satisfy this are the affine functions, but the only affine function that makes $\tilde h = h + u$ bounded below is the constant function. Therefore, $u$ must be constant.
\end{proof}
\begin{remark}
	\label{rem:setup:harmonic-function-higher-derivatives}
	Although in Lemma \ref{lem:setup:harmonic-function}(a) we used the supremum norm, estimates for the derivatives can be obtained using elliptic regularity estimates. For example, when $B = \R^3$, the map $h(x) - \alpha - \frac{2 |p_i| -4}{2 r}$ is a harmonic function on the asymptotic region. According to the weighted Schauder estimate from \cite{Bartnik1986} Proposition 1.6, for each $k \in \N$ there exists a $C>0$ such that
	$$
	\left\| r^{k + 2} \:\nabla^k \left(h(x) - \alpha - \frac{2 |p_i| -4}{2 r}\right) \right\|_{C^0} 
	\le C \left\| r^{2}\left( h(x) - \alpha - \frac{2 |p_i| -4}{2 r} \right)\right\|_{C^0}
	< \infty.
	$$
	This implies that $\nabla^k \left(h(x) - \alpha - \frac{2 |p_i| -4}{2 r}\right)= \O(r^{-2-k})$ for all $k$. 
\end{remark}

\subsubsection{Circle bundle and involution}
In the next step, we need to find a circle bundle $P$ over $B'$ such that $c_1(P) = [* \d h] \in H^2(B', \Z)$. Using Mayer-Vietoris one can calculate the second homology of $B'$. Most elements of $H_2(B')$ are given by the $2$-spheres centred around the singularities $\pm p_i$ and $q_j$. When $B = \R \times T^2$, there is one extra cycle which is the $2$-torus at infinity. The space $H_2(B')$ has no torsion. Therefore the map $H^2(B, \Z) \to H^2_{dR}(B')$ is injective and its image contains all $[\sigma] \in H^2_{dR}(B')$ such that\footnote{The factor $-1/2\pi$ is due to the identification of $\mathfrak{u}(1) = i \R$ with $\R$.} $\frac{-1}{2 \pi}\int_\Sigma \sigma \in \Z$. So to uniquely determine $P$, it is sufficient to show
\begin{lemma}
	\label{lem:setup:existence-principal-bundle}
	For all $\Sigma \in H_2(B')$, we have that $\frac{-1}{2 \pi} \int_\Sigma * \d h \in \Z$.
\end{lemma}
\begin{proof}
 Using Lemma \ref{lem:setup:harmonic-function}(b) we have explicit estimates for $* \d h$ near the singularities. Using this and the fact that $\int_{S^2(\pm p_i)} * \d h$ must be radially independent, we conclude that
$$
\frac{-1}{2 \pi} \int_{S^2(\pm p_i)} * \d h = 1,  \qquad \frac{-1}{2 \pi} \int_{S^2(q_j)} * \d h = -4.
$$
Finally, we need to calculate $\int_{T^2} * \d h$ over the $2$-torus for the case $B = \R \times T^2$. We use a similar idea as in \cite{Charbonneau2008a} Proposition 3.5: Pick some $x > 0$ sufficiently large and consider the integral $\frac{-1}{2 \pi}\int_{[-x, x] \times T^2} \d * \d  h$. This integral must vanish due to the harmonicity of $h$. The boundary of $[-x, x]\times T^2 \subset B'$ decomposes into
$$
\{\pm x\} \times T^2 \bigsqcup \sqcup_{i} S^2(\pm p_i) \bigsqcup \sqcup_{j} S^2(q_j),
$$
and hence Stokes theorem implies
$$
0 
= \int\limits_{\{x\}\times T^2} * \d h 
+ \int\limits_{\{-x\}\times T^2} * \d h 
-\sum_i \int\limits_{S^2(\pm p_i)} * \d h 
-\sum_j \int\limits_{S^2_\delta(q_j)} * \d h.
$$
When we impose the $\Z_2$ invariance of $h$,
$$
2 \cdot \frac{-1}{2 \pi}\int_{\{x\}\times T^2} * \d h  = 4 |q_i| - 2 |p_i|  = 16 - 2|p_i| \in 2 \Z.
$$
\end{proof}

\begin{definition}
	\label{def:setup:principal-bundle}
	Let $h$ be the harmonic function defined in Lemma \ref{lem:setup:harmonic-function}.
	The principal circle bundle $P$ that satisfies $c_1(P) = [* \d h]$ will be referred to as the principal bundle.
\end{definition}

Following the construction of \mycite{Sen1997}, we lift the $\Z_2$-action $\tau$ that is induced by the antipodal map on $\R^3$ to a free $\Z_2$ action $\tilde \tau$ on $P$. In order for our gluing around the fixed-points $q_j$ to work, we need that $\tilde{\tau}$ coincides with the branched covering map defined for the Atiyah-Hitchin manifold, i.e. $\tilde \tau(e^{i \phi} \cdot p) = e^{- i \phi} \cdot \tilde \tau(p)$ for all $p \in P$ and $\phi \in \R$. Using the explicit bijection between principal $S^1$-bundles and $H^2(B',\Z)$ in \cite{Chern1977}, one can show $\tilde \tau$ exists if and only if 
$$\tau^* c_1(P) = - c_1(P).$$
Because the harmonic function $h$ is invariant under $\tau$, this is always satisfied.
\begin{definition}
	\label{def:setup:involution}
	Unless specified otherwise, we refer to $\tilde \tau$ as the $\Z_2$ action on $P$.	
\end{definition}

\subsubsection{The connection}
In order to apply the Gibbons-Hawking ansatz, we need to have a connection $\eta$ on the circle bundle $P$ such that $* \d h = \d \eta$. Such a connection can always be found. By the Mayer-Vietoris sequence, $H^1(B') = H^1(B)$, and so $\eta$ is determined by $H^1(B,\R) / H^1(B, \Z)$ up to gauge transformation. For the gluing to work, we do need to work in a certain gauge, which we explain now.

According to Lemma \ref{lem:setup:harmonic-function}, there is some constant $c \in \R$ such that
$$
\d \eta = * \d h = \begin{cases}
	-\frac{c}{2} \cdot \Vol_{S^2} + \O(r^{-4}) &\text{if } B = \R^3 \\
	c \cdot \Vol_{T^2} + \O(r^{-3}) &\text{if } B = \R^2 \times S^1 \\
	c \cdot \Vol_{T^2} + \O(e^{-r}) &\text{if } B = \R \times T^2.
\end{cases}
$$
The closed $2$-forms $-\frac{c}{2} \cdot \Vol_{S^2}$ and $c \cdot \Vol_{T^2}$ are representatives of elements in $H^2(S^2, \Z)$ and $H^2(T^2, \Z)$  respectively and hence there is a connection $\eta_{\infty}$ on a circle bundle over $S^2$ or $T^2$ such that 
$$
	\d \eta = \d \eta_\infty + \begin{cases}
		\O(r^{-4}) &\text{if } B = \R^3 \\
		\O(r^{-3}) &\text{if } B = \R^2 \times S^1 \\
		\O(e^{-r}) &\text{if } B = \R \times T^2.
	\end{cases}
$$
Because $\d \eta$ and $\d \eta_\infty$ represent the same element in $H^2$, there is a $1$-form $\tilde \eta_\infty$ on the asymptotic region of $B$ such that $\d \eta = \d \eta_\infty + \d \tilde \eta_\infty$. By the following version of the Poincar\'e lemma, we can pick $\tilde \eta_\infty$ with an explicit decay rate:

\begin{lemma}
	\label{lem:almost-hk:def-gauge}
	Let $\Sigma$ be a compact $n$-dimensional manifold and consider $U = \R \times \Sigma$ with coordinates $r \in \R$ and $t_i \in \Sigma$. 
	Fix $r_0 \in \R$ and let $\tau = \d r \wedge \mu + \nu$ be a closed $k$-form on $U$.
	Then the radial integrand
	$$
	\tilde \eta (r, t_i) = \int_{s \in (r_0, r)} \mu (s, t_i) \: \d s
	$$
	satisfies $
	\d \tilde \eta(r,t_i) = \tau(r, t_i) - \nu (r_0, t_i).$
\end{lemma}

This lemma can be proved by calculating $\d \eta$ in local coordinates and by applying the fundamental theorem of calculus.

We apply this lemma on $\tau = \d \eta - \d \eta_\infty$. By our error estimation we can pick $r_0 = \infty$ and the integral in Lemma \ref{lem:almost-hk:def-gauge} is still finite. Moreover, for this choice the boundary term $\nu (r_0, t_i)$ vanishes.
Using Remark \ref{rem:setup:harmonic-function-higher-derivatives} we estimate the higher derivatives and we conclude 

\begin{lemma}
	\label{lem:almost-hk:gauge-near-infty}
	Let $r$ be the Euclidean distance from the origin on $\R^3$, $\R^2$ or $\R$ when $B = \R^3$, $B = \R^2 \times S^1$ or $B = \R \times T^2$ respectively. Far away from the singularities, there exists an $r$-independent connection $\eta_\infty$ on a $S^1$-bundle over a compact set and a $1$-form $\tilde \eta_\infty$ on the asymptotic region of the base space such that
	$$
	\eta = \eta_{\infty} + \tilde \eta_\infty
	$$
	up to gauge transformation.
	With respect to $g_{B}$,
	$$
	\nabla^k \tilde \eta_\infty = \left\{ \text{\begin{tabular}{lll}
			$\O(r^{-3-k})$ &if $B = \R^3$\\
			$\O(r^{-2-k})$ &if  $B = \R^2 \times S^1$\\
			$\O(e^{-r})$ &if $B = \R \times T^2$
	\end{tabular}} 
	\right.
	$$
	for all $k \ge 0$.
\end{lemma}
\begin{definition}
	\label{def:connection}
	Pick a connection $\eta$ on the principal bundle $P$ that satisfies $* \d h = \d \eta$. Moreover assume that $\eta$ is antisymmetric under the involution $\tilde \tau$. (This enables us to project the Gibbons-Hawking metric to $P/\Z_2$.) Finally, use Lemma \ref{lem:almost-hk:gauge-near-infty} to fix a gauge on $\eta$. From now on this will be the canonical connection on $P$.
\end{definition}

\subsubsection{The collapsing parameter and metric}
\label{sec:setup:scaling-parameter}
Finally, we equip the bulk space $P/\Z_2$ with a \hk metric. For the gluing construction, we also introduce a collapsing parameter $\epsilon \in (0,1)$ in this step.
From the Gibbons-Hawking construction there are two obvious parameters to choose: The constant in the harmonic function $h$ or the global scale of the metric. Although these parameters look independent, they are actually related by a rescaling of the lattice and a translation of the singularities.

Because of this, we pick our collapsing parameter as a combination of them. We choose our metric such that for any point on $B'$, the length of the fibre converges to $2 \pi \epsilon$ as our collapsing parameter $\epsilon$ tends to zero:


\begin{definition}
	\label{def:setup:gibbons-hawking-with-epsilon}
	Consider $P$, $\tilde \tau$ from Definitions \ref{def:setup:principal-bundle} and \ref{def:setup:involution} and $h$ from Lemma \ref{lem:setup:harmonic-function}. Let $\eta$ be the antisymmetric connection on $P$ given in Definition \ref{def:connection}. For any $\epsilon > 0$ we define the harmonic function 
	$$
	h_\epsilon = 1 + \epsilon h.
	$$
	From now on, the metric that is induced by the Gibbons-Hawking ansatz with the harmonic function $\epsilon^{-1} + h$ and $\eta$, and is rescaled by a factor of $\epsilon$, will be called the Gibbons-Hawking metric and will be denoted as $g^{GH}$. Explicitly, it is given by
	$$
	g^{GH} = h_\epsilon g_B + \frac{\epsilon^2}{h_\epsilon} \eta^2
	$$
	and its K\"ahler forms are
	\begin{align*}
		\omega^{GH}_1 =& \epsilon \d x_1 \wedge \eta + h_\epsilon \d x_2 \wedge \d x_3 \\
		\omega^{GH}_2 =& \epsilon \d x_2 \wedge \eta + h_\epsilon \d x_3 \wedge \d x_1 \\
		\omega^{GH}_3 =& \epsilon \d x_3 \wedge \eta + h_\epsilon \d x_1 \wedge \d x_2.
	\end{align*}
	The \hk space $(P/\Z_2, g^{GH}, \omega^{GH}_i)$ will be called the bulk space.
\end{definition}

\subsection{The interpolation of the K\"ahler forms}
\label{sec:almost-hk}
Following the method by \mycite{Sen1997}, we make the bulk space complete and equip it with an almost \hk metric. To do this, we have to identify the asymptotic regions of the Atiyah-Hitchin manifolds with the neighbourhoods of the fixed 
points $q_j$ and the Taub-NUT spaces to a neighbourhood of the non-fixed points 
$p_i$. Topologically, these neighbourhoods already coincide.

We only need to define a global metric.
Instead of interpolating the metrics we will interpolate the K\"ahler forms while keeping them closed. In order to get the correct error estimates, we have to modify the diffeomorphism between the bulk space $P/\Z_2$ and the asymptotic regions of the Atiyah-Hitchin manifolds and Taub-NUT spaces using a suitable gauge transformation. We explain our choice of gauge transformation and we give the interpolated forms explicitly.

\subsubsection{Conformally rescaled metrics}
Before we compare $g^{GH}$ to the model metrics of the Taub-NUT space and the Atiyah-Hitchin manifold, we need to introduce the metric by which we will measure the error. This different metric will play an important role as the standard metric will not work with the inverse function theorem. This is studied in depth in \cite{Salm001}.

The global setup will be given later in Section \ref{sec:weighted-analysis-of-functions}. 
For now it is important that we do not use the standard H\"older norms with respect to some model metric $g_{model}$, but we will consider the standard H\"older norms with respect to a conformally rescaled norm $g_{cf} := \Omega^2 \:g_{model}$. 
The correct choice for the model metric and $\Omega$ is also found in \cite{Salm001}. In our case these are:

\begin{definition}
	\label{def:almost-hk:model-metric-pi}
	Let $p_i$ be a non-fixed singularity and let $r_i$ be the distance to $p_i$ on $B$. Let $\alpha_i \in \R $ be such that, near $p_i$, $h(x) = \alpha_i + \frac{1}{2|x - p_i|} + \O(|x - p_i|)$. For the model metric near $p_i$ define
	\begin{align*}		
		h^{p_i} :=&  \alpha_i + \frac{1}{2r_i},
		&\rho_{p_i} :=& \log r_i,  \\
		h_\epsilon^{p_i} :=& 1 + \epsilon \: h^{p_i},
		& \Omega_{p_i} :=& r_i^{-1} \left(h_\epsilon^{p_i}\right)^{-\frac{1}{2}}.
	\end{align*}
	Let $U^{p_i} \subseteq B'$ be a punctured neighbourhood of $p_i$ homotopic to $S^2$ and let $\eta^{p_i}$ be an $r_i$-invariant connection of $P|_{U^{p_i}}$ satisfying the Bogomolny equation
	$$
	* \d h^{p_i} = \d \eta^{p_i}.
	$$
	Define $g^{p_i}$ to be the Gibbons-Hawking metric induced by $h^{p_i}$ and $\eta^{p_i}$, i.e.
	$$
	g^{p_i} := h^{p_i}_\epsilon g_{U^{p_i}} + \frac{\epsilon^2}{h^{p_i}_\epsilon} (\eta^{p_i})^2.
	$$
	We call the \hk manifold $(P|_{U^{p_i}}, g^{p_i})$  the model space near $p_i$.
	Its K\"ahler forms are
	\begin{align*}
		\omega^{p_i}_1 =& \epsilon \d x_1 \wedge \eta^{p_i} + h_\epsilon^{p_i} \d x_2 \wedge \d x_3 \\
		\omega^{p_i}_2 =& \epsilon \d x_2 \wedge \eta^{p_i} + h_\epsilon^{p_i} \d x_3 \wedge \d x_1 \\
		\omega^{p_i}_3 =& \epsilon \d x_3 \wedge \eta^{p_i} + h_\epsilon^{p_i} \d x_1 \wedge \d x_2.
	\end{align*}
	Also define the conformally rescaled model metric $$g^{p_i}_{cf} := \Omega^2_{p_i} g^{p_i}
	= \d \rho_{p_i}^2 + g_{S^2} + \frac{\epsilon^2}{r_i^2 (h_\epsilon^{p_i})^2} (\eta^{p_i})^2
	.$$
\end{definition}	
\begin{definition}
	\label{def:almost-hk:model-metric-qj}
	Let $q_j$ be a fixed point singularity and let $r_j$ be the distance to $q_j$ on $B$. Let $\alpha_j \in \R$ be such that, near $q_j$, $h(x) = \alpha_j - \frac{2}{|x - q_j|} + \O(|x - q_j|^2)$. For the model metric near $q_j$ define
	\begin{align*}		
		h^{q_j} :=&  \alpha_j - \frac{2}{r_j},
		&\rho_{q_j} :=& \log r_j,  \\
		h_\epsilon^{q_j} :=& 1 + \epsilon \: h^{q_j},
		& \Omega_{q_j} :=& r_j^{-1} \left(h_\epsilon^{q_j}\right)^{-\frac{1}{2}}.
	\end{align*}
	Let $U^{q_j} \subseteq B'$ be a punctured neighbourhood of $q_j$ homotopic to $S^2$ and let $\eta^{q_j}$ be an $r_j$-invariant connection of $P|_{U^{q_j}}$ satisfying the Bogomolny equation
	$$
	* \d h^{q_j} = \d \eta^{q_j}.
	$$
	Define $g^{q_j}$ to be the Gibbons-Hawking metric induced by $h^{q_j}$ and $\eta^{q_j}$, i.e.
	$$
	g^{q_j} := h^{q_j}_\epsilon g_{U^{q_j}} + \frac{\epsilon^2}{ h^{q_j}_\epsilon} (\eta^{q_j})^2.
	$$
	We call the \hk manifold $(P|_{U^{q_j}}, g^{q_j})$ the model space near $q_j$.
	Its K\"ahler forms are
	\begin{align*}
		\omega^{q_j}_1 =& \epsilon \d x_1 \wedge \eta^{q_j} + h_\epsilon^{q_j} \d x_2 \wedge \d x_3 \\
		\omega^{q_j}_2 =& \epsilon \d x_2 \wedge \eta^{q_j} + h_\epsilon^{q_j} \d x_3 \wedge \d x_1 \\
		\omega^{q_j}_3 =& \epsilon \d x_3 \wedge \eta^{q_j} + h_\epsilon^{q_j} \d x_1 \wedge \d x_2.
	\end{align*}
	Also define the conformally rescaled model metric $$g^{q_j}_{cf} := \Omega^2_{q_j} g^{q_j}
	= \d \rho_{q_j}^2 + g_{S^2} + \frac{\epsilon^2}{r_j^2 (h_\epsilon^{q_j})^2} (\eta^{q_j})^2
	.$$
\end{definition}

Notice that the conformal rescaling changes the metric of the base space from the flat metric of $\R^3$ to the cylindrical metric of $\R^+ \times S^2$.
Therefore, when using $g^{p_i}_{cf}$ and $g^{q_j}_{cf}$, higher derivatives will have the same growth/decay rate as the function itself. 
For example, according to Remark \ref{rem:setup:harmonic-function-higher-derivatives}, $\nabla^k (h - h^{q_j}) = \O(r_j^{2-k})$ with respect to the Euclidean metric on the base space. Converting this to the conformal metric, one concludes
	\begin{equation}
		\label{eq:revision-difference-h-and-model-h}
		\nabla^k_{cf}(h - h^{q_j}) \sim  r^k\nabla^k_{g_B} (h - h^{q_j}) = \O(r_j^2).
	\end{equation}
	Therefore, higher order estimates follow automatically from the $C^0$ estimate and we can omit them in our calculations.

\subsubsection{Error estimates}
We want to glue in rescaled copies of the Taub-NUT (TN) metric and the Atiyah-Hitchin (AH) manifold to the non-complete bulk space $P/\Z_2$. 
Even more, we want to find some $0< R_0 < R_1 \ll 1$ and some scaling factors, so that we can define a definite triple of the form
\begin{align*}
	\omega =& \left\{\begin{matrix}
		\text{scaling factor} \cdot \omega^{TN}  & \textit{if}&& \|x - p_i\|_{g^B} &\le R_0\\
		\text{scaling factor} \cdot \omega^{AH} &\textit{if}&& \|x - q_j\|_{g^B} &\le R_0\\
		\omega^{p_i}+ \d \left(\text{interpolation term}
		\right) &\textit{if}& R_0 \le& \|x - p_i\|_{g^B} &\le R_1 \\
		\omega^{q_j}+ \d \left(\text{interpolation term}
		\right) &\textit{if}& R_0 \le& \|x - q_j\|_{g^B} &\le R_1 \\
		\omega^{GH} &&  & \text{otherwise}
	\end{matrix}\right.
\end{align*}
on this global space.
To do this, we first need to ask the following questions:
\begin{enumerate}
	\item By what must we rescale the Taub-NUT and the Atiyah-Hitchin metric,
	\item How much do the K\"ahler forms differ between these rescaled spaces and the model spaces defined in Definition \ref{def:almost-hk:model-metric-pi} and Definition \ref{def:almost-hk:model-metric-qj}, and
	\item How much do the K\"ahler forms differ between the bulk space $P/\Z_2$ and the model spaces?
\end{enumerate}
We start with question 3. The difference between the K\"ahler forms of the bulk space $(P/\Z_2, g^{GH}, \omega^{GH}_i)$ and the model metrics are exactly
\begin{align*}
	\omega^{GH}_i - \omega^{p_i} =& \epsilon \d x_i \wedge (\eta - \eta^{p_i}) + (h_\epsilon - h^{p_i}_\epsilon) \d x_j \wedge \d x_k, \\
	\omega^{GH}_i - \omega^{q_j} =& \epsilon \d x_i \wedge (\eta - \eta^{q_j}) + (h_\epsilon - h^{q_j}_\epsilon) \d x_j \wedge \d x_k.
\end{align*}
According to Equation \eqref{eq:revision-difference-h-and-model-h}, the difference between $h$ and $h^{q_j}$ is $\O(r_j^2)$ w.r.t. $g^{q_j}_{cf}$. By the definition of $h_\epsilon$, the difference $h_\epsilon-h^{q_j}_\epsilon$ is $\O(\epsilon r_j^2)$.
Therefore, we only need to estimate the difference between the connection $\eta$ and $\eta^{q_j}$. For this we revisit Lemma \ref{lem:almost-hk:gauge-near-infty}. To estimate the closed $2$-form $\d (\eta - \eta^{q_j}) = *^B \d (h - h^{q_j})$, we use Equation \eqref{eq:revision-difference-h-and-model-h} to get $h - h^{q_j} = \O(r_j^2)$ w.r.t. $g^{q_j}_{cf}$. The exterior derivative is a uniformly bounded operator, and so $\d (h - h^{q_j}) = \O(r_j^2)$. The Hodge dual $*^B$ introduces an extra factor $r_j$ when comparing it to $g^{q_j}_{cf}$, and so $\d (\eta - \eta^{q_j}) = \O(r_j^3)$ w.r.t. $g^{q_j}_{cf}$.
By Remark \ref{rem:setup:harmonic-function-higher-derivatives} the same can be said about its derivatives.

Finally, we integrate $\d (\eta - \eta^{q_j})$ from $r_0 = 0$ using Lemma \ref{lem:almost-hk:def-gauge}. 
This integration does not introduce new factors $r_j$. Indeed, the integrand in this Lemma is $\iota_{\del_r} \d (\eta - \eta^{q_j}) = \O(r_j^2)$, because $\del_r$ is $\O(r_j^{-1})$ w.r.t. $g_{cf}^{q_j}$. At the same time, integration along $r_j$ reintroduces a factor $r_j$. In the end these factors cancel out.

Summarizing, we have a $1$-form $\tilde \eta^{q_j}$ such that $\d \eta = \d \eta^{q_j} + \d \tilde \eta^{q_j}$.
Because $H^1(S^2) = 0$, the form $\eta - \eta^{q_j} - \tilde \eta^{q_j}$ is exact and hence we have:
\begin{lemma}
	\label{lem:almost-hk:gauge-near-q}
	On a small annulus around each fixed point singularity $q_j$, there exists a gauge transformation which identifies $\eta$ with $\eta^{q_j} + \tilde \eta^{q_j}$,
	where $\tilde \eta^{q_j}$ and all its derivatives are of order $r_j^3$ with respect to $ g^{q_j}_{cf}$.

	Similarly, on a small annulus around each non-fixed singularity $p_i$, there exists a gauge transformation which identifies $\eta$ with $\eta^{p_i} + \tilde \eta^{p_i}$,
	where $\tilde \eta^{p_i}$ and all its derivatives are of order $r_i^2$ with respect to $ g^{p_i}_{cf}$.
\end{lemma}
Using the estimates from Lemmas \ref{lem:setup:harmonic-function} and Lemma \ref{lem:almost-hk:gauge-near-q} the difference between the K\"ahler forms of $g^{GH}$ and $g^{q_j}$ is given by $\epsilon \d x_i \wedge \O(r_j^3) + \O(\epsilon r_j^2) \d x_j \wedge \d x_k$.  Because $r^{-1} \d x_i$ and its derivatives are uniformly bounded in $g^{q_j}_{cf}$, 
$$
\| \nabla^k (\omega^{GH} - \omega^{q_j}) \|_{g^{q_j}_{cf}} = \O(\epsilon r_j^{4})
$$ for all $k \ge 0$. Similarly, one can estimate $
\| \nabla^k (\omega^{GH} - \omega^{p_i}) \|_{g^{p_i}_{cf}} = \O(\epsilon r_i^{3})
$ near a non-fixed singularity $p_i$. This enables us to apply the radial integration from Lemma \ref{lem:almost-hk:def-gauge} again to find:
\raggedbottom
\begin{lemma}
	\label{lem:almost-hk:diff-kahler-GH-M}
	On a small annulus around the non-fixed singularity $p_i$, there exists a smooth triple of $1$-forms, which we denote by $\sigma^{p_i, GH}$, such that 
	$$
	\omega^{GH} = \omega^{p_i} + \d \sigma^{p_i, GH}.
	$$
	The $1$-forms $\sigma^{p_i, GH}$ and all their derivatives are of order $\O(\epsilon r_j^{3})$ with respect to $g^{p_i}_{cf}$.
	
	On a small annulus around the fixed point singularity $q_j$, there exists a smooth triple of $1$-forms, which we denote by $\sigma^{q_j, GH}$, such that 
	$$
	\omega^{GH} = \omega^{q_j} + \d \sigma^{q_j, GH}.
	$$
	The $1$-forms $\sigma^{q_j, GH}$ and all its derivatives are of order $\O(\epsilon r_j^{4})$ with respect to $g^{q_j}_{cf}$.
\end{lemma}

Having answered question 3, we now answer question 1 and 2 for a fixed point singularity $q_j$.
According to \mycite{Atiyah1988}, the Atiyah-Hitchin metric has a radial parameter $r_{AH}$ and for large values of $r_{AH}$ the metric on the branched double cover is
\begin{equation}
	\label{eq:metric-AH}
	g^{AH} = \left(1 - \frac{2}{r_{AH}}\right) (\d r^2_{AH} + r^2_{AH} g_{S^2}) + \left(1 - \frac{2}{r_{AH}}\right)^{-1} (\eta^{q_j})^2 + \O(e^{- r_{AH}}).
\end{equation}
By identifying $r_j := \frac{\epsilon}{1 + \epsilon \alpha_j} r_{AH}$, where $\alpha_j$ is defined in Lemma \ref{lem:setup:harmonic-function} and applying the radial integration from Lemma \ref{lem:almost-hk:def-gauge}, one can show
\begin{lemma}
	\label{lem:almost-hk:diff-kahler-AH-M}
	On the asymptotic region of the Atiyah-Hitchin manifold, there exists a triple of $1$-forms $\sigma^{q_j, AH}$ such that 
	$$
	\frac{\epsilon^2}{1 + \epsilon \alpha_j}\omega^{AH} = \omega^{q_j} + \d \sigma^{q_j, AH}
	$$
	and $\sigma^{q_j, AH}$ and all its derivatives are $\O(r_j^2 e^{- \frac{1 + \epsilon \alpha}{\epsilon}r_j} )$ with respect to $g^{q_j}_{cf}$.
\end{lemma}
\begin{remark}
	For simplicity we assume that $\sigma^{q_j, AH}$ all its derivatives are $\O(\epsilon^3 r_j^{-1})$ with respect to $g^{q_j}_{cf}$.
\end{remark}
This lemma answers question 1 and 2 for the Atiyah-Hitchin metric: Namely, we need to rescale the Atiyah-Hitchin manifolds with a factor of $\frac{\epsilon^2}{1 + \epsilon \alpha_j}$. The error between the rescaled Atiyah-Hitchin manifold and the model metrics will exponentially decay in $r_j$.

In a similar manner we can compare the K\"ahler forms for the model metric near $p_i$ with a rescaled version of a fixed Taub-NUT space. In this case, there is no exponentially decaying error term and we get the result:
\begin{lemma}
	\label{lem:almost-hk:diff-kahler-TN-M}
	By identifying $r_i := \frac{\epsilon}{1 + \epsilon \alpha_i} r_{TN}$, where $\alpha_i$ is defined in Lemma \ref{lem:setup:harmonic-function},
	\begin{align*}
		g^{p_i} = \frac{\epsilon^2}{1 + \epsilon \alpha_i} g^{TN}, \text{ and }
		\omega^{p_i} = \frac{\epsilon^2}{1 + \epsilon \alpha_i} \omega^{TN},
	\end{align*}
	where $g^{TN}$ is the fixed Taub-NUT space
	$$
	g^{TN} := \left(1 + \frac{1}{2 r_{TN}}\right) (\d r_{TN}^2 + r_{TN}^2 g_{S^2}) + \frac{1}{1 + \frac{1}{2 r_{TN}}} (\eta^{p_i})^2.
	$$
\end{lemma}

\subsubsection{The complete manifold}
With these ingredients we can finally construct a complete manifold and equip it with a definite triple that is almost \hk. Near the fixed point singularities we complete the bulk space using Atiyah-Hitchin manifolds. Near the non-fixed points, we glue in Taub-NUT spaces\footnote{Alternatively, one can complete the region near the non-fixed points in a similar way as the Taub-NUT space is completed by adding some extra points. To unify the gluing procedure for the Atiyah-Hitchin manifold and the Taub-NUT space we prefer to use the first method.}. With this setup, we define the complete 4-dimensional manifold underlying our gravitational instantons.
\begin{definition}
	Let $n$ be the number of non-fixed points $p_i$ and $m$ be the number of fixed point singularities $q_j$, defined in Definition \ref{def:setup:non-fixed-points} resp. \ref{def:setup:basespace}. Let $P/\Z_2$ be the bulk space. Identify the asymptotic region of the Atiyah-Hitchin manifold and the neighbourhoods of $q_j$ on $P/\Z_2$ with the $\Z_2$ quotient of the model space defined in Definition \ref{def:almost-hk:model-metric-qj}. Similarly,
	identify the asymptotic region of the Taub-NUT space and the neighbourhoods of $p_i$ on $P/\Z_2$ with the the model space defined in Definition \ref{def:almost-hk:model-metric-pi}.
	Consider the connected sum of $P/\Z_2$ with $m$ copies of the Atiyah-Hitchin manifold and $n$ copies of the Taub-NUT space. We call this space the global space and we denote it as $M_{B,n}$.
\end{definition}

In order to equip $M_{B,n}$ with a definite triple, let $\epsilon \in (0,1)$, $R_0, R_1 \in (0, \infty)$ be small. Assume that the gluing in the connected sum construction happens on the region $ \bigcup_{i} B_{R_1}(p_i) \backslash B_{R_0}(p_i)$ and $ \bigcup_{j} B_{R_1}(q_j) \backslash B_{R_0}(q_j)$.
For each point $p_i$ and $q_j$, pick a family of smooth step functions $\chi_\epsilon(x)$ on $B$ such that $\chi_\epsilon(x) = 0$ when $\|x - p_i\|_{g^B}, \|x - q_j\|_{g^B} \le R_0$ and $\chi_\epsilon(x) = 1$ when $\|x - p_i\|_{g^B}, \|x - q_j\|_{g^B} \ge R_1$. We pick the following triple on the connected sum:
\begin{align*}
	\omega =& \left\{\begin{matrix}
		\frac{\epsilon^2}{1 + \epsilon \alpha_i}\omega^{TN}  & \textit{if}&& \|x - p_i\|_{g^B} &\le R_0\\
		\frac{\epsilon^2}{1 + \epsilon \alpha_j}\omega^{AH} &\textit{if}&& \|x - q_j\|_{g^B} &\le R_0\\
		\omega^{p_i}+ \d \left(
		\chi_\epsilon \sigma^{p_i, GH}
		\right) &\textit{if}& R_0 \le& \|x - p_i\|_{g^B} &\le R_1 \\
		\omega^{q_j}+ \d \left[
		(1- \chi_\epsilon) \sigma^{q_j, AH}
		+
		\chi_\epsilon \sigma^{q_j, GH}
		\right] &\textit{if}& R_0 \le& \|x - q_j\|_{g^B} &\le R_1 \\
		\omega_i^{GH} &&  & \text{otherwise.}
	\end{matrix}\right.
\end{align*}
We need to find $\chi_\epsilon$, $R_0$ and $R_1$ such that $\omega_i$ is \hk outside $r \in [R_0, R_1]$ and behaves well enough inside. Assume that $R_0 = C_0 \epsilon^{\kappa}$ and $	R_1 = C_1 \epsilon^{\kappa} $
for some $C_0, C_1 > 0$, $\kappa \in \R$. We need to balance the following factors:
\begin{itemize}
	\item For the approximations of $\sigma^{p_i, GH}$ and $\sigma^{q_j, GH}$ we need the radial distance to the singularity to be small. This is satisfied when $\kappa > 0$.
	\item At the same time we need that $h_\epsilon > 0$ and so $r_j$ cannot be too small. This is satisfied when $C_0 = 4$ and $\kappa < 1$, because Lemma \ref{lem:setup:harmonic-function} implies $h_\epsilon > 0$ if $4 \epsilon^{\kappa} > 4 \epsilon$.
	\item  For the approximation of $\sigma^{q_j, AH}$ we need $r_{AH}$ to be large. Combining $r_j = \O(\epsilon^{\kappa})$ and $r_{AH} = \frac{1 + \epsilon \alpha_j}{\epsilon} r_j$, it follows $r_{AH} = \O(\epsilon^{\kappa - 1})$. This is large when $\kappa < 1$.
	\item Finally we need that $R_0 < R_1$. This happens when $C_0 < C_1$.
\end{itemize}
From Lemma \ref{lem:almost-hk:diff-kahler-GH-M} and \ref{lem:almost-hk:diff-kahler-AH-M}, we have decay estimates for $\sigma^{p_i, GH}$, $\sigma^{q_j, GH}$ and $\sigma^{q_j, AH}$. It is sufficient if we assume $\sigma^{p_i, GH},\sigma^{q_j, GH} = \O(\epsilon r_j^3)$ and $\sigma^{q_j, AH} = \O(\epsilon^3 r_j^{-1})$. 
When we pick
$$
R_0 = 4 \epsilon^{\frac{2}{5}} \text{ and } R_1 = 5 \epsilon^{\frac{2}{5}}
$$
all the above requirements are satisfied.
By estimating $\chi_\epsilon$, one notices that $\d \chi_\epsilon = \O(1)$ and may conclude:

\begin{theorem}
	\label{thm:almost-hk:global-symplectic-triple}
	There exists an $\epsilon_1 > 0$ such that for all $0 < \epsilon < \epsilon_1$:
	\begin{enumerate}
		\item $\omega_i$ is a closed $2$-form in $M_{B,n}$.
		\item Outside the gluing region (i.e. $r_i \in [4\epsilon^{\frac{2}{5}}, 5 \epsilon^{\frac{2}{5}}]$ or $r_j \in [4\epsilon^{\frac{2}{5}}, 5 \epsilon^{\frac{2}{5}}]$),
		$\omega_i$ is a \hk triple.
		\item Inside the gluing region near the fixed point $q_j$, $\omega_i - \omega_i^{q_j}$ and all its derivatives are of order $\O(\epsilon^3 r_j^{-1}) + \O(\epsilon r_j^3)$ w.r.t. $g_{cf}^{q_j}$. In particular, inside this gluing region $\omega_i$ is a definite triple of closed $2$-forms such that
		$$
		\frac{1}{2} \omega_i \wedge \omega_j = \left( \operatorname{Id} + \O(\epsilon^{7/5})\right)_{ij} \otimes \Vol^{g^{q_j}}.
		$$
		\item Inside the gluing region near the non-fixed point $p_i$, $\omega_i - \omega_i^{p_i}$ and all its derivatives are of order $\O(\epsilon r_j^3)$ w.r.t. $g_{cf}^{p_i}$. In particular, inside this gluing region $\omega_i$ is a definite triple of closed $2$-forms such that
		$$
		\frac{1}{2} \omega_i \wedge \omega_j = \left( \operatorname{Id} + \O(\epsilon^{7/5})\right)_{ij} \otimes \Vol^{g^{p_i}}.
		$$
	\end{enumerate} 
\end{theorem}

\section{The deformation problem}
\label{sec:deformation-problem}
To perturb the approximate solution, we phrase the \hk conditions as an elliptic PDE which we solve using the inverse function theorem. To do this, we introduce an alternative definition of \hk manifolds in terms of their K\"ahler forms. We use this alternative definition to set up the deformation problem. The perturbation argument explained here is a slightly modified version of that used in \cite{Schroers2020}.

\subsection{Deformation of definite triples}
In general, for a given \hk 4-manifold $(M, g, I_1, I_2, I_3)$, the K\"ahler forms $\omega_i$ satisfy
$$
\frac{1}{2} \omega_i \wedge \omega_j = \delta_{ij} \Vol^g \text{ for all } i,j \in \{1,2,3\}.
$$
Therefore the K\"ahler forms are an orthonormal basis of $\Lambda^+(M)$ with respect to $g$. 
According to \mycite{Donaldson2006}, the converse is also true, i.e. for each triple of closed $2$-forms $\omega_i$ and volume form $\mu$ that satisfy
\begin{equation}
	\label{eq:hk-triple-condition}
	\frac{1}{2} \omega_i \wedge \omega_j = \delta_{ij}\: \mu \text{ for all } i,j \in \{1,2,3\},
\end{equation}
there exists a unique \hk metric $g$ with volume form $\mu$ and K\"ahler forms $\omega_i$.

In Theorem \ref{thm:almost-hk:global-symplectic-triple} we found a triple of closed $2$-forms $\omega_i$ that are approximately orthonormal. Assume  there exists a triple of $1$-forms $a_i$ such that
$$
\tilde \omega_i = \omega_i + \d a_i
$$
solve Equation \eqref{eq:hk-triple-condition}. Then the expression
$\tilde{\omega}_i \wedge \tilde{\omega}_j - \frac{1}{3} \delta_{ij} \sum_k \tilde \omega_k \wedge \tilde\omega_k$ is a traceless, symmetric $3 \times 3$ matrix with values in $\Omega^4(M_{B,n})$. Therefore, we consider the projection map
\begin{equation} \label{eq:hk-triples:Tf-operator}
	\begin{split}
		\Tf \colon \Mat_{3\times 3}(\R) \otimes \Omega^4(M_{B,n}) \to \Sym^2_0(\R^3) \otimes \Omega^4(M_{B,n}) \\
		P \otimes \mu \mapsto \left(\frac{1}{2} P + \frac{1}{2} P^T - \frac{1}{3} \operatorname{Tr}(P) \operatorname{Id}\right) \otimes \mu,
	\end{split}
\end{equation}
and our goal is to find $a \in \Omega^1(M_{B,n}) \otimes \R^3$ such that 
\begin{equation}
	\label{eq:hk-triples:diff-eq-v-neg1}
	\Tf((\omega + \d a)^2)
	= 
	\Tf(\omega \wedge \omega) + 2 \Tf(\d a \wedge \omega) + \Tf(\d a \wedge \d a) = 0.
\end{equation}

This does not have a unique solution. In order to solve this issue, we first remove the gauge freedom $a \mapsto a + \d f$: According to \mycite{Donaldson2006}, there is a unique metric $g$ such that $\omega_i$ span $\Omega^+(M_{B,n})$ and $\Vol^g = \frac{1}{3}\sum_k \omega_k \wedge \omega_k$. 

\begin{definition}
	\label{def:definition-g}
	We call $g$ the approximate \hk metric on $M_{B,n}$.
\end{definition}

Secondly, we fix the gauge by assuming $\d^* a = 0$. Finally, we also assume that $a$ satisfies
\begin{equation}
	\label{eq:hk-triples:diff-eq-v0}
	\d a \wedge \omega = \d^+ a \wedge \omega = - \frac{1}{2} \Tf \left(\omega\wedge \omega + \d a \wedge \d a\right).
\end{equation}

Recall that $\omega_i$ span $\Omega^+(M_{B,n})$ and the wedge product is a non-degenerate pairing on $\Omega^+$. Therefore, the map
\begin{equation} \label{eq:hk-triples:Lambda-operator}
	\begin{split}
		\Lambda\colon \Omega^+(M_{B,n}) \otimes \R^3 \to& \Mat_{3\times 3}(\R) \otimes \Omega^4(M_{B,n}) \\
		\sigma \mapsto& \sigma \wedge \omega 
	\end{split}
\end{equation}
is a bijection and Equation \eqref{eq:hk-triples:diff-eq-v0} is equivalent to
\begin{equation}
	\label{eq:hk-triple:diff-eq}
	\d^+ a  = - \frac{1}{2} \Lambda^{-1} \Tf(\omega \wedge \omega + \d a \wedge \d a).
\end{equation}
Combining Equation \eqref{eq:hk-triples:diff-eq-v0} with the gauge fix $\d^* a = 0$, we conclude $a$ must satisfy
$$
(\d^* + \d^+) \:a = - \frac{1}{2} \Lambda^{-1} \Tf(\omega \wedge \omega + \d a \wedge \d a).
$$
Our choice of gauge is convenient, because the operator 
\begin{align*}
	\D \colon \Omega^0(M_{B,n}) \oplus \Omega^1(M_{B,n}) \oplus \Omega^{\mp}(M_{B,n}) &\to \Omega^0(M_{B,n}) \oplus \Omega^1(M_{B,n}) \oplus \Omega^{\pm}(M_{B,n}) \\
	f &\mapsto \d f \tag*{$f \in \Omega^0(M_{B,n})$} \\
	a &\mapsto (\d^* + \d^\pm) a \tag*{$a \in \Omega^1(M_{B,n})$} \\
	\sigma &\mapsto 2\d^* \sigma \tag*{$\sigma \in \Omega^\mp(M_{B,n})$}
\end{align*}
is a Dirac operator and $\D^2$ equals the Hodge Laplacian.

Next, we assume that $a$ lies in the image of $\D \colon (\Omega^0(M_{B,n}) \oplus \Omega^+(M_{B,n})) \otimes \R^3 \to \Omega^1(M_{B,n}) \otimes \R^3$.
This has the advantage that the linearized version of Equation \eqref{eq:hk-triple:diff-eq} is the Hodge Laplacian and that $a$ can be described by a section of a trivial bundle.
Moreover, if we write $a = \D(u + \zeta)$ with $u \in \Omega^0(M_{B,n}) \otimes\R^3$ and $\zeta \in \Omega^+(M_{B,n})\otimes \R^3$, then $u$ and $\zeta$ must satisfy
\begin{align}
	\label{eq:hk-triple:diff-eq-v2}
	\Delta \zeta =&  - \frac{1}{2} \Lambda^{-1} \Tf(\omega \wedge \omega) - 2 \Lambda^{-1} \Tf(\d \d^* \zeta \wedge \d \d^* \zeta), \\
	\Delta u =& 0. \notag
\end{align}
We fix the gauge $\d^*a = 0$ by setting $u = 0$. We will solve Equation \eqref{eq:hk-triple:diff-eq-v2} using the version of the inverse function theorem given in Lemma 6.15 in \cite{Foscolo2016}:
\begin{theorem}[The inverse function theorem]
	\label{thm:hk-triple:inverse-function-theorem}
	Let $F(x) = F(0) + L(x) + N(x)$ be a smooth function between Banach spaces such that there exist $r, q, C > 0$ satisfying
	\begin{enumerate}
		\item $L$ is an invertible linear operator with $\|L^{-1}\| < C$,
		\item $\| N(x) - N(y)\|\le q \cdot \|x + y \| \cdot \| x - y \|$ for all $x, y \in B_r(0)$, and
		\item $\|F(0) \| < \min \left\{
		\frac{1}{4 q C^2}, \frac{r}{2C}
		\right\}$.
	\end{enumerate}
	Then, there exists a unique $x$ in the domain of $F$ such that $F(x) = 0$ and $\|x\| \le 2 C \|F(0)\|$.
\end{theorem}
We need to find suitable Banach spaces such that the Hodge Laplacian on $\Omega^+(M_{B,n}) \otimes \R^3$ is invertible with bounded inverse. A small calculation will show it is sufficient to study the Laplacian acting on functions instead. Indeed, trivialize $\zeta \in \Omega^+(M_{B,n}) \otimes \R^3$ into $\zeta_i = \sum_j u_{ij} \: \omega_j$ and use Riemann normal coordinates $\{x_k\}$.
By the Weitzenb\"ock formula \cite[Equation 3.8]{Roe1998},
$$
\Delta (u_{ij} \omega_j) = \D^2 (u_{ij} \omega_j) = - \nabla^k \nabla_k (u_{ij} \omega_j) + \centernot{R} (u_{ij} \omega_j),
$$
where $\centernot{R}$ is the Clifford contraction of the Riemann curvature tensor. Using the trivialisation of $\zeta$ and the fact that the Clifford contraction is $C^{\infty}$-linear,
\begin{align}
	\label{eq:hk-triple:Weitzenbock-formula}
	\Delta (u_{ij} \omega_j)
	&=
	(\Delta_{g} u_{ij}) \: \omega_j
	- 2\nabla_{\nabla u_{ij}}  \:\omega_j
	+ u_{ij} \cdot \D^2(\omega_j) \notag\\
	&=
	(\Delta_{g} u_{ij}) \: \omega_j
	- 2\nabla_{\nabla u_{ij}}  \:\omega_j.
\end{align}
The term $\D^2(\omega_j)$ vanishes because $\omega_j$ is closed and self-dual. When $\omega$ is a \hk triple, $\nabla \omega = 0$ and hence $\Delta(u_{ij} \omega_j) = (\Delta u_{ij}) \omega_j$. We expect that, when $\omega$ is sufficient close to being \hk, the Hodge Laplacian on $\Omega^+(M_{B,n})$ and the Laplacian on functions define equivalent operators.

\subsection{Weighted analysis of functions} 
\label{sec:weighted-analysis-of-functions}
Next, we need to set up the correct Banach spaces on which the Laplacian is invertible with uniformly bounded inverse.
We will use weighted H\"older spaces for this and in this section we will determine the suitable weight functions.
The first step will be to set up the analysis on the asymptotic region of $M_{B,n}$. This is already done in \cite{Salm001} and in Appendix \ref{sec:appendix}. To summarize these results, one considers strictly positive, smooth functions $\Omega$ and $\rho$ on $M_{B,n}$ such that outside some large compact set
\begin{equation}
	\label{eq:def-omega-rho-near-infinity}
	\Omega := r^{-1} h_\epsilon^{-\frac{1}{2}}
	\quad\text{and}\quad
	\rho := \log r.
\end{equation}

Here $r$ is the Euclidean distance from the origin on $\R^3$, $\R^2$ or $\R$ when $B = \R^3$, $B = \R^2 \times S^1$ or $B = \R \times T^2$ respectively.
When $B \not = \R^3$, also consider a function $\phi \in C^\infty(B)$, such that $\phi$ vanishes on a large compact set and equals $\rho$ or $r$ near infinity when $B = \R^2 \times S^1$ or $B = \R \times T^2$ respectively.

Next, one defines the conformally rescaled norm $g_{cf} := \Omega^2 g$ and considers the standard H\"older norm $C^{k, \alpha}_{cf}$ that uses the norms, Levi-Civita connection and parallel transport induced by $g_{cf}$. 
According to Theorem 1.4 in \cite{Salm001} and Proposition \ref{prop:appendix:elliptic-regularity}, there are uniform elliptic regularity estimates for the Laplacian when one uses the weighted H\"older norm
$$
\|u\|_{C^{k, \alpha}_{\delta}(M_{B,n})} = \|e^{- \delta \rho} u\|_{C^{k, \alpha}_{cf}(M_{B,n})}.
$$
Even more, according to Theorem 1.6 in \cite{Salm001} and Theorem \ref{thm:appendix:isomorphism}, the operator
$$
\Omega^{-2} \Delta^g \colon \begin{Bmatrix}
	C^{k+2, \alpha}_\delta(M_{B,n}) & \text{if } B = \R^3 \\
	C^{k+2, \alpha}_\delta(M_{B,n}) \oplus \R \phi & \text{otherwise}
\end{Bmatrix}
\to C^{k, \alpha}_\delta(M_{B,n})
$$
is an isomorphism if $k \in \N$, $\alpha \in (0,1)$, $\delta < 0$ and $|\delta| \ll 1$.

If one is only interested in the Fredholm properties of the Laplacian, it is sufficient to understand the weighted H\"older norms on the asymptotic region of the manifold. In our case however, we also need a bound on the inverse of the Laplacian which, if not careful, can blow up as the collapsing parameter $\epsilon$ goes to zero. Therefore, we need to define $g_{cf}$, $\Omega$, $\rho$ and $\phi$ on the interior of $M_{B,n}$ explicitly. We define these as follows:

\begin{definition}
 	\label{def:global-analysis:global-metric-and-functions}
	Let $\pi \colon P \to B'$ be the circle bundle covering the bulk space.
 	Let $R_1 > 0$ be such that $P_\infty :=  \pi^{-1} (B' \setminus B_{R_1}(0))$ describes the asymptotic region of $M_{B,n}$. 
	Let $g^{TN}$, $g^{AH}$, $g^{p_i}$, $g^{q_j}$ and $g^{GH}$ be the metrics defined in Lemma \ref{lem:almost-hk:diff-kahler-TN-M}, Equation \eqref{eq:metric-AH}, Definition \ref{def:almost-hk:model-metric-pi}, Definition \ref{def:almost-hk:model-metric-qj} and Definition \ref{def:setup:gibbons-hawking-with-epsilon} respectively.
	Pick $R_2, R_3 > 0$ such that $R_2 \ll 1$ and $R_3 \gg 1$. 
	Consider $M_{B,n}$ as the disjoint union of the regions
	\begin{align*}
		 &\{r_{TN} < R_3\}, && \{r_{AH} < R_3\}, \\
		&\{r_{TN} > R_3 \text{ and } r_i < R_2\}, && \{r_{AH} > R_3 \text{ and } r_j < R_2\}, \\
		& \{r_{i},r_{j} > R_2 \text{ and } r < R_1\}, \text{ and }&&
		 \{r > R_1\},
	\end{align*} 	
	where $r_{TN}$, $r_{AH}$, $r_i$, $r_j$ and $r$ are the radial parameters induced by $g^{TN}$, $g^{AH}$, $g^{p_i}$, $g^{q_j}$ and $g^{GH}$ respectively.
 	On the interior of each region, define the metric $g_{cf}$ and the functions $\Omega, \rho,  \phi \in C^{\infty}(M_{B,n})$ as shown in Table \ref{table:definition-g-cf} and interpolate the metric and functions on the overlap.
	\begin{table}
		\centering 
		\begin{tikzpicture}
 			\draw[very thick] (0,0) -- (7,0);
 			\draw[very thick,dashed] (7,0) -- (8,0);
 			\draw[very thick,-{latex}] (8,0) -- (10,0) node[below]{$r \to \infty$};
 			\draw (0,3pt) -- (0,-3pt) node[below]{$0$};
 			\draw (3.5,-3pt) -- (3.5,3pt) node[above]{$r_{AH} = R_3$};
			\draw (3.5,15pt) node[above]{$r_{TN} = R_3$};
			\draw (4.5,3pt) -- (4.5,-3pt) node[below]{$4 \epsilon^{\frac{2}{5}}$};
 			\draw (5.5,3pt) -- (5.5,-3pt) node[below]{$5 \epsilon^{\frac{2}{5}}$};
 			\draw (6.5,-3pt) -- (6.5,3pt) node[above]{$r_j = R_2$};
			\draw (6.5,15pt) node[above]{$r_i = R_2$};
 			\draw (8.5,-3pt) -- (8.5,3pt) node[above]{$r = R_1$};
			\draw
 			(0,0) node[midway,yshift=-8.2em,xshift=-3.3em]{\begin{tabular}{l}
 					$p_j$ \\				
					$q_j$ \\
 					$p_j$ \\				
					$q_j$ \\
					$p_j$ \\				
					$q_j$ \\
					$p_j$ \\				
					$q_j$ \\
 			\end{tabular}};
 			\draw
 			(0,0) node[midway,yshift=-8.2em,xshift=-1.1em]{\begin{tabular}{l}
 					$g_{cf}$ \\				
					$g_{cf}$ \\
 					$\Omega$ \\ 
					$\Omega$ \\
 					$\rho$ \\ 
					$\rho$ \\ 
 					$\phi$ \\
					$\phi$
 			\end{tabular}};
 			\draw [decorate,decoration={brace,amplitude=5pt,mirror,raise=4ex}]
 			(0,0) -- (3.5,0) node[midway,yshift=-8.3em]{\begin{tabular}{c}
 				$g^{TN}$ \\		
				$g^{AH}$ \\				
 				$\epsilon^{-1}\sqrt{1 + \epsilon \alpha_i}$ \\ 
				$\epsilon^{-1}\sqrt{1 + \epsilon \alpha_j}$ \\ 
				$\log \left(\epsilon (1 + \epsilon \alpha_i)^{-1}R_3\right)$ \\ 
 				$\log \left(\epsilon (1 + \epsilon \alpha_j)^{-1}R_3\right)$ \\ 
 					$0$ \\ $0$
 			\end{tabular}};
 			\draw [decorate,decoration={brace,amplitude=5pt,mirror,raise=4ex}]
 			(3.6,0) -- (6.4,0) node[midway,yshift=-8.2em]{\begin{tabular}{c}
 				$g^{p_i}_{cf}$\\
				$g^{q_j}_{cf}$\\
				$r_i^{-1} (h^{p_i}_\epsilon)^{- \frac{1}{2}}$ \\ 
				$r_j^{-1} (h^{q_j}_\epsilon)^{- \frac{1}{2}}$ \\ 
				$\log(r_i)$ \\
				$\log(r_j)$ \\
 				$0$ \\ $0$
 			\end{tabular}};
 			\draw [decorate,decoration={brace,amplitude=5pt,mirror,raise=4ex}]
 			(6.6,0) -- (8.4,0) node[midway,yshift=-8.2em]{\begin{tabular}{c}
 					$g^{GH}$ \\
					$g^{GH}$ \\
 					$1$ \\ $1$ \\ 
					$1$ \\ 
 					$1$ \\ 
 					$0$ \\ $0$
 			\end{tabular}};
 			\draw [decorate,decoration={brace,amplitude=5pt,mirror,raise=4ex}]
 			(8.6,0) -- (10.75,0) node[midway,yshift=-8.2em]{\begin{tabular}{c}
 					$g^{GH}_{cf}$ \\
					$g^{GH}_{cf}$ \\
 					(Eq. \eqref{eq:def-omega-rho-near-infinity})\\ 
 					(Eq. \eqref{eq:def-omega-rho-near-infinity})\\ 
					(Eq. \eqref{eq:def-omega-rho-near-infinity})\\ 
					(Eq. \eqref{eq:def-omega-rho-near-infinity})\\ 
 					$\phi = \rho/r$ resp.\footnotemark\\
					$\phi = \rho/r$ resp.\footnotemark[\value{footnote}]
 			\end{tabular}};
 			
			\draw (-4em,-3em) -- (10.75,-3em);
			\draw (-0.75,-2.5em) -- (-0.75,-13em);
			\draw (0,-2.5em) -- (0,-13em);
 			\draw (3.5,-2.5em) -- (3.5,-13em);
 			\draw (6.5,-2.5em) -- (6.5,-13em);
 			\draw (8.5,-2.5em) -- (8.5,-13em);
 			
 		\end{tikzpicture}
		\caption{Definition of the conformally rescaled metric $g_{cf}$, and the functions $\Omega$, $\rho$ and $\phi$ on the global space $M_{B,n}$.}
	\label{table:definition-g-cf}
\end{table}
\footnotetext{On the asymptotic regiom, pick $\phi = \rho$ when $B = \R^2 \times S^1$ and pick $\phi = r$ when $B = \R \times T^2$.}
	For any $k \in \N$, $\alpha \in (0,1)$ and $\delta \in \R$, we define the weighted H\"older norms on $M_{B,n}$
	\begin{align*}
		\|u\|_{C^{k, \alpha}_{\delta}(M_{B,n})} =& \|e^{- \delta \rho} u\|_{C^{k, \alpha}_{cf}(M_{B,n})}, \\
		\|u + \lambda \phi \|_{C^{k, \alpha}_{\delta}(M_{B,n}) \oplus \R \phi} =& \|u\|_{C^{k, \alpha}_{\delta}(M_{B,n})} + |\lambda|.
	\end{align*}
 \end{definition} 

\begin{remark}
	The metric $g_{cf}$ is not a conformally rescaled version of the approximate \hk metric $g$ that is defined in Definition \ref{def:definition-g}!
\end{remark}

As a sanity check, one can consider $\Omega^{-2} \Delta^g \colon C^{k+2,\alpha}_{\delta}(M_{B,n}) \to C^{k,\alpha}_{\delta}(M_{B,n})$ and calculate its operator norm. Using the same argument as in Proposition 2.5 in \cite{Salm001}, one can calculate the norm on each region and show
 \begin{proposition}
	 \label{prop:global-analysis:strict-ellipticity}
	 Let $g$ be the approximate \hk metric from Definition \ref{def:definition-g}.
	 For all $k \in \N_{\ge 2}$, $\alpha \in (0,1)$, $\delta \in \R$, the operator 
	 $\Omega^{-2} \Delta^g$ is a linear map between $C^{k,\alpha}_{\delta}(M_{B,n})$ and $C^{k-2,\alpha}_{\delta}(M_{B,n})$, bounded uniformly with respect to the collapsing parameter $\epsilon$.
 \end{proposition}

We return to the inverse function theorem stated in Theorem \ref{thm:hk-triple:inverse-function-theorem}. Our goal is to apply it to Equation \eqref{eq:hk-triple:diff-eq-v2}.
Instead, for $\zeta \in \Omega^+_{g} \otimes \R^3$ let 
\begin{equation}
	\label{eq:def-fct-F}
	F(\zeta) := \Omega^{-2}\Delta \zeta + \frac{1}{2} \Omega^{-2}\Lambda^{-1} \Tf(\omega \wedge \omega) + 2 \Omega^{-2}\Lambda^{-1} \Tf(\d \d^* \zeta \wedge \d \d^* \zeta),
\end{equation}
measuring the hyperk\"ahlerness of $\omega_i + 2 \d \d^* \zeta_i$.
We identify the constant, linear and non-linear parts as
\begin{align*}
	F(0) =& \frac{1}{2} \Omega^{-2}\Lambda^{-1} \Tf(\omega \wedge \omega) \\
	L(\zeta) =& \Omega^{-2}\Delta \zeta \\
	N(\zeta) =& 2 \Omega^{-2}\Lambda^{-1} \Tf(\d \d^* \zeta \wedge \d \d^* \zeta),
\end{align*} 
where $\Delta$ is the Hodge Laplacian with respect to the approximate \hk metric $g$.
Using the decomposition $\zeta_i = \sum_{j} u_{ij} \omega_j$ and the Weitzenb\"ock formula from Equation \eqref{eq:hk-triple:Weitzenbock-formula}, 

\begin{equation}
	\label{eq:linearized-equation}
	L(u_{ij} \omega_j) = (\Omega^{-2} \Delta^{g} u_{ij})\: \omega_j - 2 \Omega^{-2}\nabla_{\nabla u_{ij}} \omega_j.
\end{equation}

According to Theorem 1.6 in \mycite{Salm001} and Theorem \ref{thm:appendix:isomorphism}, the operator
\begin{equation}
	\label{eq:laplacian-with-correct-domains}
\Omega^{-2} \Delta^{g}\colon \begin{Bmatrix}
	C^{k+2, \alpha}_{\delta}(M_{B,n}) & \text{if } B = \R^3 \\	
	C^{k+2, \alpha}_{\delta}(M_{B,n}) \oplus \R \phi & \text{otherwise}
\end{Bmatrix}
\to
C^{k, \alpha}_{\delta}(M_{B,n})
\end{equation}
is an isomorphism if $k \in \N$, $\alpha \in (0,1)$, $\delta < 0$ and
$|\delta|$ is sufficiently small. With this in mind we pick the domain of $F$ to be $(C^{k+2, \alpha}_{\delta}(M_{B,n}) \cdot \omega) \otimes \R^3 \subseteq \Omega^+(M_{B,n}) \otimes \R^3$ or $((C^{k+2, \alpha}_{\delta}(M_{B,n}) \oplus \R \phi) \cdot \omega)\otimes \R^3$ when $B \not= \R^3$. That is, we use the above decomposition of $\zeta_i = \sum_j u_{ij} \omega_j$ and assume that $u_{ij} \in C^{k+2, \alpha}_{\delta}(M_{B,n})$ or $C^{k+2, \alpha}_{\delta}(M_{B,n}) \oplus \R \phi$ respectively. Similarly, for the codomain we pick $(C^{k, \alpha}_{\delta}(M_{B,n}) \cdot \omega )\otimes \R^3$.

Having defined $F$ as a smooth map between Banach spaces, we check the three conditions of Theorem \ref{thm:hk-triple:inverse-function-theorem}. 
It turns out we only need to check these conditions for the $C^0$-norm. Namely, the error functions we will get are exponential with respect to the radial parameter $\rho$ that is given in Definition \ref{def:global-analysis:global-metric-and-functions}. Therefore, all higher regularity estimates will have the same growth and decay behaviour.

\subsection{The constant part} 
Let us examine the constant term $F(0) = \frac{1}{2}\Omega^{-2} \Lambda^{-1} \Tf(\omega \wedge \omega)$ with respect to the norm $(C^{k, \alpha}_\delta(M_{B,n}) \cdot \omega) \otimes \R^3$. We restrict our attention to the gluing regions, because outside of these regions, our manifold is \hk, which implies $\Tf(\omega \wedge \omega) = 0$. To study $\Lambda^{-1} \circ \Tf$ in local coordinates, let $\mu \in \Omega^4(M_{B,n})$ be a volume form and define $P \colon M_{B,n} \to \Mat_{3\times 3} (\R)$ by $\omega_i \wedge \omega_j = P_{ij} \mu$. By unpacking the definitions of $\Tf$ and $\Lambda$, one can show that
$$
\Lambda^{-1}\Tf(P \otimes \mu) = \sum_{ij} \left(\operatorname{Id} - \frac{1}{3} \operatorname{tr}(P) P^{-1}\right)_{ij} \omega_j \otimes e_i
$$
where $e_i$ is the standard orthonormal basis on $\R^3$.
According to Theorem \ref{thm:almost-hk:global-symplectic-triple},
\begin{equation*}
	\frac{1}{2}\omega_i \wedge \omega_j = \left(\operatorname{Id} + \O(\epsilon^{7/5})\right)_{ij} \otimes \Vol^{g^{q_j}}
\end{equation*}
on the gluing regions. Setting $\mu = \Vol^{g^{q_j}}$ and $P = 2 \operatorname{Id} + \O(\epsilon^{7/5})$, we conclude that $F(0)= \O(\Omega^{-2}\epsilon^{7/5})$ with respect to $g^{q_j}_{cf}$. By definition $\Omega^{-2} = r_i^2 \: h_\epsilon^{p_i}$ or $\Omega^{-2} = r_j^2 \: h_\epsilon^{p_j}$ depending on the type of singularity, but in both cases $\Omega^{-2} = \O(r_i^2) = \O(\epsilon^{4/5})$ on the gluing region.

We conclude that $F(0) = \O(\epsilon^{11/5})$ on the gluing region. These errors are measured with respect to the unweighted norm $g_{cf}$. Using Definition \ref{def:global-analysis:global-metric-and-functions} one can reintroduce the weights and conclude:
\begin{proposition}
	\label{prop:inverse-function-theorem:constant-part}
	The constant term $F(0)$ is of order $\O\left(\epsilon^{\frac{11 - 2 \delta}{5}}\right)$ with respect to $(C^{k, \alpha}_{\delta}(M_{B,n})\cdot \omega) \otimes \R^3$.
\end{proposition} 

\subsection{The linearised equation}
\label{subsec:linearised-eq}
Next we study the linearized part of $F$, which is given in Equation \eqref{eq:linearized-equation}. In order to apply the inverse function theorem, we need that 
\begin{enumerate}
	\item the operator in Equation \eqref{eq:laplacian-with-correct-domains} is invertible,
	\item the operator in Equation \eqref{eq:laplacian-with-correct-domains} has a uniformly bounded inverse, and 
	\item The error term $2 \Omega^{-2}\nabla^g_{\nabla^g u_{ij}} \omega_j$ in Equation \eqref{eq:hk-triple:Weitzenbock-formula} is sufficiently small.
\end{enumerate}
According to Theorem 1.6 in \mycite{Salm001} and Theorem \ref{thm:appendix:isomorphism}, the first condition is satisfied if $\delta < 0$ and $|\delta|$ sufficiently small. The last condition can be checked explicitly using the Koszul formula. Indeed, using the estimates in Theorem \ref{thm:almost-hk:global-symplectic-triple} --- i.e. $g = g^{p_i} + \O(\epsilon^{11/5})$, $\omega_i = \omega_i^{p_i} + \O(\epsilon^{11/5})$ and $\Omega^{2} = \O(\epsilon^{-4/5})$ with respect to $g_{cf}$ --- and the fact that $\nabla^{g^{p_i}} \omega^{p_i} = 0$, one can write $\Omega^{-2} \nabla^g u_{ij}$ and $\nabla^g \omega_j$ in local coordinates, and get expressions of the form
\begin{align*}
	\Omega^{-2} \nabla^g u_{ij} =& \sum_{\mu \nu} \left(\Omega^2 g\right)^{-1}_{\mu \nu} \: \frac{\del u_{ij}}{\del x^\mu} \del_\nu = \O(1) \\
	2\nabla^g_{\del_{\mu}} \omega_j (\del_\nu, \del_\rho) 
	=& 
	2 \frac{\del \omega_j(\del_\nu, \del_\rho)}{\del x_\mu} 
	+ 2 \omega_j(\nabla^g_{\del_\mu} \del_\nu, \del_\rho) + \ldots \\
	=& 
	\frac{\del \omega_j(\del_\nu, \del_\rho)}{\del x_\mu} 
	+ \sum_{\sigma \tau} (g^{-1})_{\sigma \tau} \omega_j(\del_\sigma, \del_\rho) \frac{\del g_{\mu \tau}}{ \del x^{\nu}} + \ldots \\
	=& 2\nabla^{g^{p_i}}_{\del_{\mu}} \omega^{p_i}_j (\del_\nu, \del_\rho) + \O(\epsilon^{7/5}) = \O(\epsilon^{7/5}).
\end{align*}
Therefore, we only need to show that the linearized operator has a uniformly bounded inverse, for which we will spend the rest of this section.
\begin{theorem}
	\label{thm:global-analysis:bounded-inverse}
	Let $\delta \in (-1, 0)$ (with $|\delta|$ sufficiently small if $B \not = \R^3$), $k \in \N_{\ge 2}$, and $\alpha \in (0,1)$.
	There exist $\epsilon_0, C > 0$ such that for any collapsing parameter $\epsilon \in (0, \epsilon_0)$ and $u \in C^{k, \alpha}_{\delta}(M_{B,n})$ (or $u + \lambda\phi \in C^{k, \alpha}_{\delta}(M_{B,n}) \oplus \R \phi$ when $B \not = \R^3$),
	\begin{align*}
		\|u\|_{C^{k, \alpha}_{\delta}(M_{B,n})} \le& C \: \|\Omega^{-2} \Delta^g u\|_{C^{k-2,\alpha}_{\delta}}
		&& \text{if } B  = \R^3\\
		\|u\|_{C^{k, \alpha}_{\delta}(M_{B,n})} + |\lambda| \le& C \: \|\Omega^{-2} \Delta^g(u + \lambda \phi)\|_{C^{k-2,\alpha}_{\delta}} 
		&& \text{otherwise }.
	\end{align*}
\end{theorem}

The proof of this theorem can be split into the following steps.
\begin{enumerate}
	\item Assume that there is no uniformly bounded inverse. There must be a sequence of functions $u_i$ and a sequence $\epsilon_i > 0$, such that $u_i$ has norm one, but $\Delta u_i$ and $\epsilon_i$ converge to zero\footnote{When $\epsilon$ does not tend to zero, the operator $\Omega^{-2} \Delta^g$ is continuous in $\epsilon$ and hence, the existence of the uniformly bounded inverse can be done by taking limits. We only need to consider the non-trivial case, when the collapsing parameter $\epsilon$ tends to zero.}.
	\item Using regularity estimates, construct a sequence of points $x_i$ at which the functions $|u_i|$ are uniformly bounded below, away from zero.
	\item Modify the functions $u_i$, such that their domain is on a fixed limiting space.
	\item Use the Arzela–Ascoli theorem to find a subsequence that converges to a non-zero harmonic function $u$.
	\item Argue that the limiting space has no non-zero harmonic functions, and reach a contradiction.
\end{enumerate}
Depending on whether the $x_i$ will concentrate near one of the singularities, we will pick different limiting spaces and apply different transformations to $u_i$. But for each case, we will follow the above steps.
\begin{remark}
	The proof for the case $B = \R^3$ will be a simplified version of the proof for the case $B \not = \R^3$. Hence the rest of this section we only consider the latter case.
\end{remark}

\subsubsection*{Step 1.}
\begin{lemma}
	Suppose that Theorem \ref{thm:global-analysis:bounded-inverse} is false. Then there exist sequences $u_i \in C^{k, \alpha}_\delta(M_{B,n})$, $\lambda_i \in \R$, $\epsilon_i \in (0,1)$ and $c > 0$ such that
	\begin{align*}
		\|u_i\|_{C^{k, \alpha}_\delta(M_{B,n})} + |\lambda_i| &= 1,
		& \|u_i\|_{C^{k, \alpha}_\delta(M_{B,n})} &> c \\
		\|\Omega^{-2} \Delta^g (u_i + \lambda_i \phi)\|_{C^{k-2, \alpha}_\delta(M_{B,n})}  &\to 0, \text{ and}  & \epsilon_i &\to 0.
	\end{align*}
\end{lemma}
\begin{proof}
	The conditions on the left follow directly from the negation of Theorem \ref{thm:global-analysis:bounded-inverse}. We only need to show $\|u_i\|_{C^{k, \alpha}_\delta(M_{B,n})}$ is bounded below.
	Suppose not, and assume that $\|u_i\|_{C^{k, \alpha}_\delta(M_{B,n})}$ converges to zero.
	Because $|\lambda_i| \le 1$, there must be a converging subsequence with limit $\lambda$. Because $\|u_i\|_{C^{k, \alpha}_\delta(M_{B,n})} + |\lambda_i| = 1$ and $\|u_i\|_{C^{k, \alpha}_\delta(M_{B,n})}$ converges to zero, the limit $\lambda$ must be equal to $\pm1$.
	
	At the same time, $\|\Omega^{-2} \Delta^g (u_i + \lambda_i \phi)\|_{C^{k-2, \alpha}_\delta(M_{B,n})} \to 0$ and hence $\Delta^g \phi = 0$. The function $\phi$ is not a harmonic function, which yields a contradiction. Hence, $\|u_i\|_{C^{k, \alpha}_\delta(M_{B,n})}$ is uniformly bounded away from zero.
\end{proof}

\textbf{Step 2.} Next we study the property $\|u_i\|_{C^{k, \alpha}_\delta(M_{B,n})} > c$ in more detail. Notice that for any Riemannian metric, its $C^{k, \alpha}$-norm can be written as 
$$
\|u\|_{C^{k, \alpha}(U)} = \sup_{x \in U} \left[\sum_{j=0}^k \|\nabla^j u(x) \| + \sup_{\substack{y \in U \\ d(x,y) < \operatorname{InjRad}(x)}} \frac{\| \nabla^k u(x) - \nabla^k u(y) \|}{d(x,y)^\alpha} \right]
$$
which enables us to define a `pointwise norm':
$$
\|u_i\|_{C^{k, \alpha}(\{x\})} := \sum_{j=0}^k \|\nabla^j u(x) \| + \sup_{\substack{y \in U \\ d(x,y) < \operatorname{InjRad}(x)}} \frac{\| \nabla^k u(x) - \nabla^k u(y) \|}{d(x,y)^\alpha}.
$$
With this in mind, we define a weighted `pointwise norm' with respect to $g_{cf}$.
The condition $\|u_i\|_{C^{k, \alpha}_\delta(M_{B,n})} > c$ implies there is a sequence $x_i \in M_{B,n}$ such that $\| u_i \|_{C^{k, \alpha}_\delta(\{x_i\})} > \frac{c}{2} >0$. The sequence of points $x_i$ can behave in two different ways, which is shown in Figure \ref{fig:possible-ways-xi-can-behave}:
\begin{figure}[h!]
\centering
	\begin{tikzpicture}
		\draw[very thick] (0,0) -- (7.5,0);
		\draw[very thick,dashed] (7.5,0) -- (9.5,0);
		\draw[very thick,-{latex}] (9.5,0) -- (12,0) node[below]{$r \to \infty$};
		\draw (0,3pt) -- (0,-3pt) node[below]{$0$};
		\draw (4,3pt) -- (4,-3pt) node[below]{$4 \epsilon^{\frac{2}{5}}$};
		\draw (6,3pt) -- (6,-3pt) node[below]{$5 \epsilon^{\frac{2}{5}}$};
		\draw (3,-3pt) -- (3,3pt) node[above]{\begin{tabular}{c}
				$r_{TN} = R_3$ \\
				$r_{AH} = R_3$ \\
		\end{tabular}};
		\draw (7,-3pt) -- (7,3pt) node[above]{$r_j = R_2$};
		\draw (10,-3pt) -- (10,3pt) node[above]{$r = R_1$};
		\draw [decorate,decoration={brace,amplitude=5pt,raise=9ex}]
				(6.5,0) -- (12.75,0) node[midway,yshift=6em]{\begin{tabular}{c}
						Case 2: \\
						$x_i$ bounded away\\from the singularities \\
						~
				\end{tabular}};
		\draw [decorate,decoration={brace,amplitude=5pt,raise=9ex}]
		(0,0) -- (6.4,0) node[midway,yshift=6em]{\begin{tabular}{c}
				Case 1: \\
				$x_i$ concentrate near singularity\\
				~
		\end{tabular}};
	\end{tikzpicture}
	\caption{The two possible ways $x_i$ can behave.}
	\label{fig:possible-ways-xi-can-behave}
\end{figure}
\begin{enumerate}
	\item The sequence $x_i$ \textit{concentrates near a singularity}. That is, there is a subsequence of $x_i$ such that the radial coordinate $r_i$ or $r_j$ at $x_i$ converges to zero.
	\item The sequence $x_i$ \textit{is bounded away from the singularities}. That is, there is a subsequence of $x_i$ such that the radial coordinate at $x_i$ is uniformly bounded below.
\end{enumerate}
At least one of these cases must happen, and we study them separately.

\begin{remark}
	The case when $x_i$ concentrate near a non-fixed point singularity $p_i$, is similar to the case when $x_i$ concentrate near a fixed point singularity $q_j$. Therefore, we only explain the latter case.
\end{remark}

\subsubsection*{Case 1: $x_i$ concentrates near a singularity.}
\textbf{Step 3.} We consider the case when $x_i$ concentrates near a singularity. In this case $u_i|_{\{r_j \le 2 r_j(x_i)\}}$ is uniformly bounded away from zero in the $C^{k, \alpha}_\delta$ norm. At the same time, $\{r_j \le 2 r_j(x_i)\}$ can be viewed as a subset of the Atiyah-Hitchin manifold.
Therefore, we use the Atiyah-Hitchin manifold as our limiting space.

To make our contradiction argument work, we need the norms, operators, and weights on the limiting space to be independent of $\epsilon$. We constructed $g_{cf}$ such that this is true. We also chose $\Omega$ such that $\Omega^{-2} \Delta^g$ is independent of $\epsilon$. However, the radial parameter $\rho$ does depend on $\epsilon$. To solve this we define a new radial parameter $\rho_{AH} := \rho - \log \left(\frac{\epsilon}{1 + \epsilon \alpha_j}\right)$ that is independent of $\epsilon$, and we equip the Atiyah-Hitchin manifold with the weighted norm
$$
\|u\|_{C^{k, \alpha}_{\delta}(AH)} = \|e^{- \delta \rho_{AH} }u\|_{C^{k, \alpha}_{g_{cf}}(AH)}.
$$
Luckily, the weighted operator $L_\delta$ is the same whether we use $\rho$ or $\rho_{AH}$.

Next, we will restrict $u_i$ such that it is fully supported on the Atiyah-Hitchin manifold. For this we consider the family of smooth step functions $\chi_i$ on $M_{B,n}$ that are equal to $1$ when $r_j \le 2 r_j(x_i)$ and equal to $0$ when\footnote{$R_2$ is defined in Definition \ref{def:global-analysis:global-metric-and-functions}.} $r_j \ge R_2$. Then $u_i \cdot \chi_i$ are compactly supported functions on the Atiyah-Hitchin manifold. Because $u_i \in C^{k, \alpha}_\delta(M_{B,n})$ is equivalent to $e^{- \delta \rho} u_i \in C^{k, \alpha}_{cf}(M_{B,n})$, we get
$$
e^{- \delta \: \rho} u_i = e^{- \delta \:\rho_{AH}} \left(\frac{\epsilon}{1 + \epsilon \alpha_j}\right)^{- \delta} u_i \in C^{k, \alpha}_{cf}(M_{B,n}).
$$
With this insight, we consider a new sequence of functions $\tilde u_i := \chi_i \cdot \left(\frac{\epsilon}{1 + \epsilon \alpha_j}\right)^{- \delta} u_i$ defined on the Atiyah-Hitchin manifold. In the following lemma, we show that $\tilde u_i$ has the same properties as $u_i$:

\begin{lemma}
	\label{lem:global-analysis:bounded-inverse-transformation-AH}
	Suppose that Theorem \ref{thm:global-analysis:bounded-inverse} is false and that $x_i$ concentrate near a singularity $q_j$. Let $AH$ be the Atiyah-Hitchin manifold. Then the sequence $\tilde u_i := \chi_i \cdot \left(\frac{\epsilon}{1 + \epsilon \alpha_j}\right)^{- \delta} u_i \in C^{k, \alpha}_\delta(AH)$ satisfy
	\begin{align*}
		\frac{c}{2}< \|\tilde u_i\|_{C^{k, \alpha}_\delta(AH)} \le 1, \text{ and } 
		\|\Omega^{-2} \Delta^g \tilde u_i\|_{C^{k-2, \alpha}_\delta(AH)}  &\to 0.
	\end{align*}
\end{lemma}

\begin{proof}
	Firstly, we only modified $u_i$ outside of the region where $x_i$ concentrates and hence $\|\tilde u_i\|_{C^{k, \alpha}_\delta(AH)}$ is bounded below by $\frac{c}{2}$.
	
	Secondly, the step function $\chi_i$ is chosen such that $\d \chi_i$ and its derivatives are of order $(\log(R_2) -\log(2 r_j(x_i)))^{-1}$ with respect to $g_{cf}$ and this converges to zero. Hence,
	$$
	\|\tilde u_i\|_{C^{k, \alpha}_{\delta}(AH)} 
	\le \|\chi_i\|_{C^{k, \alpha}_{cf}(M_{B,n})} \cdot \| u_i\|_{C^{k, \alpha}_{\delta}(M_{B,n})} 
	\le 1.
	$$

	Finally, to estimate $\|\Omega^{-2} \Delta^g \tilde u_i\|_{C^{k-2, \alpha}_\delta(AH)}$ notice that on the support of $\chi_i$ the function $\phi$ is identically zero and
	$$
	\Omega^{-2} \Delta^g(\chi_i \cdot u_i) = \chi_i \cdot \Omega^{-2} \Delta^g(u_i) + u_i \cdot \Omega^{-2} \Delta^g(\chi_i) - 2 \Omega^{-2} \langle \d \chi_i, \d u_i \rangle_{g}.
	$$
	Using that $g$ and $g^{q_j}$ are equivalent norms and that $\d\chi_i$  is decaying, we estimate
	\begin{align*}
		\|\Omega^{-2} \Delta^g(\tilde u_i)\|_{C^{k, \alpha}_{\delta}(AH)}
		\le& \|\chi_i\|_{C^{k, \alpha}_{cf}(M_{B,n})} \cdot \|\Omega^{-2} \Delta^g(u_i)\|_{C^{k, \alpha}_{\delta}(M_{B,n})} \\		
		 & \qquad+ \|u_i\|_{C^{k, \alpha}_{\delta}(M_{B,n})} \cdot \|\Omega^{-2} \Delta^g(\chi_i)\|_{C^{k, \alpha}_{cf}(M_{B,n})} \\
		 & \qquad + \O\left(\frac{1}{ -\log( r_j(x_i))}\right).
	\end{align*}
	By Proposition \ref{prop:global-analysis:strict-ellipticity}, $\Omega^{-2} \Delta^g \chi_i$ is uniformly bounded by $\| \d \chi_i \|_{C^{k, \alpha}_{cf}(M_{B,n})}$ and so $\|\Omega^{-2} \Delta^g \tilde u_i\|_{C^{k-2, \alpha}_\delta(AH)}$ converges to $0$.
\end{proof}

\textbf{Step 4.} Using the Arzela-Ascoli theorem, there exists a subsequence of $\tilde{u}_i$ which converges to some $\tilde u \in C^0_\delta(K)$ for any compact set $K$. We restrict $\tilde{u}_i$ to this subsequence. According to Theorem 1.5 in \cite{Salm001}, there is a uniform constant $ C > 0$ and a compact set $K \subset AH$ such that for any $i,j \in \N$,
\begin{align*}
	\| \tilde u_i - \tilde u_j\|_{C^{k,\alpha}_{\delta}(AH)} \le& C \left[
	\| \Omega^{-2}\Delta^{g} (\tilde u_i - \tilde u_j)\|_{C^{k-2, \alpha}_{\delta}(AH)}
	+ \| \tilde u_i - \tilde u_j\|_{C^0_{\delta}(K)} 
	\right].
\end{align*}
Therefore,  $\tilde u_i$ is a Cauchy sequence in $C^{k, \alpha}_\delta(AH)$ and its limit is $\tilde u \in C^{k, \alpha}_\delta(AH)$.

\textbf{Step 5.} This limiting function $\tilde u$ is harmonic because $\Omega^{-2} \Delta$ is a continuous operator. By assumption, $\delta < 0$, and hence $\tilde u$ must be decaying. By the maximum principle $\tilde u$ must vanish everywhere.
We conclude $\|\tilde{u}_i\|_{C^{k, \alpha}_{\delta}(AH)}$ converges to zero, which contradicts the fact that $\|\tilde{u}_i\|_{C^{k, \alpha}_{\delta}(AH)} > c/2 > 0$.
Therefore, the sequence $x_i$ cannot concentrate near the singularities.

\subsubsection*{Case 2: $x_i$ is bounded away from the singularities}
\textbf{Step 3.} Next we consider the case in which $x_i$ is bounded away from the singularities. Again we need to modify $u_i$, such that their domain is defined on a fixed limiting space. The points $\{x_i\}$ lie inside the circle bundle $P$ on which the Gibbons-Hawking metric is defined.
The radius of the fibres of $P$ is $\O(\epsilon)$, and hence we expect that, in the limit $\epsilon_i \to 0$, $P$ collapses to its base space $B'$. At the same time, the non-complete regions of $B'$ shrink at rate $\O(\epsilon^{2/5})$, and hence we pick the flat space $B$ as our limiting space.

Next we modify $u_i$ such that they are well-defined on the limiting space.
Because the points $\{x_i\}$ are bounded away from the singularities, there is a constant $R_B$ such that $r_j(x_i)>R_B$. Therefore, consider the family of smooth step functions $\chi_i$ on $M_{B,n}$ that are equal to 1 when $r_i, r_j \ge R_B$ and equal to zero when $r_i,r_j \le 5 \epsilon^{2/5}$. We consider a new sequence of functions $\tilde{u}_i$, that is the $S^1$-invariant component of $u_i \cdot \chi_i$ on the circle bundle $\pi \colon P \to B'$, i.e. 
$$
\tilde{u}_i(x) := \frac{1}{2 \pi} \int_{\pi^{-1}(x)} u_i \: \chi_i \cdot \eta
$$
where $\eta$ is the connection on $P$. In the following lemma, we show that $\tilde u_i$ has the same properties as $u_i$:

\begin{lemma}
	\label{lem:global-analysis:bounded-inverse-transformation-GH}
	Suppose that Theorem \ref{thm:global-analysis:bounded-inverse} is false and $x_i$ are bounded away from the singularities. Then, the sequence $\tilde u_i$ satisfy
	\begin{align*}
		c< \|\tilde u_i\|_{C^{0}_\delta(\{x_i\})} + |\lambda_i| \le 2, \text{ and } 
		\|\Omega^{-2} \Delta^g (\tilde u_i + \lambda_i \phi)\|_{C^{0, \alpha}_\delta(P)}  &\to 0.
	\end{align*}
	for some constant $c > 0$.
\end{lemma}
\begin{proof}
	Using the same arguments given in Lemma \ref{lem:global-analysis:bounded-inverse-transformation-AH}, one can show
	$$
	\frac{c}{2} \le \|u_i \: \chi_i\|_{C^{k, \alpha}_\delta(P)} \le 1
	\text{ and }
	\|\Omega^{-2} \Delta^g (u_i \chi_i  + \lambda_i \phi)\|_{C^{k-2, \alpha}_\delta(P)} \to 0.
	$$
	To compare these results with those stated in the Lemma, recall that $g_{cf}$ is constructed from $S^1$-invariant metrics. Therefore, the operator that projects any function to its $S^1$-invariant component, is a bounded operator on $C^{0, \alpha}_\delta(P)$ and commutes with the Laplacian. This implies that $\|\Omega^{-2} \Delta^g (\tilde u_i + \lambda_i \phi)\|_{C^{0, \alpha}_\delta(P)}$ must converge to zero and that $\|\tilde u_i\|_{C^{0}_\delta(\{x_i\})} + |\lambda_i| \le 2$.
	
	For the last part, assume that both $\|\tilde u_i\|_{C^{0}_\delta(\{x_i\})}$ and $|\lambda_i|$ converge to zero. According to the local Schauder estimate given\footnote{Proposition 2.10 in \cite{Salm001} only proves this estimate for the asymptotic region of $P$. However, its proof is based on the fact that up to some local universal cover, the asymptotic region of $P$ has uniformly bounded geometry with respect to $g_{cf}$. In our case this is true everywhere away from the singularities.} by Proposition 2.10 in \cite{Salm001} and Proposition \ref{prop:appendix:elliptic-regularity}, there exist some constants $C,r>0$, independent of $x_i$ and $\epsilon$, such that
	\begin{align*}
		\frac{c}{2} \le \| u_i \chi_i\|_{C^{k,\alpha}_{\delta}(B_r(x_i))} \le& C \left[
		\| \Omega^{-2}\Delta^{GH} (u_i\chi_i + \lambda_i \phi) \|_{C^{k-2, \alpha}_{\delta}(B_{2r}(x_i))}
		\right. \\
		& + |\lambda_i|\cdot \| \Omega^{-2}\Delta^{GH} \phi \|_{C^{k-2, \alpha}_{\delta}(B_{2r}(x_i))}
		+ \| \tilde{u}_i\|_{C^0_{\delta}(B_{2r}(x_i))}  \\
		& 
		\left. 
		+ \| u_i \chi_i - \tilde{u}_i\|_{C^0_{\delta}(B_{2r}(x_i))} 
		\right].
	\end{align*}
	Except for $\| u_i \chi_i - \tilde{u}_i\|_{C^0_{\delta}(B_{2r}(x_i))}$, all terms on the right hand side converge to zero. So we conclude $\| u_i \chi_i - \tilde{u}_i\|_{C^0_{\delta}(B_{2r}(x_i))}$ is bounded below.
	
	We can estimate $\tilde{u}_i^f := u_i \chi_i - \tilde{u}_i$ explicitly. 
	Indeed, the function $\tilde{u}_i^f$ has no $S^1$-invariant part, and so for each $x \in B_{2r}(x_i)$ there exists a $t_x \in [0,2\pi]$ such that $\tilde{u}_i^f(e^{i t_x} \cdot x) = 0$. Using that $\tilde{u}_i$ is $S^1$-invariant, the fundamental theorem of calculus implies
	$$\tilde{u}_i^f(x) = -\int_{t=0}^{t= t_x} \frac{\del}{\del t} \tilde{u}_i^f (e^{i t}\cdot x) \d t = -\int_{t=0}^{t= t_x} \frac{\del}{\del t} u_i (e^{i t}\cdot x) \cdot \chi_i(x) \d t.$$
	With respect to $g_{cf}$, the fibre at $x_i$ has length of order $\epsilon$ and so
	$$
	|\tilde{u}_i^f(x_i)| \le \int_{t=0}^{t= 2 \pi} 
	\left|\d u_i (\del t) \right| \d t = \O(\epsilon) \cdot \|\d u_i \|.
	$$

	We conclude that $\| u_i \chi_i - \tilde{u}_i\|_{C^0_{\delta}(B_{2r}(x_i))}$ converges to zero and is bounded below at the same time. Therefore, the assumption that both $\|\tilde u_i\|_{C^{0}_\delta(P)}$ and $|\lambda_i|$ converge to zero is false.
\end{proof}

\textbf{Step 4.} Next we want to find a subsequence of $\tilde u_i + \lambda_i \phi$ which converges to some harmonic function on $B$. First we need to determine what the limiting metric will be. 
For this, notice that on the support of $\tilde u_i$ the metric $g_{cf}$ is an interpolation of metrics that can be decomposed into some uniform metric $\tilde g_B$ on the base space and a part that is of order $\epsilon$. For example, the metric
$$
g^{q_j}_{cf} = r_j^{-2} \d r_j + g_{S^2} + \frac{\epsilon^2}{r_j h_\epsilon^2} \eta_j^2
$$
can be written in the form $\tilde g_B + \O(\epsilon)$ and the limiting metric is $\tilde g_B =  r_j^{-2} \d r_j + g_{S^2}$. We conclude that in the limit $\epsilon \to 0$, the metric $g_{cf}$ degenerates to a metric on $B \setminus \cup\{p_i, q_j\}$.
Therefore, for any compact sets $K \subset K' \subset B \setminus \cup\{p_i, q_j\}$, we have the Schauder estimate
$$
\|\tilde{u}_i - \tilde{u}_j\|_{C^{k, \alpha}_{\tilde g_B}(K)} \le 
C \left[
\|\Omega^{-2} \Delta^g (\tilde{u}_i - \tilde{u}_j)\|_{C^{k-2, \alpha}_{\tilde g_B}(K')}
+ \|\tilde{u}_i - \tilde{u}_j\|_{C^{0}_{\tilde g_B}(K')}
\right].
$$
This estimate does not change significantly if we introduce $\lambda_i \phi$ on the right hand side.
Therefore, according to the Arzela-Ascoli theorem, there is a subsequence of $\tilde u_i$ which converges in $C^{k, \alpha}_{\tilde g_B}(K)$. By exhausting the punctured base space by compact sets, applying Arzela-Ascoli on each of them, and taking the diagonal sequence, we conclude:
\begin{lemma}
	\label{lem:global-analysis:bounded-inverse-harmonic-basespace}
	There exists a twice differentiable function $\tilde u$ on $B \setminus \cup\{p_i, q_j\}$ and a $\lambda \in [-1,1]$, such that for any compact set $K \subset B \setminus \cup\{p_i,q_j\}$, 
	\begin{align*}
		\tilde u_i \to& \tilde u \in C^{2, \alpha}_{\tilde g_B}(K) &
		\lambda_i \to& \lambda
	\end{align*}
	and
	$$
	\Delta^{B} (\tilde u + \lambda \phi) = \Delta^g (\tilde u + \lambda \phi) = 0.
	$$
\end{lemma}

In the last part of the lemma, one has to notice that on the support of $u$ the metric $g$ is $g^{GH}$ and $\Delta^{GH} = \frac{1}{h_\epsilon} \Delta^B$ for $S^1$-invariant functions.

Before we can make any qualitative statement about $\tilde u + \lambda \phi$, we need to consider its behaviour near the boundary of $B \setminus \cup\{p_i,q_j\}$. 
According to Theorem 1.5 in \cite{Salm001} and Proposition \ref{prop:appendix:Fredholm-estimate}, there is a uniform constant $C> 0$ and there are large compact sets $K \subset K' \subset P$ covering the singularities such that
\begin{align*}
	\|\tilde{u}_i - \tilde{u}_j\|_{C^{2, \alpha}_{\delta}(P\setminus K')} 
	\le& 
	C \left[
	\|\Omega^{-2} \Delta^g (\tilde{u}_i + \lambda_i \phi)\|_{C^{0, \alpha}_{\delta}(P\setminus K)}
	\right. \\
	& \qquad+ \|\Omega^{-2} \Delta^g (\tilde{u}_j + \lambda_j \phi)\|_{C^{0, \alpha}_{\delta}(P\setminus K)}	\\
	& 
	\qquad + |\lambda_i - \lambda_j| \cdot \|\Omega^{-2} \Delta^g \phi \|_{C^{0, \alpha}_{\delta}(P\setminus K)} \\
	& \qquad+ \left.\|\tilde{u}_i - \tilde{u}_j\|_{C^{0}_{\delta}(K' \setminus K)}
	\right].
\end{align*}
Because the right hand side converges to zero, $\tilde{u}_i$ is a Cauchy sequence in $C^{2, \alpha}_{\delta}(P\setminus K')$, which implies its limits decays with order $e^{\delta \rho}$ near infinity.

Next we study the behaviour near the punctures, where $g_{cf}$ is just $g^{p_i}_{cf}$ or $g^{q_j}_{cf}$. We show that $\tilde{u}_b$ can be smoothly extended over the singularities, by applying the removable singularity theorem. For this we need to show that $\tilde u$ has some slow polynomial divergence near each singularity. 
\begin{lemma}
	On any compact neighbourhood $K$ of $p_i$ or $q_j$ inside $B$ and any $\tilde \delta < \delta$, $r_j^{- 2\tilde \delta} \tilde{u} \in C^0(K)$.
\end{lemma}
\begin{proof}
	Assume without loss of generality that $K$ is a compact neighbourhood of $q_j.$ Let $\tilde \epsilon > 0$ be arbitrary. There is some small open ball $B(q_j)$ such that on this ball $|r_j^{\delta - \tilde \delta}| < \tilde{\epsilon}$. Hence, for any $k,l \in \N$,
	\begin{align*}
		\|r_j^{- \tilde \delta} (\tilde{u}_k - \tilde{u}_l) \|_{C^0(K)} 
		\le&
		\|r_j^{\delta - \tilde \delta}\|_{C^0(B(q_j))} \cdot 
		\left(
		\|r_j^{- \delta} \tilde{u}_k  \|_{C^0(B(q_j))}
		+ \|r_j^{- \delta} \tilde{u}_l \|_{C^0(B(q_j))}
		\right)  \\
		&+ \|r_j^{\delta - \tilde \delta}\|_{C^0(K \setminus B(q_j))} \cdot \|r_j^{- \delta} (\tilde{u}_k - \tilde{u}_l )\|_{C^0(K \setminus B(q_j))} \\
		\le&
		2 \tilde \epsilon
		+ \|r_j^{\delta - \tilde \delta}\|_{C^0(K \setminus B(q_j))} \cdot \|r_j^{- \delta} (\tilde{u}_k - \tilde{u}_l )\|_{C^0(K \setminus B(q_j))},
	\end{align*}
	where in the last step we used the upper bound of $r_j^{- \delta} \tilde{u}_k$ from Lemma \ref{lem:global-analysis:bounded-inverse-transformation-GH}.

	Next we apply Lemma \ref{lem:global-analysis:bounded-inverse-harmonic-basespace} on the compact set $K \setminus B(q_j)$ and this implies there exists an $N \in \N$ such that all $k,l > N$, 
	$$
	\|r_j^{- \delta} (\tilde{u}_k - \tilde{u}_l )\|_{C^0(K \setminus B(q_j))} < \frac{\tilde \epsilon}{\|r_j^{\delta - \tilde \delta}\|_{C^0(K \setminus B(q_j))}}.
	$$ 
	We conclude $\|r_j^{- \tilde \delta} (\tilde{u}_k - \tilde{u}_l) \|_{C^0(K)}  \le 3 \tilde 
	\epsilon$ and so $r^{- \tilde \delta} \tilde{u}_k$ is a Cauchy sequence in $C^0(K)$ with limit $r_j^{- \tilde \delta} \tilde{u}$.
\end{proof}

We conclude that $\tilde{u} + \lambda \phi$ is a harmonic function on $B$ and $\tilde u$ decays with order $e^{\delta \rho}$ near infinity.

\textbf{Step 5.} With the asymptotic behaviour of $\tilde u$ understood, we can show $\tilde u = \lambda \phi = 0$. The only harmonic functions on $B$ that are of order $\O(\rho)$ are the harmonic polynomials of degree 1. Using that $\tilde{u} + \lambda \phi$ is $\Z_2$ invariant, we conclude $\tilde{u} + \lambda \phi$ is constant. Because the map $\phi$ is unbounded, $\lambda$ must be equal to zero. Finally, the function $\tilde u$ is decaying, and the only constant that is decaying is the constant zero function. Therefore, $\tilde u = \lambda \phi = 0$.

Finally, we prove Theorem \ref{thm:global-analysis:bounded-inverse}. If this theorem is false, then the sequence $x_i$ cannot have a converging subsequence in $P$, due to the lower bound in Lemma \ref{lem:global-analysis:bounded-inverse-transformation-GH}. When $x_i$ diverges, Theorem 1.5 in \cite{Salm001}, Proposition \ref{prop:appendix:Fredholm-estimate} and Lemma \ref{lem:global-analysis:bounded-inverse-transformation-GH} imply
\begin{align*}
	c < \| \tilde u_i\|_{C^{2,\alpha}_{\delta}(\{x_i\})} \le& C \left[
	\| \Omega^{-2}\Delta^{g}  (\tilde u_i + \lambda_i \phi) \|_{C^{0, \alpha}_{\delta}(P)}
	\right.\\
	&+|\lambda_i|\cdot \| \Omega^{-2}\Delta^{g} \phi \|_{C^{0, \alpha}_{\delta}(P)} \left. 
	+ \| \tilde u_i\|_{C^0_{\delta}(K)} 
	\right],
\end{align*}
for some compact set $K$ and constants $c, C > 0$. At the same time, the right hand side will converge to zero as $\|\tilde{u}_i\|_{C^0_\delta(K)}$ converges to $\|\tilde{u}\|_{C^0_\delta(K)} = 0$.

\subsection{The non-linear part and existence}
Finally, we study the non-linear part $N(\zeta) = 2 \Omega^{-2} \Tf ( (\d \d^* \zeta)^2)$ and prove the existence of the \hk triple. We do this in multiple steps: First, we estimate $\d \d^* \zeta$ in terms of $C^{k, \alpha}_{cf}(\Omega^2(M_{B,n}))$. Secondly, we work out $N(\zeta) - N(\xi)$ using the product rule for H\"older norms, which yields an explicit error. Finally, we calculate this error on each region separately.
\begin{lemma}
	\label{lem:inverse-function-theorem:derivative-non-linear-part}
	Let $\zeta \in (C^{k+2, \alpha}_{\delta}(M_{B,n}) \cdot \omega) \otimes \R^3$.
	There exists a constant $C > 0$, independent of $\zeta$ and $\epsilon$, such that
	$$
	\|\d \d^* \zeta\|_{C^{k,\alpha}_{\delta}(\Omega^2(M_{B,n}))} \le C \cdot \| \zeta\|_{(C^{k+2, \alpha}_{\delta}(M_{B,n}) \cdot \omega) \otimes \R^3}.
	$$
\end{lemma}
\begin{proof}
	Expand $\zeta$ into $\zeta_i = \sum_j u_{ij} \omega_j$.
	In any local coordinates $\{x_i\}$, $\d \d^* \zeta_i$ can be written as
	$$
	\d \d^* \zeta_i
		= - \sum_j \d \iota_{\nabla^g u_{ij}} \omega_j
		= - \sum_{jkl} \d \left[g^{-1}_{kl} \frac{\del u_{ij}}{ \del x_l} \cdot \omega_j(\del_k, \ldots) \right]
	$$
	According to Theorem \ref{thm:almost-hk:global-symplectic-triple}, $\omega$ is of order $\O(\Omega^{-2})$, while $g^{-1}$ is of order $\O(\Omega^{2})$, making $\d \d^* \zeta_i = \O(1)$.
\end{proof}

We also calculate $\d \d^* (\phi \omega)$. Sadly, this term does not lie in $C^{k,\alpha}_{\delta}(\Omega^2(M_{B,n}))$. To measure its norm we need to separate the ALH* case from the ALG/ALG*/ALH cases.

\begin{lemma}
	\label{lem:inverse-function-theorem:additional-estimates-non-linear-part}
	When $B = \R^2 \times S^1$, there exists a constant $C > 0$, independent of $\epsilon$, such that 
	$$
	\|\d \d^* (\phi \: \omega)\|_{C^{k,\alpha}_{cf}(\Omega^2(M_{B,n}))} \le C.
	$$
	The same estimate is true when $B = \R \times T^2$ and $P$ is the trivial bundle.

	When $B = \R\times T^2$ and $P$ is a non-trivial circle bundle, there exists a constant $C > 0$, independent of $\epsilon$, such that 
	$$
	\|e^{- \rho } \d \d^* (\phi \: \omega)\|_{C^{k,\alpha}_{cf}(\Omega^2(M_{B,n}))} \le C.
	$$
\end{lemma}
\begin{proof}
	When $\phi$ is non-zero, $\omega$ is self-dual with respect to $*^{GH}$. Therefore, one can show that 
	$$
	\d \d^* (\phi \: \omega) = - \d *^{GH} ( \d \phi \wedge \omega ) = - \d *^{cf} ( \d \phi \wedge (\Omega^{2} \omega) ).
	$$

	When $B = \R^2 \times S^1$ the norm of $\d \phi$ is uniformly bounded with respect to $g_{cf}$, and so $\|\d \d^* (\phi \omega)\|_{C^{k,\alpha}_{cf}(\Omega^2(M_{B,n}))} = \O(\Omega^2 \omega) = \O(1)$.
	When $B = \R \times T^2$, $\d \phi = \O(e^\rho)$, which explains the other estimate.

	The only case we need to study is when $B = \R \times T^2$ and $P$ is the trivial vector bundle. According to Lemma \ref{lem:setup:harmonic-function} and Definition \ref{def:setup:principal-bundle}
	$$
	h_\epsilon = 1 + \epsilon(A + \O(e^{-r})),
	$$
	where $A \in \R$. An explicit calculation shows that
	$$
	\d\d^* (\phi \: \omega_i) = \d *^{GH} \d \phi \wedge \omega_i = \d \left(- \frac{\del \phi}{\del x_i} \: \epsilon h^{-1}_\epsilon \eta + *^B (\d \phi \wedge \d x_i)\right).
	$$
	On the region $\phi = r = x_1$, this simplifies to
	$$
	\d^* (\phi \: \omega_i) = \d\begin{cases}
		- \epsilon  h^{-1}_\epsilon \eta & i = 1 \\
		\d x_3 & i = 2 \\
		-\d x_2 & i = 3,
	\end{cases}
	$$
	and so we only need to estimate $|\epsilon \d (h_\epsilon^{-1} \eta)|_{cf}$, which is
	\begin{align*}
		|\epsilon \d ( h_\epsilon^{-1} \eta)|_{cf}
		=& |h_\epsilon^{-1} *^{B }\d h_\epsilon - \epsilon h_\epsilon^{-2}\d h_\epsilon \wedge \eta|_{cf} \\
		=& \sqrt{2} \left| \frac{r}{h_\epsilon} \d h_\epsilon\right|_{cf}
		= \frac{\epsilon r}{h_\epsilon} \O(e^{-r}) = \O(1).
	\end{align*}
\end{proof}
For any $f, g \in C^{k, \alpha}_{cf} (M_{B,n})$, the product rule of H\"older norms implies
$f\cdot g \in C^{k, \alpha}_{cf} (M_{B,n})$ and $\|f \cdot g \|_{C^{k, \alpha}_{cf}(M_{B,n})} \le C \|f \|_{C^{k, \alpha}_{cf}(M_{B,n})} \cdot \| g \|_{C^{k, \alpha}_{cf}(M_{B,n})}$ for some uniform constant $C$. This implies that the wedge product  can be viewed as the bounded linear map
$$
\wedge \colon C^{k, \alpha}_{cf}(\Omega^{2}(M_{B,n})) \times C^{k, \alpha}_{cf}(\Omega^{2}(M_{B,n})) \to C^{k, \alpha}_{cf}(\Omega^{4}(M_{B,n})).
$$
With this version of the H\"older product rule, we can prove the non-linear condition for the inverse function theorem.

\begin{proposition}
	\label{prop:inverse-function-theory:non-linear-part}
	Let $N(\zeta) = 2 \Omega^{-2} \Lambda^{-1} \Tf ( \d\d^* \zeta \wedge \d \d^* \zeta)$.
	There exists a $q > 0$ of order $\O(\epsilon^{\delta-2})$ such that for any $\zeta, \xi \in (C^{k+2, \alpha}_{\delta}(M_{B,n}) \cdot \omega) \otimes \R^3$  (or $((C^{k+2, \alpha}_{\delta}(M_{B,n}) \oplus \R \phi)\cdot \omega ) \otimes \R^3$ when $B \not = \R^3$),
	$$
	\|N(\zeta) - N(\xi) \|_{(C^{k, \alpha}_{\delta}(M_{B,n}) \cdot \omega) \otimes \R^3} \le q \cdot \|\zeta + \xi\| \cdot \|\zeta - \xi\|,
	$$
	where $\| \zeta \pm \xi\|$ is measured with the $(C^{k+2, \alpha}_{\delta}(M_{B,n}) \cdot \omega) \otimes \R^3$ or $((C^{k+2, \alpha}_{\delta}(M_{B,n}) \oplus \R \phi)\cdot \omega) \otimes \R^3$ norm respectively.
\end{proposition}

\begin{proof}
	Using the `identity' $a^2 - b^2 = (a+b)(a-b)$, the expression of for $N(\zeta) - N(\xi)$ can be rewritten as
	$$
	N(\zeta) - N(\xi) = 2 \Omega^{-2} \Lambda^{-1} \Tf(\d \d^* (\zeta + \xi) \wedge \d \d^* (\zeta - \xi)).
	$$
	Write $\zeta \pm \xi$ in local coordinates
	$$
	\zeta \pm \xi = \sum_{ij} (u^\pm_{ij} + \lambda^\pm_{ij} \phi) \: \omega_i \otimes e_j.
	$$
	The expansion of $N(\zeta) - N(\xi)$, gives us the following three types of terms
	\begin{align*}
		N_1 :=& 2 \Omega^{-2} \Lambda^{-1} \Tf(\d \d^* (u_{ij}^+ \: \omega_i \otimes e_j) \wedge \d \d^* (u_{kl}^- \: \omega_k \otimes e_l)), \\
		N_2 :=& \lambda_{kl}^- \cdot 2 \Omega^{-2} \Lambda^{-1} \Tf(\d \d^* (u_{ij}^+ \: \omega_i \otimes e_j) \wedge \d \d^* (\phi \: \omega_k \otimes e_l)), \text{ and}\\
		N_3 :=& \lambda_{ij}^+ \lambda_{kl}^- \cdot 2 \Omega^{-2} \Lambda^{-1} \Tf(\d \d^* (\phi \: \omega_i \otimes e_j) \wedge \d \d^* (\phi \: \omega_k \otimes e_l)).
	\end{align*}
	Using Lemma \ref{lem:inverse-function-theorem:derivative-non-linear-part} and the product rule, $N_1$ can be estimated by
	$$
	N_1 = \O(2 e^{\delta \rho} \Omega^{-2} \Lambda^{-1} \Tf(  \Vol^{g_{cf}})) \cdot \|u_{ij}^+\|_{C^{k+2, \alpha}_{\delta}(M_{B,n})} \cdot \|u_{kl}^-\|_{C^{k+2, \alpha}_{\delta}(M_{B,n})}.
	$$
	Recall that the map $\Tf$ projects the space of 3 by 3 matrices to its symmetric traceless subspace. This projection is uniformly bounded, and hence

	$$
	N_1 = \O(2 e^{\delta \rho} \Omega^{-2} \Lambda^{-1} (  \Vol^{g_{cf}})) \cdot \|u_{ij}^+\|_{C^{k+2, \alpha}_{\delta}(M_{B,n})} \cdot \|u_{kl}^-\|_{C^{k+2, \alpha}_{\delta}(M_{B,n})}.
	$$
	Using Theorem \ref{thm:almost-hk:global-symplectic-triple} and that $g$ is \hk outside the gluing region, we estimate the inverse of $\Lambda$, which yields
	$$
	N_1 = \O(e^{\delta \rho} \Omega^{2}) \cdot \omega \cdot \|u_{ij}^+\|_{C^{k+2, \alpha}_{\delta}(M_{B,n})} \cdot \|u_{kl}^-\|_{C^{k+2, \alpha}_{\delta}(M_{B,n})}.
	$$
	We conclude that $q$ must be at least of order $\O(e^{\delta \rho} \Omega^2)$. 
	We calculate $\O(e^{\delta \rho} \Omega^2)$ explicitly for each region of $M_{B,n}$, which are given in Definition \ref{def:global-analysis:global-metric-and-functions}. We summarise the estimates in Table \ref{table:final-estimates-error}.
	\begin{table}[b]
		\centering
		\begin{tikzpicture}
			\draw[very thick] (0,0) -- (6.5,0);
			\draw[very thick,dashed] (6.5,0) -- (8,0);
			\draw[very thick,-{latex}] (8,0) -- (10.25,0) node[below]{$r \to \infty$};
			\draw (0,3pt) -- (0,-3pt) node[below]{$0$};
			\draw (3.5,3pt) -- (3.5,-3pt) node[below]{$4 \epsilon^{\frac{2}{5}}$};
			\draw (5,3pt) -- (5,-3pt) node[below]{$5 \epsilon^{\frac{2}{5}}$};
			\draw (2.5,-3pt) -- (2.5,3pt) node[above]{$r_{AH} = R_3$};
			\draw (6,-3pt) -- (6,3pt) node[above]{$r_j = R_2$};
			\draw (8.5,-3pt) -- (8.5,3pt) node[above]{$r = R_1$};
			\draw
			(0,0) node[midway,yshift=-5.7em,xshift=-2.3em]{\begin{tabular}{l}
					$\Omega$ \\
					$\rho$ \\
					$\O(e^{\delta \rho} \Omega^2)$
			\end{tabular}};
			\draw [decorate,decoration={brace,amplitude=5pt,mirror,raise=4ex}]
			(0,0) -- (2.4,0) node[midway,yshift=-5.7em]{\begin{tabular}{c}
					$\frac{\sqrt{1 + \epsilon \alpha_j}}{\epsilon}$ \\
					$\log \left(\frac{\epsilon}{1 + \epsilon \alpha_j} R_3\right)$ \\ 
					$\O(\epsilon^{\delta-2})$
			\end{tabular}};
			\draw [decorate,decoration={brace,amplitude=5pt,mirror,raise=4ex}]
			(2.6,0) -- (5.9,0) node[midway,yshift=-6em]{\begin{tabular}{c}
					$r_j^{-1} (h^{q_j}_\epsilon)^{- \frac{1}{2}}$ \\
					$\log(r_j)$ \\
					$\O(r_j^{\delta -2}) \le \O(\epsilon^{\delta -2})$
			\end{tabular}};
			\draw [decorate,decoration={brace,amplitude=5pt,mirror,raise=4ex}]
			(6.1,0) -- (8.4,0) node[midway,yshift=-5.7em]{\begin{tabular}{c}
					$1$ \\
					$1$ \\
					$\O(1)$
			\end{tabular}};
			\draw [decorate,decoration={brace,amplitude=5pt,mirror,raise=4ex}]
			(8.6,0) -- (10.75,0) node[midway,yshift=-5.7em]{\begin{tabular}{c}
					(Eq. \eqref{eq:def-omega-rho-near-infinity})\\ 
 					(Eq. \eqref{eq:def-omega-rho-near-infinity})\\ 
					$\O(1)$
			\end{tabular}};
			\draw (0,-2.5em) -- (0,-8em);
			\draw (-5em,-3em) -- (11,-3em);
			\draw (2.5,-2.5em) -- (2.5,-8em);
			\draw (6,-2.5em) -- (6,-8em);
			\draw (8.5,-2.5em) -- (8.5,-8em);
			
			\draw [decorate,decoration={brace,amplitude=5pt,raise=5ex}]
			(8,0) -- (11,0) node[midway,yshift=3.5em]{Asymptotic region};
			\draw [decorate,decoration={brace,amplitude=5pt,raise=5ex}]
			(0,0) -- (6.5,0) node[midway,yshift=3.5em]{Near singularity $q_j$};
		\end{tikzpicture}
		\caption{The estimate of $\O(e^{\delta \rho} \Omega^2)$ on the different regions of $M_{B,n}$.}
		\label{table:final-estimates-error}
	\end{table}
	The parameter $q$ attains its largest value inside the bubbles, and hence $q$ is at least of order $\O(\epsilon^{\delta - 2})$.

	We repeat this calculation for $N_2$ and $N_3$. Notice that for these cases, we only need to estimate the error on the asymptotic region of $M_{B,n}$, as $\phi = 0$ on the interior region. Terms like $N_2$ can be estimated by
	$$
	N_2 = \O(\Omega^{2}) \cdot \omega \cdot \|u_{ij}^+\|_{C^{k+2, \alpha}_{\delta}(M_{B,n})} \cdot |\lambda_{kl}^-|.
	$$
	By construction $\O(\Omega^2) = 1$ and so $q$ is still at least of order $\O(\epsilon^{\delta - 2})$. 
	
	Depending on the cases in Lemma \ref{lem:inverse-function-theorem:additional-estimates-non-linear-part}, $N_3$ can be estimated by
	$$
	N_3 = \begin{cases}
		\O(e^{- \delta \rho } \Omega^{2}) \cdot \omega \cdot |\lambda_{ij}^+| \cdot |\lambda_{kl}^-| & M_{B,n} \text{ is of type ALG/ALG*/ALH} \\
		\O(e^{(2- \delta) \rho } \Omega^{2}) \cdot \omega \cdot |\lambda_{ij}^+| \cdot |\lambda_{kl}^-| & M_{B,n} \text{ is of type ALH*}.
	\end{cases}
	$$
	In the former case $e^{- \delta \rho } = \frac{r^{- \delta}}{r^2 h_\epsilon}$, which is uniformly bounded. In the latter case, use that 
	$$
	h_\epsilon = 1 + \epsilon(A + B r + \O(e^{-r}))
	$$
	for some $A \in \R $ and $B > 0$ in order to get
	$$
	e^{(2 - \delta) \rho }\Omega^{2} 
	= \frac{r^{- \delta}}{h_\epsilon}
	= r^{-1-\delta} \frac{1}{\epsilon B} 
	\left(1 - \frac{1 + \epsilon(A + \O(e^{-r}))}{h_\epsilon}.\right)
	= \O(\epsilon^{-1}).
	$$
	For $\epsilon$ sufficiently small, $\O(\epsilon^{-1}) < \O(\epsilon^{\delta -2})$, which concludes the result.
\end{proof}

Recall that our goal is to find a zero for Equation \eqref{eq:def-fct-F} using the inverse function theorem given in Theorem \ref{thm:hk-triple:inverse-function-theorem}.
According to Proposition \ref{prop:inverse-function-theorem:constant-part}, the constant part $F(0)$ of this equation is of order $\O(\epsilon^{\frac{11 - 2 \delta}{5}})$. In Section \ref{subsec:linearised-eq} we have shown that the 
linearised operator is invertible with uniformly bounded inverse. Proposition \ref{prop:inverse-function-theory:non-linear-part} implies that the non-linear part satisfies
$$
\|N(\zeta) - N(\xi) \|_{(C^{k, \alpha}_{\delta}(M_{B,n}) \cdot \omega) \otimes \R^3} \le \O(\epsilon^{\delta -2}) \cdot \|\zeta + \xi\| \cdot \|\zeta - \xi\|.
$$ 
Therefore, Theorem \ref{thm:hk-triple:inverse-function-theorem} can be applied if
$\epsilon^{\frac{11 - 2 \delta}{5}} \le \O(\epsilon^{2-\delta}).$
This is indeed true for sufficiently small $\epsilon$, and hence:
\begin{proposition}
	\label{prop:conclusion-inverse-function-theorem}
	For sufficiently small $\epsilon > 0$, there exists a triple of $\zeta_i \in C^{k, \alpha}_\delta(M_{B,n}) \cdot \omega$ (or $\zeta \in (C^{k, \alpha}_\delta(M_{B,n}) \oplus \R \phi) \cdot \omega$ when $B \not = \R^3$), such that 
	$$
	\omega_i + 2 \d \d^* \zeta_i
	$$
	is an orthonormal triple of closed $2$-forms, and the norm of $\zeta_i$ is of order $\O(\epsilon^{\frac{11 - 2 \delta}{5}})$.
\end{proposition}

\begin{proof}[Proof of Theorem \ref{main-theorem}]
	Given the data in the theorem, we constructed in Section \ref{sec:approximate-solution} a 4-manifold $M_{B,n}$ and a 1-parameter family of closed definite triples $\omega$ that are approximately \hk. By Theorem \ref{prop:conclusion-inverse-function-theorem}, there exists 
	a triple of $\zeta_i \in C^{k, \alpha}_\delta(M_{B,n}) \cdot \omega$ (or $\zeta \in (C^{k, \alpha}_\delta(M_{B,n}) \oplus \R \phi) \cdot \omega$ when $B \not = \R^3$), such that 
	$\omega_i + 2 \d \d^* \zeta_i$ is an orthonormal triple of closed $2$-forms. Using an elliptic bootstrap argument, one can show that $\omega_i + 2 \d \d^* \zeta_i$ is smooth and thus induces a \hk metric on $M_{B,n}$. Moreover, our genuine gravitational instanton differs from our approximate solution with an error of $\O(\epsilon^{\frac{11 - 2 \delta}{5}})$. Hence, for sufficiently small $\epsilon$, properties 1 to 3 of Theorem \ref{main-theorem} are satisfied.
\end{proof}

\section{Global properties of $M_{B,n}$}
\label{sec:setup:topology}
Before we finish this paper, we study the topology of our constructed \hk manifold and we compare the number of parameters in our construction to the dimension of the respective moduli spaces. By this we prove the remaining Propositions and Theorems stated in Section \ref{sec:results}.

\subsection{The topology of $M_{B,n}$}
In order to understand the topology of $M_{B,n}$ and its intersection form, we first revisit the topology of the multi-Taub-NUT space and the work of \mycite{Sen1997}. Namely, consider the multi-Taub-NUT space that has the ordered set of points $\{p_1, \ldots, p_n\} \subset \R^3$ as its singularities. On $\R^3$ one can find a path that goes through each point $p_i$ once. The Taub-NUT space retracts to the total space over this path, and using the explicit nature of the metric, one can show that the total space over this path is a chain of wedge sums of $n-1$ copies of $S^2$. According to \mycite{Sen1997}, the intersection matrix for these spheres is the negative Cartan matrix of a $A_{n-1}$-type Dynkin diagram.\\

\begin{figure}[h!]
	\centering
	\def\svgwidth{\linewidth}
	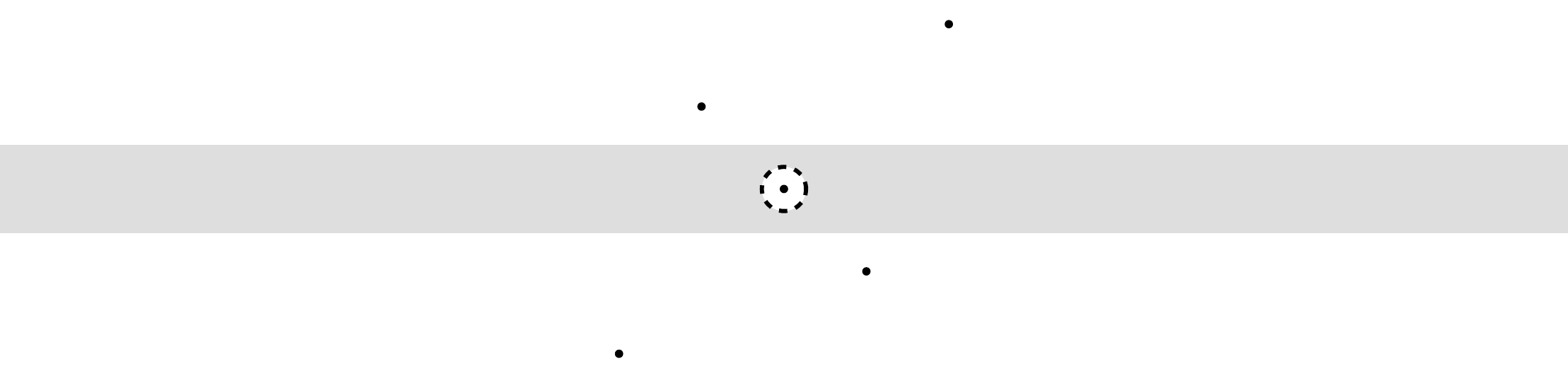
	\caption{{The underlying manifold $M_{\R^3, n}$ can be seen as the union of the Atiyah-Hitchin manifold and the Multi-Taub-NUT space.}}
	\label{fig:setup:mayer-vietoris-r3}
\end{figure}

Similarly, \mycite{Sen1997} argued that the intersection matrix for $M_{\R^3, n}$ is the negative Cartan matrix for a $D_n$ Dynkin diagram. In order to extend his argument to the other gravitational instantons, we first need to understand the homology groups of $M_{\R^3,n}$: 

\begin{lemma}
	The homology of $M_{\R^3, n}$ is given by
	$$
	H_k(M_{\R^3, n}) = \begin{cases}
	\Z &\text{if } k=0 \\
	\Z_2 &\text{if } k=1 \text{ and } n = 0 \\
	\Z^{n} &\text{if } k=2 \\
	0 & \text{otherwise}.
	\end{cases}
	$$
\end{lemma}
\begin{proof}
Decompose the underlying base-space $\R^3$ into two regions as shown in Figure \ref{fig:setup:mayer-vietoris-r3}, and apply the Mayer-Vietoris sequence. Namely, let $X := (-\delta, \delta) \times \R^2$ be a thin plate inside $B$ that does not contain any of the non-fixed singularities $\pm p_i$. Using rotation, $X$ can always be found. The complement of the base space $B'$ and the plate $X$ has two connected components, which can be identified using the antipodal map. Denote one of these connected components as $Y$. The antipodal map sends $X$ onto itself and therefore the bulk space $P/\Z_2$ can be written as
$$
P/\Z_2 = (P|_{X})/\Z_2 \cup P|_{Y}.
$$
From the gluing construction we identify $M_{\R^3,n}$ with $\widetilde{P|_{X}/\Z_2 }\cup \widetilde{P|_{Y}}$, where $\widetilde{P|_{X}/\Z_2 }$ is the connected sum of $P|_{X}/\Z_2$ with the Atiyah-Hitchin manifold and $\widetilde{P|_{Y}}$ is the connected sum of $P|_{Y}$ with $n$ copies of Taub-NUT. The space $P|_{X}/\Z_2$ retracts to its boundary at the origin, which, after the connected sum construction, will be identified with the asymptotic region of the Atiyah-Hitchin manifold. Because the Atiyah-Hitchin manifold retracts to $\R P^2$, $\widetilde{P|_{X}/\Z_2}$ must also retract to $\R P^2$. 

Like for the multi-Taub-NUT space, $\widetilde{P|_{Y}}$ is homotopic to the wedge sum of $n-1$ copies of $S^2$. 
In order to apply the Mayer-Vietoris sequence we need to calculate $ \widetilde{P|_{X}/\Z_2 } \cap \widetilde{P|_{Y}}$. Because the two connected components $\{\pm \delta\} \times \R^2$ of the boundary of $X$ are identified by the antipodal map, $\widetilde{P|_{X}/\Z_2 } \cap \widetilde{P|_{Y}}$ is diffeomorphic to a circle bundle over $\R^2$. Therefore, $\tilde{H}_k(M_{\R^3, n})$ is given by the following exact sequence:
\\

\centerline{
	\xymatrix{
		\ldots  & 
		\tilde{H}_k(M_{\R^3, n}) \ar[l] & 
		\tilde{H}_k(\R P^2) \oplus \tilde{H}_k(\bigvee_{i=1}^{n-1} S^2) \ar[l] & 
		\tilde{H}_k(S^1) \ar[l] & 
		\ldots \ar[l] 
}}

\vspace{0.5cm}

The only non-trivial step in this sequence is the map $\del \colon \tilde{H}_1(S^1) \to \tilde{H}_1(\R P^2)$. This map is the embedding of a fibre over a point into $\widetilde{P|_X/\Z_2}$. As explained in \cite{Schroers2020}, this fibre is homotopic to the generator of $H_1$ inside the Atiyah-Hitchin manifold and so $\del(1) = [1]$. With this in mind, one can show that the homology groups of $M_{\R^3, n}$ are
$$
H_k(M_{\R^3, n}) = \begin{cases}
	\Z &\text{if } k=0 \\
	\Z_2 &\text{if } k=1 \text{ and } n = 0 \\
	\Z^{n} &\text{if } k=2 \\
	0 & \text{otherwise}.
\end{cases}
$$	
\end{proof}
\begin{figure}[h!]
	\centering
	\includegraphics[width=0.9\linewidth]{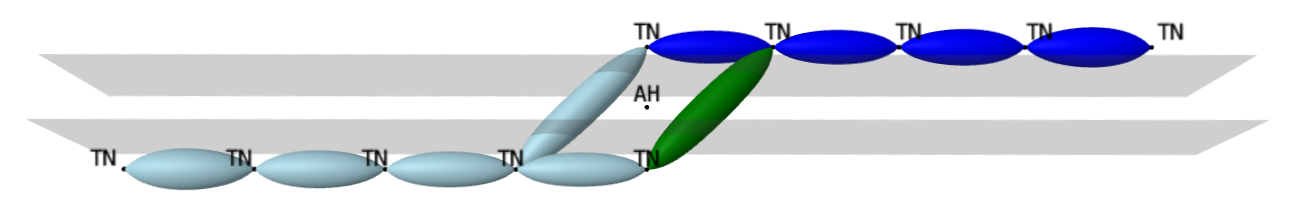}
	\caption{{
			Depiction of the $2$-cycles with self-intersection $-2$ inside $P|_{X} \cup \widetilde{P|_{Y}} \cup \: \Z_2 \cdot \widetilde{P|_{Y}}$.
			The grey planes depict the boundary between these regions. 
			The dark-blue and green spheres form a basis of $H_2(M_{\R^3,5})$ such that its intersection matrix is the negative Cartan matrix of $D_5$.	
			The light-blue spheres are the $\Z_2$ images of the other spheres.
	 }}
	\label{fig:topology:d5-alf}
\end{figure}

In order to show that the intersection matrix for can be given as the negative Cartan matrix for a $D_n$ Dynkin diagram, \mycite{Sen1997} started with the generators of $H_2(\widetilde{P|_Y})$ such that the intersection matrix is given by the negative Cartan matrix of $A_{n-1}$. In Figure \ref{fig:topology:d5-alf} these are depicted by dark-blue spheres. Inside $H_2(M_{\R^3, n})$, there is one extra generator and is given by a $2$-cycle that intersects the boundary between $\widetilde{P|_{X}/\Z_2}$ and $\widetilde{P|_{Y}}$ by two generators of $H_1(S^1)$.
According to \mycite{Sen1997}, this extra generator can be represented by a $2$-sphere with self-intersection $-2$ and this sphere can be similarly constructed as one of the spheres inside the multi-Taub-NUT. In Figure \ref{fig:topology:d5-alf} this extra representative is coloured green. 
Taking account of the orientation, \mycite{Sen1997} showed that for this basis of two-spheres the intersection matrix is the negative Cartan matrix of type $D_n$.

Although Sen only considered the case $n \ge 4$, his argument still works for the cases $2 \le n \le 3$. We show this explicitly in Table \ref{table:topology:Dn-alf}. Namely, we can still find the (dark-blue and green) spheres in $M_{\R^3, n}$ that generate $H_2(M_{\R^3,n})$. These spheres have self-intersection $-2$. Neighbouring spheres still intersect with multiplicity $\pm 1$ depending on the following rule found by \cite{Sen1997}: If the dark-blue and green spheres intersect, they give a factor of $+1$ to the intersection matrix. If the dark-blue and green spheres intersect a light-blue sphere, they induce a $-1$ factor to the intersection matrix.
Adding all the contributions yields the second column of Table \ref{table:topology:Dn-alf}. From these intersection matrices we can still find a corresponding Dynkin diagrams, which are given in the third and fourth column.

The case $n=1$ is special. We know that $H_2(M_{\R^3, n})$ has one generator $S$. In Table \ref{table:topology:Dn-alf}, $S$ is depicted as the green sphere. To find the self-intersection of $S$, we embed it inside $M_{\R^2, 2}$, which has $H_2(M_{\R^2, 2}) = \Z^2$. By comparing the pictures in Table \ref{table:topology:Dn-alf}, we see that $S$ is homologous to $(1,1) \in \Z^2 = H_2(M_{\R^2, 2})$. Using the intersection matrix for $M_{\R^2, 2}$, we conclude $S$ has self-intersection $-4$.

\begin{table}[h!]
	\centering
\begin{tabular}{c||c|c|c|c|}
	$n$ & $2$-cycles & Intersection matrix & Diagram & Type \\
	\hline \hline
	4 & 
	\adjustbox{valign=m}{
	\includegraphics[width=0.30\linewidth,trim={4.5cm 2.5cm 4.5cm 2.5cm},clip]{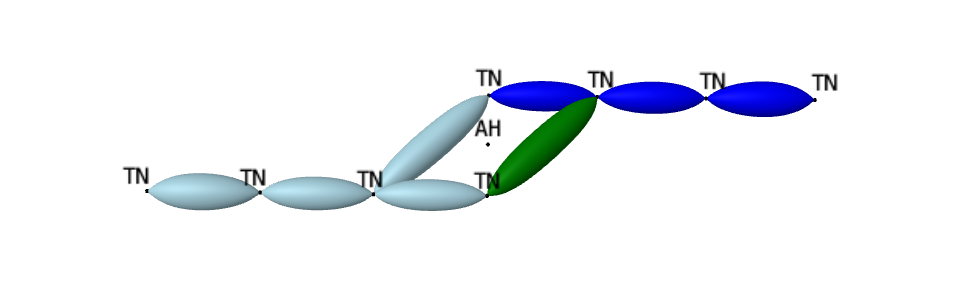}} &
	\adjustbox{valign=m, margin=0 0.2cm}{
	$\begin{pmatrix}
		-2 & 1  \\ 
		1 & -2 & 1 & 1\\
		& 1 & -2 & \\
		& 1 & & -2
	\end{pmatrix}$}&
	\adjustbox{valign=m}{\begin{tikzpicture}
			\draw[black] (0,0) -- (1,0);
			\draw[black] (1,0) -- (1.8,0.8);
			\draw[black] (1,0) -- (1.8,-0.8);
			\filldraw[black] (0,0) circle (2pt);
			\filldraw[black] (1,0) circle (2pt);
			\filldraw[black] (1.8,-0.8) circle (2pt);
			\filldraw[black] (1.8,0.8) circle (2pt);
	\end{tikzpicture}}
	&
    $D_4$ \\
	\hline
3 & 
\adjustbox{valign=m}{
	\includegraphics[width=0.30\linewidth,trim={4.5cm 2.5cm 4.5cm 2.5cm},clip]{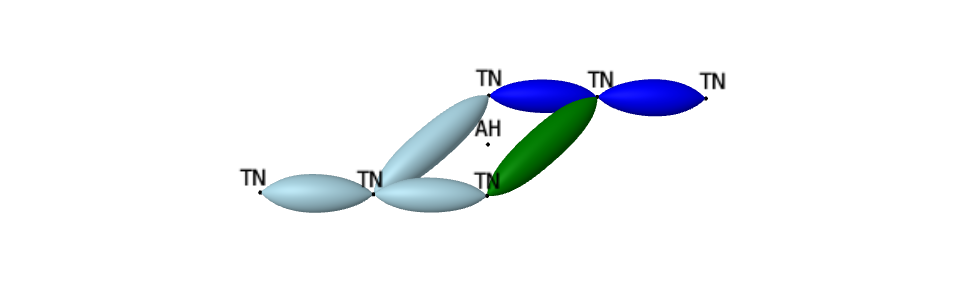}} &
\adjustbox{valign=m,margin=0 0.2cm}{
	$\begin{pmatrix}
		-2 & 1 & 1 \\ 
		1 & -2 & 0 \\
		1 & 0 & -2 & \\
	\end{pmatrix}$}&
\adjustbox{valign=m}{\begin{tikzpicture}
		\draw[black] (1,0) -- (2,0);
		\draw[black] (1,1) -- (1,0);
		\filldraw[black] (1,0) circle (2pt);
		\filldraw[black] (2,0) circle (2pt);
		\filldraw[black] (1,1) circle (2pt);
\end{tikzpicture}}
&
$A_3$ \\
\hline
2 & 
\adjustbox{valign=m}{
	\includegraphics[width=0.30\linewidth,trim={4.5cm 2.5cm 4.5cm 2.5cm},clip]{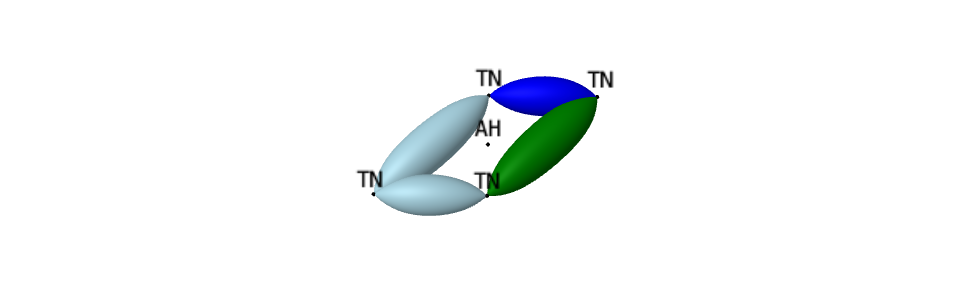}} &
\adjustbox{valign=m,margin=0 0.2cm}{
	$\begin{pmatrix}
		-2 & 0 \\
		0 & -2 & \\
	\end{pmatrix}$}&
\adjustbox{valign=m}{\begin{tikzpicture}
		\filldraw[black] (1,0) circle (2pt);
		\filldraw[black] (2,0) circle (2pt);
\end{tikzpicture}}
&
$A_1 + A_1$ \\
\hline
1 & 
\adjustbox{valign=m,margin=0 0.2cm}{
	\includegraphics[width=0.30\linewidth,trim={4.5cm 0.5cm 4.5cm 0.5cm},clip]{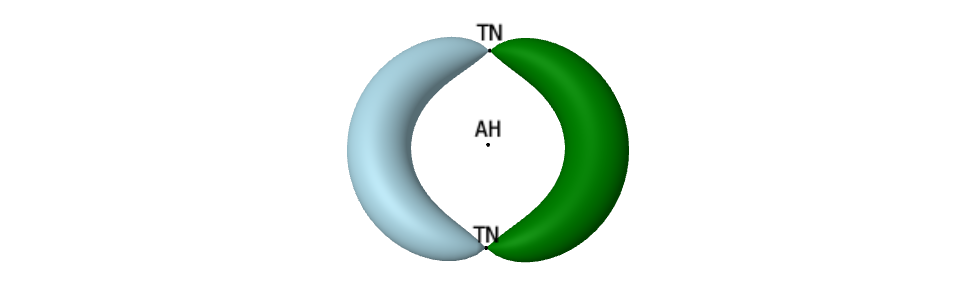}} &
\adjustbox{valign=m}{
	$\begin{pmatrix}
		-4 
	\end{pmatrix}$} & N/A & N/A. \\
\hline
\end{tabular}
		\caption{{Intersection matrices for $M_{\R^3,n}$ with $n \le 4$.}}
	\label{table:topology:Dn-alf}
\end{table}

\begin{proof}[Proof of Proposition \ref{prop:topology:R2xS1}]
	Let $X := (\R^2 \times (-\delta, \delta)) \cap B' \subseteq \R^2 \times S^1$ be a thin plate that does not contain any of the non-fixed singularities $\pm p_i$, but only contains one of the fixed-point singularities $q_j$. Again denote $Y$ as the complement of $X$ inside $B'$. Different to the $B = \R^3$ case, both $X$ and $Y$ are $\Z_2$ invariant under the antipodal map, and hence the bulk space $P/\Z_2$ can be written as
	$$
	P/\Z_2 = P|_{X}/\Z_2 \cup P|_{Y}/\Z_2.
	$$
	Like before, we consider $\widetilde{P|_{X}/\Z_2}$ and $\widetilde{P|_{Y}/\Z_2}$ as the completions of $P|_{X}/\Z_2$ and $P|_{Y}/\Z_2$.
	From the construction of $M_{\R^3, n}$, we identify the topology of $\widetilde{P|_{X}/\Z_2}$ and $\widetilde{P|_{Y}/\Z_2}$ with the topology of $M_{\R^3, 0}$ and $M_{\R^3, n}$ respectively. The intersection $\widetilde{P|_{X}/\Z_2} \cap \widetilde{P|_{Y}/\Z_2}$ must be an $S^1$-bundle over a plane, which retracts to a circle. Therefore, $\tilde{H}_k(M_{\R^3, n})$ is given by the 	following exact sequence:\\
	\centerline{
		\xymatrix{
			\ldots  & 
			\tilde{H}_k(M_{\R^2 \times S^1, n}) \ar[l] & 
			\tilde{H}_k(M_{\R^3, 0}) \oplus \tilde{H}_k(M_{\R^3, n}) \ar[l] & 
			\tilde{H}_k(S^1) \ar[l] & 
			\ldots \ar[l] 
	}}
	\vspace{0.5cm}
	
	The only non-trivial step in this sequence is the map $\del \colon \tilde{H}_1(S^1) \to \tilde{H}_1(M_{\R^3, 0}) \oplus \tilde{H}_1(M_{\R^3, n})$.
	This map again sends a fibre over a point into the generator of $H_1$ inside the Atiyah-Hitchin manifold. 
	We conclude that 
	$$
	\del(1) = \begin{cases}
		([1] , [1])  &\text{if } n = 0 \\
		[1] &\text{if } n \not = 0,
	\end{cases}
	$$
	and with this one can calculate the homology groups of $M_{\R^2 \times S^1, n}$ explicitly.
	
	Using Sen's method, one can construct the generators of $H_2(M_{\R^2 \times S^1,n})$ and calculate their intersection matrix. In Table \ref{table:topology:Dn-alg} these generators are explicitly given.
\end{proof}
\begin{table}[h!]
	\centering
	\begin{tabular}{c||c|c|c|c|}
		$n$ & $2$-cycles & Intersection matrix & Diagram & Type \\
		\hline \hline
		4 & 
		\adjustbox{valign=m}{
			\includegraphics[width=0.25\linewidth,trim={2.5cm 2.5cm 2.5cm 2.5cm},clip]{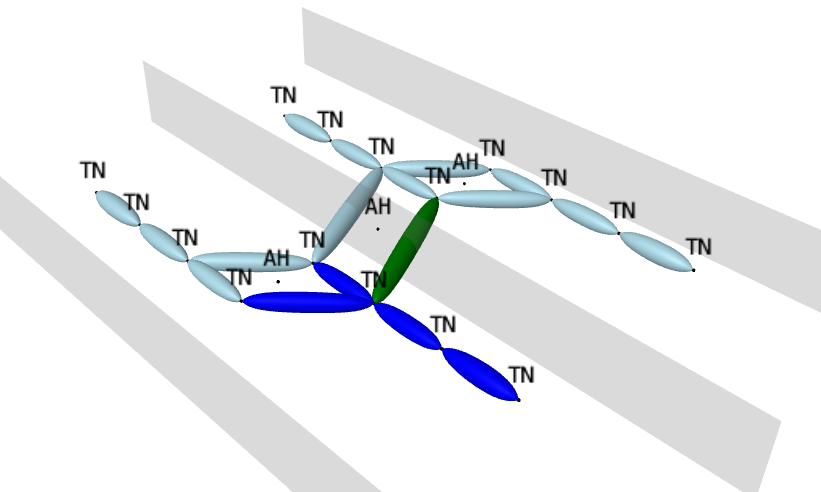}} &
		\adjustbox{valign=m, margin=0 0.2cm}{
			$\begin{pmatrix}
				-2 & 1  \\ 
				1 & -2 & 1 & 1 & 1\\
				& 1 & -2 & \\
				& 1 & & -2 \\
				& 1 & & & -2
			\end{pmatrix}$}&
		\adjustbox{valign=m, margin=0 0.2cm}{\begin{tikzpicture}
				\draw[black] (0,0) -- (2,0);
				\draw[black] (1,1) -- (1,-1);
				\filldraw[black] (0,0) circle (2pt);
				\filldraw[black] (1,0) circle (2pt);
				\filldraw[black] (2,0) circle (2pt);
				\filldraw[black] (1,1) circle (2pt);
				\filldraw[black] (1,-1) circle (2pt);
		\end{tikzpicture}}
		&
		$\tilde{D}_4$ \\
		\hline
		3 & 
		\adjustbox{valign=m, margin=0 0.2cm}{
			\includegraphics[width=0.25\linewidth,trim={2.5cm 2.5cm 2.5cm 2.5cm},clip]{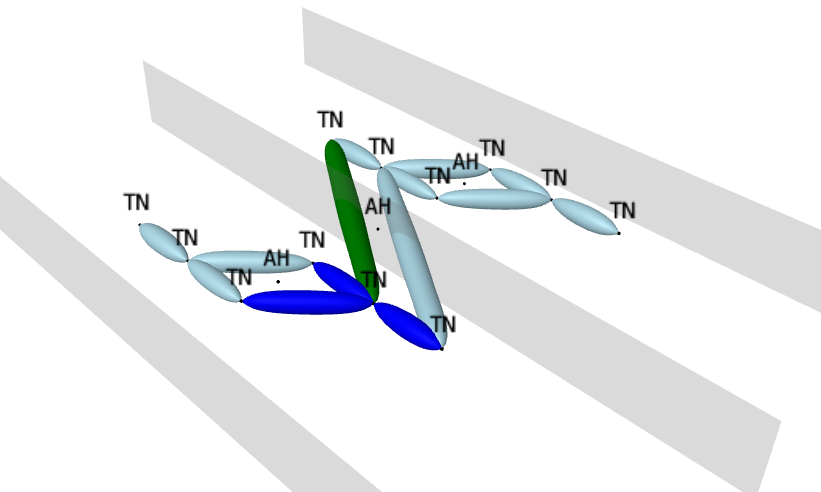}} &
		\adjustbox{valign=m,margin=0 0.2cm}{
			$\begin{pmatrix}
				-2 & 1 & 1 & 0 \\ 
				1 & -2 & 0 & 1\\
				1 & 0 & -2 & 1 \\
				0 & 1 & 1 & -2
			\end{pmatrix}$}&
		\adjustbox{valign=m}{\begin{tikzpicture}
				\draw[black] (1,0) -- (2,0);
				\draw[black] (1,0) -- (1,1);
				\draw[black] (2,1) -- (2,0);
				\draw[black] (2,1) -- (1,1);
				\filldraw[black] (1,0) circle (2pt);
				\filldraw[black] (2,0) circle (2pt);
				\filldraw[black] (1,1) circle (2pt);
				\filldraw[black] (2,1) circle (2pt);
		\end{tikzpicture}}
		&
		$\tilde{A}_3$ \\
		\hline
		2 & 
		\adjustbox{valign=m, margin=0 0.2cm}{
			\includegraphics[width=0.25\linewidth,trim={2.5cm 2.5cm 2.5cm 2.5cm},clip]{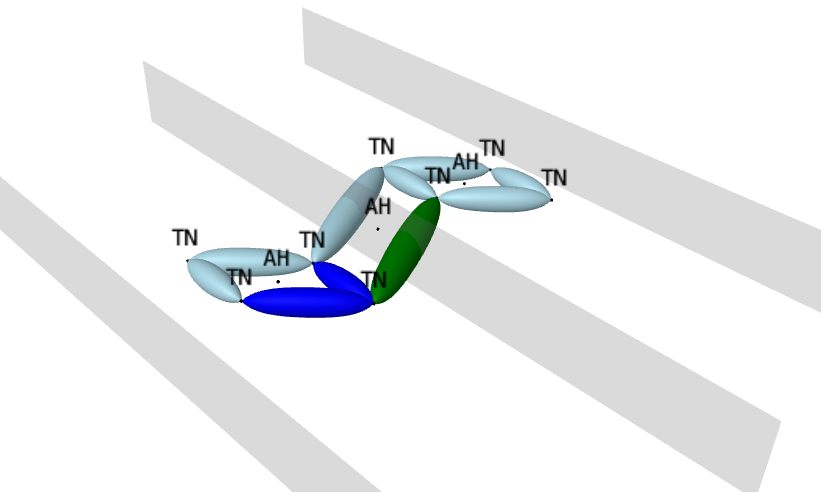}} &
		\adjustbox{valign=m,margin=0 0.2cm}{
			$\begin{pmatrix}
				-2 & 0 & 0\\
				0 & -2 & 2\\ 
				0 & 2 & -2 
			\end{pmatrix}$}&
		\adjustbox{valign=m}{\begin{tikzpicture}
				\filldraw[black] (0.5,1) circle (2pt);
				\filldraw[black] (0,0) circle (2pt);
				\filldraw[black] (1,0) circle (2pt);
				\draw[black] (0.5,0) ellipse (0.5 and 0.2);
		\end{tikzpicture}}
		&
		$A_1 + \tilde{A}_1$ \\
		\hline
		1 & 
		\adjustbox{valign=m,margin=0 0.2cm}{
			\includegraphics[width=0.25\linewidth,trim={2.5cm 2.5cm 2.5cm 2.5cm},clip]{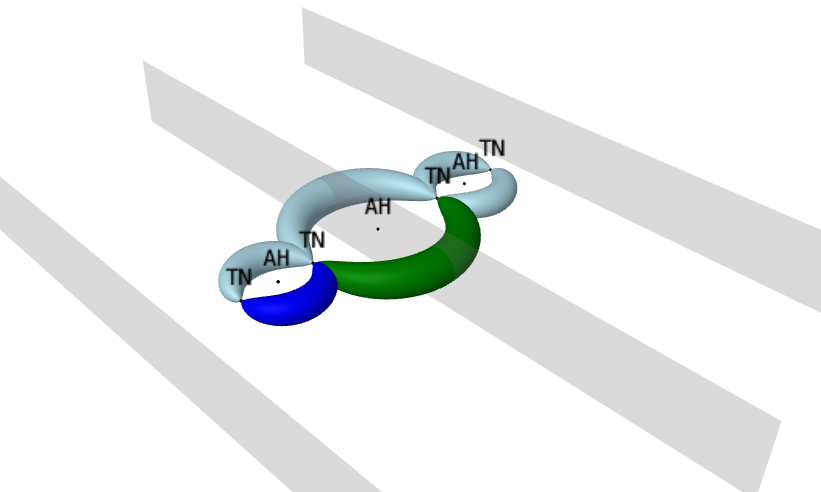}} &
		\adjustbox{valign=m}{
			$\begin{pmatrix}
				-4 & 4 \\
				4 & -4
			\end{pmatrix}$} & N/A & N/A. \\
		\hline
	\end{tabular}
	\caption{{Intersection matrices for $M_{\R^2 \times S^1,n}$. The dark-blue spheres are the generators of $H_2(M_{\R^3,n})$ given in Table \ref{table:topology:Dn-alf}. The green spheres are the extra $2$-cycles that are induced by the kernel of $\del \colon H_1(S^1) \to H_1(M_{\R^3, 0}) \oplus H_1(M_{\R^3, n})$. The light-blue spheres are the images of the dark-blue and green spheres under the antipodal map and the action of the lattice on $\R^3$. The gray planes depict the boundary of the fundamental domain of $\R^2 \times S^1$ inside its universal cover.
	}}
	\label{table:topology:Dn-alg}
\end{table}
%
%
\begin{proof}[Proof of Proposition \ref{prop:topology:RxT2}]
	The argument is identical to the argument given in Proposition \ref{prop:topology:R2xS1}, except for the fact that we view $M_{\R \times T^2, n}$ as the union of $M_{\R^2 \times S^1, 0}$ and $M_{\R^2 \times S^1, n}$ along an $S^1$-bundle over $ \R \times S^1$.
\end{proof}

\subsection{The moduli space}
We compare $M_{B, n}$ with the known classifications of gravitational instantons: ALE, ALF, ALG, ALG*, ALH and ALH*. In Section 6.4 of \cite{Sun2021}, there is an overview of all these classes including an explicit model at infinity for each of them.
Comparing these with the metric $g^{GH}$ and using the estimates found in Lemma \ref{lem:setup:harmonic-function}, we conclude that $M_{B,n}$ is of the following type:
\begin{table}[th!]
	\centering
	\begin{tabular}{llll}
		 &  & Definition in & \\
		Space & Class & \cite{Sun2021} & Remarks \\
		\hline
		$M_{\R^3,n}$ & ALF-$D_{n}$ & Definition 6.8(2) &  \\
		$M_{\R^2 \times S^1,n}$ & ALG*-$I^*_{4-n}$ & Definition 6.12(2) & $0 \le n < 4$ \\
		$M_{\R^2 \times S^1,4}$ & ALG$_{\frac{1}{2}}$ & Definition 6.11 &  \\
		$M_{\R \times T^2,n}$ & ALH*-$I_{8-n}$ & Definition 6.16(1) & $0 \le n < 8$ \\
		$M_{\R \times T^2,8}$ & ALH & Definition 6.15 &\\
	\end{tabular}
	\caption{{Classification of $M_{B,n}$ using the descriptions given in \cite{Sun2021}.}}
	\label{table:classification}
\end{table}

\begin{remark}
	The suffix $D_n$ in ALF-$D_n$ is not arbitrary: According to \mycite{Chen2019} Remark 6.3, gravitational instantons of type ALF-$D_n$ have an intersection matrix related to the $D_n$ Dynkin diagram. This is exactly what \mycite{Sen1997} found.
\end{remark}
\begin{remark}
	To understand the suffix for ALG* and ALH* manifolds, we refer to the work of \mycite{Chen2021ALG} and \mycite{Collins2020}. Namely, they showed that these gravitational instantons, which are constructed by \cite{Hein2010a} and \cite{Tian1990}, can be compactified by adding a singular Kodaira fibre of type $I^*_k$ or $I_k$ respectively.
\end{remark}
\begin{remark}
	Near infinity, ALG gravitational instantons approximate the metric of a flat torus bundle over a $2$-dimensional cone. The suffix in the ALG case is, up to a factor of $2 \pi$, the cone angle. For example, the cone angle for an ALG$_\frac{1}{2}$ manifold is $\frac{1}{2} \cdot 2 \pi$.
\end{remark}

\mycite{Chen2018} found a Torelli theorem for ALF-$D_n$ type gravitational instantons: Up to triholomorphic isometries all ALF-$D_n$ type gravitational instantons can be uniquely classified by their model at infinity and their periods. For these ALF spaces the model at infinity is fully determined by the degree of the circle bundle at infinity and the length of its fibre at infinity. These parameters correspond to the number of non-fixed singularities $p_i$ and $\epsilon$ respectively. To calculate the period one has to integrate the \hk triple over a basis of $H_2(M_{\R^3, n})$ where each element has self-intersection $-2$. There are 3 K\"ahler triples and the dimension of $H_2(M_{\R^3, n})$ is $n$. Hence, the moduli space of ALF metrics with a fixed model space is $3n$-dimensional. This number corresponds with the $n$ possible positions of the nuts in $\R^3$.

In \cite{Chen2021ALG} there is a Torelli Theorem for ALG* gravitational instantons. Here, the model at infinity is determined by the lattice and a global scale. Up to rotation, a one-dimensional lattice is only determined by the length of its generator and so the model at infinity is determined by two parameters.

In their paper they argue that the period map over count the moduli space. 
This is because $H_2(M_{\R^2\times S^1, n})$ has a $2$-cycle that is represented by a torus at infinity. This $2$-cycle can only reveal information of the model space, which is fixed. Therefore, the dimension of the moduli space of ALG* gravitational instantons with fixed model at infinity is $3 (\beta_2 -1)$, where $\beta_2$ is the second Betti number. Using Proposition \ref{prop:topology:R2xS1} we see that the dimension is $3n$. Again this corresponds to the $n$ possible positions of the non-fixed singularities in $\R^2 \times S^1$.

\mycite{Chen2021ALG} also found a Torelli theorem for ALG-type gravitational instantons. For this case, the model metric is determined by the length of the circle in the base space, the size of the circle fibre and the choice of connection. The space of connections is determined by $H^1(B, \R) / H^1(B, \Z)$, which is 1-dimensional for $B = \R^2 \times S^1$. Hence, the model metric is determined by three parameters. With the model metric fixed, \mycite{Chen2021ALG} argued that the dimension of the moduli space is $3 (\beta_2 -1) = 12$. Again we expect this, because we have $12$ degrees of freedom in choosing the location of the nuts.



According to \mycite{Hein2021} any ALH* gravitational instanton arises from the generalized Tian-Yau construction on the complement of a smooth anticanonical divisor of some weak del Pezzo surface. Given the degree of the anticanonical divisor, on can relate the Tian-Yau construction with $M_{\R \times T^2, n}$. Indeed, according to \mycite{Collins2021} and \mycite{Hein2021}, the complement of a del Pezzo surface of a smooth anticanonical divisor $D$ with $D^2 = d$ can be compactified to a rational elliptic surface by adding an $I_d$ fibre after performing a \hk rotation.
Using the classification described in Table \ref{table:classification}, we conclude that $M_{\R \times T^2,n}$ can only be compactified into a weak del Pezzo surface with anticanonical divisor of degree $8-n$ for $0 \le n < 8$.

Up to diffeomorphism there are 10 different weak del Pezzo surfaces: $\C P^2$, the blow-up of $\C P^2$ at up to 8 points and $S^2 \times S^2$. The degree of the anti-canonical divisor is $9-k$ for $\operatorname{Bl}_k \C P^2$ and 8 for $S^2 \times S^2$. As $M_{\R \times T^2, n}$ can never be compactified with an anti-canonical devisor of degree 9, we conclude

\begin{proposition}
	The gravitational instanton $\C P^2 \setminus D$, where $D$ is the anti-canonical divisor of degree 9 cannot be constructed using Theorem \ref{main-theorem}.
\end{proposition}

For $1 \le k \le 7$, there is a unique del Pezzo surface with anticanonical divisor of degree $k$. From this we immediately conclude part 1 of Proposition \ref{main-proposition-3}.

Up to diffeomorphism, there are two del Pezzo surfaces of degree 8, namely $S^2 \times S^2$ and $\operatorname{Bl}_1 \C P^2$. We claim that $\operatorname{Bl}_1 \C P^2$ cannot be used to construct $M_{\R \times T^2, 0}$.
\begin{proposition}
	The space $M_{\R \times T^2, 0}$ is not diffeomorphic to
the complement of a smooth anticanonical divisor of the blowup of $\C P^2$ at one point.
\end{proposition}

\begin{proof}
	Assume the opposite. Then $M_{\R \times T^2, 0}$ can be compactified to the blowup of $\C P^2$ by gluing a disk bundle $D$ at infinity. The boundary $\del D$ is an $S^1$-bundle over $T^2$ of degree 8. These identifications yields the following Mayer-Vietoris sequence:
	$$
	\ldots 
	\to \tilde{H}_k(\del D) 
	\xrightarrow{\alpha_k} \tilde{H}_k(D) \oplus \tilde{H}_k(M_{\R \times T^2, 0}) 
	\xrightarrow{\beta_k} \tilde{H}_k(\operatorname{Bl}_1 \C P^2) 
	\to \ldots
	$$
	with
	$$
	\tilde{H}_k(\del D) = \begin{cases}
		\Z^2 \oplus \Z_8 &\text{if } k=1\\
		\Z^{2} &\text{if } k=2 \\
		\Z &\text{if } k=3 \\
		0 & \text{otherwise},
	\end{cases}
	\qquad
	\tilde{H}_k(D) = \begin{cases}
		\Z^2 &\text{if } k=1 \\
		\Z &\text{if } k=2 \\
		0 &\text{if } k=3 \\
		0 & \text{otherwise},
	\end{cases}
	$$
	$$
	\tilde{H}_k(M_{\R \times T^2, 0}) = \begin{cases}
		\Z_2 &\text{if } k=1\\
		\Z^{3} &\text{if } k=2 \\
		0 & \text{otherwise},
	\end{cases}
	\qquad
	\tilde{H}_k(\operatorname{Bl}_1 \C P^2) = \begin{cases}
		\Z^2 &\text{if } k=2\\
		\Z &\text{if } k=4 \\
		0 & \text{otherwise}.
	\end{cases}
	$$
	Let $\del \colon H_2(Bl_1 \C P^2) \to H_1(\del D)$ be the boundary map. 
	Our goal is to reach a contradiction by showing that 
	\begin{equation}
		\label{eq:short-exact-sequence-boundary-map}
		0 \to \ker \del \to H_2(Bl_1 \C P^2) \xrightarrow{\del} \im \del \to 0
	\end{equation}
	cannot be exact.

	Consider the map $\alpha_1\colon H_1(\del D) \to H_1(D) \oplus H_1(M_{\R \times T^2, 0})$ from the Mayer-Vietoris sequence. We first study $\im \del = \ker \alpha_1$. 
	From the Gysin sequence it follows that the free part of $H_1(\del D)$ is generated by the homology of the base space of $D$. Therefore the map $\alpha_1$ is of the form
	$$
	\alpha_1\colon \Z^2 \oplus \Z_8 \to \Z^2 \oplus \Z_2 \qquad (x,0) \mapsto (x, \ldots).
	$$
	The Mayer-Vietoris sequence also implies $\alpha_1$ is surjective and so $\alpha_1(0,1) = (0,1)$. This concludes $\im \del = \ker \alpha$ is isomorphic to $\Z_4$.

	Secondly, consider the maps $\alpha_2\colon H_2(\del D) \to H_2(D) \oplus H_2(M_{\R \times T^2, 0})$ and $\beta_2 \colon H_2(D) \oplus H_2(M_{\R \times T^2, 0}) \to H_2(Bl_1 \C P^2)$ from the Mayer-Vietoris sequence.
	We study $\ker \del = \im \beta_2$. By the first isomorphism theorem $\im \beta_2 = \frac{H_2(D) \oplus H_2(M_{\R \times T^2, 0})}{\im \alpha_2}$.
	
	Using the Mayer-Vietoris sequence that proves Proposition \ref{prop:topology:RxT2}, one can explicitly determine $H_2(M_{\R \times T^2, 0}) = \Z^3$. Namely, given some $r_1, r_2 \gg 1$ and $t_1, t_2 \in S^1$, two generators of $H_2(M_{\R \times T^2, 0})$ can be identified as the circle fibre over 
	$\{r_1\} \times S^1 \times \{t_1\}$ and $\{r_2\} \times \{t_2\} \times S^1$ inside $\R \times T^2$. That is, these generators are two distinct non-intersecting tori at infinity. The third generator can be identified with a $2$-cycle that is contained on the interior of $M_{\R\times T^2, 0}$.
	In summary, $H_2(M_{\R \times T^2, 0}) = \Z \oplus H_2(\del D)$, and so $\ker \del = \im \beta_2 = \H_2(D) \oplus \Z = \Z^2$.

	We explicitly give the generators of $\ker \del = \im \beta_2$. The first generator is the smooth anti-canonical divisor $K^{-1}$ of $\operatorname{Bl}_1\C P^2$ to which $D$ retracts.
	The second generator is the generator of $H_2(M_{\R \times T^2, 0}) = \Z \oplus H_2(\del D)$ that
	is contained in the interior of $M_{\R \times T^2, 0}$.
	We denote this generator as $C$. 
	
	Simultaneously, $H_2(Bl_1\C P^2)$ is generated by $K^{-1}$ and the generator $H$ of $H_2(\C P^2)$.
	In summary, the short exact sequence in Equation \eqref{eq:short-exact-sequence-boundary-map} simplifies to
	$$
	0 \to \langle K^{-1}, C \rangle \to \langle K^{-1}, H \rangle \xrightarrow{\del} \Z_4 \to 0,
	$$
	and this implies $C = 4 H + c \cdot K^{-1}$ for some $c \in \Z$.
	Both $C$ and $K^{-1}$ can be identified with $2$-cycles on the interior of their respective domains, and hence $K^{-1} \cdot C = 0$. Recall that on $Bl_1 \C P^2$, $H \cdot K^{-1} = 3$. 
	Therefore, 
	$$
	K^{-1} \cdot C = 4 K^{-1} \cdot H + c \: K^{-1} \cdot K^{-1} = 12 + 8c = 0.
	$$
	This implies $c = - \frac{3}{2}$, which is not an integer.
\end{proof}

Because $M_{T^2, 0}$ is not diffeomorphic to $\operatorname{Bl}_1 \C P^2$, it must be diffeomorphic to $S^2 \times S^2$. This concludes the remaining part of Proposition \ref{main-proposition-3}.

\appendix
\section{Elliptic analysis for ALH/ALH*-gravitational instantons}
\label{sec:appendix}
Almost all elliptic analysis needed for this paper is already done in \mycite{Salm001}. However, the choice of weight function in that paper is not suitable for ALH* manifolds. Namely, for the ALH/ALH* case he considered weighted analysis for \textit{exponentially decaying} functions. As a helpful reviewer pointed out, the non-linear term in Proposition \ref{prop:inverse-function-theory:non-linear-part} may also contain \textit{polynomially decaying} terms.

In this appendix, we set up a weighted norm for $M_{\R \times T^2, n}$, so that we can study polynomially decaying functions. We show that in this new case, the elliptic estimates for $\Omega^{-2} \Delta^g$ are also uniformly bounded in $\epsilon$. We also show that $\Omega^{-2} \Delta^g$ is Fredholm under this new norm, and that
$$
\Omega^{-2} \Delta^g \colon 
	C^{k+2, \alpha}_\delta(M_{B,n}) \oplus \R \phi 
\to C^{k, \alpha}_\delta(M_{B,n})
$$ 
is an isomorphism.

Although we will set up a new weighted norm, most of the calculations are already done in \mycite{Salm001}. Namely, the local weighted analysis of ALH/ALH* gravitational instantons shares many similarities with the local weighted analysis for ALG/ALG* gravitational instantons. Therefore, this appendix provides a brief explanation of the modifications needed to the arguments in \mycite{Salm001}.

\subsection{Weighted analysis}
In this appendix we consider the setup given in Section \ref{sec:weighted-analysis-of-functions}. Explicitly, let $\Omega$ and $\rho$ be strictly positive, smooth functions on $M_{\R \times T^2, n}$  such that outside some large compact set we have
\begin{equation}
	\label{eq:appendix:def-omega-rho-near-infinity}
	\Omega := r^{-1} h_\epsilon^{-\frac{1}{2}}
	\quad\text{and}\quad
	\rho := \log r.
\end{equation}
Let $\widetilde{\Vol}$ be a volume form on $M_{\R \times T^2, n}$ such that outside some large compact set it is given by
$$
\widetilde{\Vol} = \d\!\rho \wedge \Vol_{T^2} \wedge \eta.
$$
Next, consider the conformally rescaled norm $g_{cf} := \Omega^2 g$ and let $C^{k, \alpha}_{cf}$ be the standard H\"older norm with respect to $g_{cf}$. For $\delta \in \R$, we define the weighted H\"older and Sobolev norms
\begin{align*}
	\|u\|_{C^{k, \alpha}_{\delta}(M_{\R \times T^2,n})} =& \|e^{- \delta \rho} u\|_{C^{k, \alpha}_{cf}(M_{\R \times T^2,n})}, \\
	\|u\|^2_{{L}^{2}_{\delta}(M_{\R \times T^2,n})} =& \int_{M_{\R \times T^2, n}} \left|e^{- \delta \rho} u\right|^2 \widetilde{\Vol}, \\
	\|u\|^2_{W^{k, 2}_{\delta}(M_{\R \times T^2,n})} =& 
	\sum_{j=0}^k \int_{M_{\R \times T^2, n}} \left|\nabla^j_{g_{cf} }(e^{- \delta \rho} u)\right|^2_{g_{cf}} \widetilde{\Vol}.
\end{align*}
These norms are almost the same as in \cite{Salm001}, but in that paper he considered $\Omega = h_\epsilon^{-\frac{1}{2}}$ and $\rho = r$.

In this section, we show that there are weighted estimates that are uniformly bounded in $\epsilon$. This will be a crucial result needed in Theorem \ref{thm:global-analysis:bounded-inverse}.

\begin{proposition}
	\label{prop:appendix:elliptic-regularity}
	Let $k \in \N$, $\alpha \in (0,1)$, $\delta \in \R$ and let $P$ be the asymptotic region of $M_{\R \times T^2, n}$ as defined in Definition \ref{def:setup:principal-bundle}. Let 
	$P'$ be a neighbourhood of $P$. There exists a $C > 0$, uniformly in $\epsilon$, such that if $u \in C^{k+2, \alpha}_{\delta}(P')$ (or $u \in W^{k+2, 2}_{\delta}(P')$), then
	\begin{align*}
		\|u\|_{C^{k+2, \alpha}_\delta(P)} \le& C \left[
	\|\Omega^{-2} \Delta^g u\|_{C^{k, \alpha}_\delta(P')}
	+
	\|u\|_{C^{0}_\delta(P')}
	\right], \text{or} \\
	\|u\|_{W^{k+2, 2}_\delta(P)} \le& C \left[
	\|\Omega^{-2} \Delta^g u\|_{W^{k, 2}_\delta(P')}
	+
	\|u\|_{L^{2}_\delta(P')}
	\right] \text{respectively.}
	\end{align*}
\end{proposition}
\begin{proof}
	We follow the proofs of Theorem 1.4 and Theorem 2.13 in \cite{Salm001}. These proofs consist of the following steps:
	\begin{enumerate}
		\item Show that the universal cover $\hat{P}$ of $P$ has bounded geometry and that the bounds of this geometry are uniform with respect to $\epsilon$. That is, show that (the derivatives of) the curvature are uniformly bounded in $\epsilon$ and show that the injectivity radius is uniformly bounded below.
		\item Show that $\Omega^{-2} \Delta^g \colon 
		C^{k+2, \alpha}_\delta(M_{B,n}) \oplus \R \phi 
	\to C^{k, \alpha}_\delta(M_{B,n})$ is a uniformly elliptic operator on $\hat{P}$.
		\item Relate the elliptic estimates on $\hat{P}$ to estimates on $P$.
	\end{enumerate}
	To show that the curvature tensor of $\hat{P}$ is uniformly bounded in all its derivatives, one can use the Koszul formula on the orthonormal co-frame
	$
	\{\d \rho, e^{-\rho} \d \theta, e^{-\rho}\d \phi, \frac{\epsilon}{e^{\rho} h_\epsilon} \eta \}.
	$
	In practice, we only need to take the exterior derivative of this co-frame and the only non-zero terms are
	\begin{align*}
		\d (e^{-\rho} \d \theta) =& - \d \rho \wedge e^{-\rho} \d \theta, \quad
		\d (e^{-\rho} \d \phi) = - \d \rho \wedge e^{-\rho} \d \phi, \text{ and } \\
		\d \left(\frac{\epsilon}{e^{\rho} h_\epsilon} \eta\right) =& - ( \d \rho + \d \log h_\epsilon)\wedge
		\frac{\epsilon}{e^{\rho} h_\epsilon} \eta + *^{cf} (\d \log h_\epsilon \wedge \frac{\epsilon}{e^{\rho} h_\epsilon} \eta).
	\end{align*}
	As in the ALG/ALG* case, the derivatives of $\d \log h_\epsilon$ provide bounds on (the derivatives of) the curvature tensor. The former are uniformly bounded, because Lemma \ref{lem:setup:harmonic-function} implies $h_\epsilon = 1 + \epsilon A + \epsilon B e^{\rho} + \O(e^{- \rho})$ for some $A, B > 0$. Thus,
	$$
	\d \log h_\epsilon = \frac{\epsilon B e^{\rho}}{h_\epsilon} \d \rho + \O(e^{- \rho})
	= \left(1 -\frac{1 + \epsilon A}{h_\epsilon} \right) \d \rho + \O(e^{- \rho}).
	$$

	Next we show that $\hat{P}$ has a uniform lower bound on the injectivity radius. As in \cite{Salm001} we use \cite{Cheeger1982}, which states that it is sufficient to show that a geodesic ball has quartic volume growth. 
	Hence, fix $R > 0$ and consider the `rectangular' neighbourhood
	$$
	\operatorname{Rect}_R (\rho_0) := 
	\left\{ (\rho, \theta, \phi, t) \in \hat{P} \colon  
	|\rho- \rho_0| < R, |\theta| < e^\rho R, |\phi| < e^\rho R, |t| < \frac{e^{\rho}h_\epsilon}{\epsilon} R \right\}.
	$$
	This rectangle has a circumscribed ball and an inscribed ball, whose radii do not depend on $\epsilon$.
	Therefore, it is sufficient to determine the volume growth of $\operatorname{Rect}_R (\rho_0)$, which is $(2 R)^4$ by construction.

	Having bounded geometry on $\hat{P}$, we can now study the operator $\Omega^{-2} \Delta^g \colon 
	C^{k+2, \alpha}_\delta(\hat{P}) \to C^{k, \alpha}_\delta(\hat{P})$. As explained in \cite{Salm001}, it is equivalent to studying $L_\delta\colon 
	C^{k+2, \alpha}_{cf}(\hat{P}) \to C^{k, \alpha}_{cf}(\hat{P})$, which is defined as
	$$
	L_\delta := e^{- \delta \rho} \Omega^{-2} \Delta^g(e^{\delta \rho} \ldots).
	$$
	Expanding this explicitly, one gets
	$$
	L_\delta u = \Delta_{g_{cf}} u + 2 \langle \d \log \Omega - \delta \d \rho, \d u\rangle_{g_{cf}}
	+ u \left(
		- \delta^2 - \delta \nabla^*_{cf} \d \rho + 2\delta \langle \d \log \Omega, \d \rho \rangle_{g_{cf}}
	\right).
	$$
	Therefore, $L_\delta$ is uniformly elliptic with respect to $\epsilon$ if and only if $\d \log \Omega$ is uniformly bounded. This is indeed true by the definition of $\Omega$.

	With all these results we can consider the elliptic estimates for $\Omega^{-2} \Delta^g$. According to Theorem 1.2 in \cite{Hebey1999}, there exists an $R_h > 0$, independent of $\epsilon$, such that on any ball of radius $R_h$ in $\hat{P}$ there are so-called harmonic coordinates. In these coordinates $g_{cf}$ is $C^{k, \alpha}$-equivalent to the flat metric. Hence, the standard Schauder estimate implies that for all $\rho_0 > 0$,
	\begin{align*}
		\|u\|_{C^{k+2, \alpha}_\delta(\operatorname{Rect}_{\frac{R_h}{2}} (\rho_0))} \le& C \left[
	\|\Omega^{-2} \Delta^g u\|_{C^{k, \alpha}_\delta(\operatorname{Rect}_{R_h} (\rho_0))}
	+
	\|u\|_{C^{0}_\delta(\operatorname{Rect}_{R_h} (\rho_0))} 
	\right], \\
	\|u\|_{W^{k+2, 2}_\delta(\operatorname{Rect}_{\frac{R_h}{2}} (\rho_0))} \le& C \left[
	\|\Omega^{-2} \Delta^g u\|_{W^{k, 2}_\delta(\operatorname{Rect}_{R_h} (\rho_0))}
	+
	\|u\|_{L^{2}_\delta(\operatorname{Rect}_{R_h} (\rho_0))} 
	\right],
	\end{align*}
	for some uniform $C > 0$.
	Finally, notice that $\operatorname{Rect}_{R_h} (\rho_0)$ contains approximately $\frac{e^{3 \rho_0} R_h^3}{\epsilon}$ fundamental domains. By keeping track of this multiplicity in the volume form and using the same standard patching arguments as done in \cite{Salm001}, we get the results stated in the proposition.
\end{proof}

\subsection{Fredholm theory}
Having these uniform estimates we can ask whether $\Omega^{-2} \Delta^g$ is Fredholm under these new norms. Like in \cite{Salm001}, this is a direct consequence of the following Proposition:

\begin{proposition}
	\label{prop:appendix:Fredholm-estimate}
	Let $k \in \N$, $\alpha \in (0,1)$, $\delta \in \R \setminus \{0,1\}$ and let $P$ be the asymptotic region of $M_{\R \times T^2, n}$ as defined in Definition \ref{def:setup:principal-bundle}. Let 
	$P'$ be a neighbourhood of $P$, such that $P' \setminus P$ is compact.
	There exists a $C > 0$, uniformly in $\epsilon$, such that if $u \in C^{2, \alpha}_{\delta}(P')$ (or $u \in W^{2, 2}_{\delta}(P')$), then
	\begin{align*}
		\|u\|_{C^{2, \alpha}_\delta(P)} \le& C \left[
	\|\Omega^{-2} \Delta^g u\|_{C^{0, \alpha}_\delta(P')}
	+
	\|u\|_{C^{0}_\delta(P'\setminus P)}
	\right], \text{or} \\
	\|u\|_{W^{2,2}_\delta(P)} \le& C \left[
	\|\Omega^{-2} \Delta^g u\|_{L^{2}_\delta(P')}
	+
	\|u\|_{L^{2}_\delta(P'\setminus P)}
	\right] \text{respectively.}
	\end{align*}
\end{proposition}
\begin{proof}
	The argument is identical to the proof of Theorem 1.5 in \cite{Salm001}. Also, the argument is identical for H\"older spaces and Sobolev spaces.
	Namely, we split $u$ into the Fourier modes of its various fibrations. For example, write $u = u_b + u_f$, where $u_b(x) = \frac{1}{2 \pi}\int_{S^1 \cdot \{x\}} u \eta$. Proposition 3.4 in \cite{Salm001} gives us a Poincar\'e inequality
	$$
	\|u_f\|_{C^0_{cf}(S^1 \cdot \{x\})} \le 2 \pi \frac{\epsilon \Omega}{\sqrt{h_\epsilon}} \|\d u_f\|_{C^0_{cf}(S^1 \cdot \{x\})},
	$$
	which combined with Proposition \ref{prop:appendix:elliptic-regularity},
	yields the estimate
	$$
	\|u_f\|_{C^{2, \alpha}_\delta(P)} \le C \left[
	\|\Omega^{-2} \Delta^g u_f\|_{C^{0, \alpha}_\delta(P')}
	+
	\|u_f\|_{C^{0}_\delta(P'\setminus P)}
	\right],
	$$
	where one may need to redefine $P$ and $P'$.
	Like in \cite{Salm001}, we repeat the above argument for all fibrations in $\R \times T^2$. In the end we write $u_b(r, \theta, \phi) = u_{bb}(r) + u_{bf} (r, \theta, \phi)$, for which we have
	$$
	\|u_{bf}\|_{C^{2, \alpha}_\delta(P)} \le C \left[
	\|\Omega^{-2} \Delta^g u_{bf}\|_{C^{0, \alpha}_\delta(P')}
	+
	\|u_{bf}\|_{C^{0}_\delta(P'\setminus P)}
	\right].
	$$
	If we have the estimate
	\begin{equation}
		\label{eq:appendix:fredholm-on-line}
		\|u_{bb}\|_{C^{2, \alpha}_\delta(P)} \le C \left[
	\|\Omega^{-2} \Delta^g u_{bb}\|_{C^{0, \alpha}_\delta(P')}
	+
	\|u_{bb}\|_{C^{0}_\delta(P'\setminus P)}
	\right],
	\end{equation}
	then Lemmas 3.2 and 3.3 in \cite{Salm001} imply
	\begin{align*}
		\|u\|_{C^{2, \alpha}_\delta(P)} 
		\le& \|u_{bb}\|_{C^{2, \alpha}_\delta(P)} + \|u_{bf}\|_{C^{2, \alpha}_\delta(P)} + \|u_f\|_{C^{2, \alpha}_\delta(P)} \\
		\le & C \left[
		\|\Omega^{-2} \Delta^g u_{bb}\|_{C^{0, \alpha}_\delta(P')} 
		+ \|\Omega^{-2} \Delta^g u_{bf}\|_{C^{0, \alpha}_\delta(P')} 
		+ \|\Omega^{-2} \Delta^g u_{f}\|_{C^{0, \alpha}_\delta(P')} 
		\right. \\
		& \left. \qquad
		+ \|u_{bb}\|_{C^{0}_\delta(P'\setminus P)}
		+ \|u_{bf}\|_{C^{0}_\delta(P'\setminus P)}
		+ \|u_{f}\|_{C^{0}_\delta(P'\setminus P)}
		\right] \\
		\le & C \left[
		\|\Omega^{-2} \Delta^g u\|_{C^{0, \alpha}_\delta(P')} 
		+ \|u\|_{C^{0}_\delta(P'\setminus P)}
		\right].
	\end{align*}
	This is the estimate we need to prove.

	To prove Equation \eqref{eq:appendix:fredholm-on-line}, denote $f_{bb} =\Delta^g u_{bb}$. We will find another $\tilde u \in C^{2, \alpha}_{\delta}(P')$, such that $\Omega^{-2} \Delta^g \tilde u = f_{bb}$ and $\tilde u$ only depends on $\rho$. 
	In this case, the equation $\Omega^{-2} \Delta^g \tilde u = f_{bb}$ simplifies to
	\begin{equation}
		\label{eq:appendix:equation-for-on-line}
		\frac{\del^2 \tilde{u}}{\del r^2} = - r^2 f_{bb}.
	\end{equation}
	Using variation of parameters, we find
	$$
	\tilde u = \int_{r_0}^r r^{-1} f_{bb} \d r - r \int_{r_1}^r r^{-2} f_{bb} \d r
	$$
	is a solution to Equation \eqref{eq:appendix:equation-for-on-line} for any choice\footnote{One might need to multiply $f_{bb}$ with a step function  in order to make this well-defined outside $P'$.} of $r_0$ and $r_1$.
	According to the proof of Lemma 6.2.1 in \cite{Pacard2008} if
	$$
	r_0 = \begin{cases}
		0 & \delta > 0 \\
		\infty & \delta < 0 \\
	\end{cases}
	\qquad \text{and} \qquad
	r_1 = \begin{cases}
		0 & \delta > 1 \\
		\infty & \delta < 1, \\
	\end{cases}
	$$
	then $\|\tilde u\|_{C^{0}_\delta(P)} \le C \|f_{bb}\|_{C^0_\delta(P')}$  or $\|\tilde u\|_{L^{2}_\delta(P)} \le C \|f_{bb}\|_{L^2_\delta(P')}$ whenever 
	$f \in C^0_\delta(P')$ or $f \in L^2_\delta(P')$ respectively.

	Next, notice that $u_{bb} - \tilde u$ is harmonic and hence it must be equal to $A + Br$ for some $A,B \in \R$. Also, $\tilde u \in C^{2, \alpha}_\delta(P)$ by Proposition \ref{prop:appendix:elliptic-regularity} and hence,
	\begin{align*}
		\| u_{bb} \|_{C^{2, \alpha}_\delta (P)} 
		\le& \| u_{bb} - \tilde u \|_{C^{2, \alpha}_\delta (P)} + \| \tilde u \|_{C^{2, \alpha}_\delta (P)} \\
		\le & C \left[
			\| \Delta \tilde u \|_{C^{0, \alpha}_\delta(P')}
			+ \| u_{bb} - \tilde u \|_{C^{0}_\delta (P')}
			+ \| \tilde u \|_{C^{0}_\delta (P')} 
		\right] \\
		\le & C \left[
			2\| f_{bb} \|_{C^{0, \alpha}_\delta(P')}
			+ \| A + B r \|_{C^{0}_\delta (P')}
		\right].
	\end{align*}
	When $\delta < 0$, $A = 0$ and when $\delta < 1$, $B = 0$. In any other cases $\| A + B r \|_{C^{0}_\delta (P')}$ is bounded by $\| A + B r \|_{C^{0}_\delta (P' \setminus P)}$. Therefore, 
	\begin{align*}
		\| u_{bb} \|_{C^{2, \alpha}_\delta (P)} 
		\le & C \left[
			2\| f_{bb} \|_{C^{0, \alpha}_\delta(P')}
			+ \| A + B r \|_{C^{0}_\delta (P' \setminus P)}
		\right] \\
		\le & C \left[
			2\| f_{bb} \|_{C^{0, \alpha}_\delta(P')}
			+ \| u_{bb} - \tilde u \|_{C^{0}_\delta (P' \setminus P)}
		\right] \\
		\le & C \left[
			3\| f_{bb} \|_{C^{0, \alpha}_\delta(P')}
			+ \| u_{bb} \|_{C^{0}_\delta (P' \setminus P)} 
		\right]. \qedhere
	\end{align*}
\end{proof}
Now that we know that $\Omega^{-2} \Delta^g$ is Fredholm, we can determine the kernel and co-kernel of $\Omega^{-2} \Delta^g$. Using the definition of $\widetilde{\Vol}$, one can see that the formal adjoint of $L_\delta$ is equal to $L_{1- \delta}$. Hence, using the maximum principle we can conclude
\begin{lemma}
	When $\delta < 0$, then
	\begin{align*}
		\Omega^{-2} \Delta^g \colon C^{2, \alpha}_\delta (M_{\R \times T^2, n}) \to C^{0, \alpha}_\delta(M_{\R \times T^2, n}), \\
		\Omega^{-2} \Delta^g \colon W^{2, 2}_\delta (M_{\R \times T^2, n}) \to L^{2}_\delta(M_{\R \times T^2, n})
	\end{align*}
	are injective.
	When $\delta > 1$, then
	\begin{align*}
		\Omega^{-2} \Delta^g \colon C^{2, \alpha}_\delta (M_{\R \times T^2, n}) \to C^{0, \alpha}_\delta(M_{\R \times T^2, n}), \\
		\Omega^{-2} \Delta^g \colon W^{2, 2}_\delta (M_{\R \times T^2, n}) \to L^{2}_\delta(M_{\R \times T^2, n})
	\end{align*}
	are surjective.
\end{lemma}
Like in \cite{Salm001}, there is no $\delta \in \R$ such that $\Omega^{-2} \Delta^g \colon C^{2, \alpha}_\delta (M_{\R \times T^2, n}) \to C^{0, \alpha}_\delta(M_{\R \times T^2, n})$ is an isomorphism. However, we know how the kernel and co-kernel behaves for small values of $\delta$ when we consider exponentially decaying weight functions. With this extra information we can make $\Omega^{-2} \Delta^g$ an isomorphism in this new norm.

\begin{theorem}
	\label{thm:appendix:isomorphism}
	Let $\delta \in (-1, 0)$, and let $\phi$ be a non-negative function on $\R^2 \times S^1$, such that 
	$\phi = r$ on $P$, then
	\begin{align*}
		\Omega^{-2} \Delta^g \colon C^{2, \alpha}_\delta (M_{\R \times T^2, n}) \oplus \R \phi \to C^{0, \alpha}_\delta(M_{\R \times T^2, n})
	\end{align*}
	is an isomorphism.
\end{theorem}
\begin{proof}
	Let $C^{k, \alpha}_{\delta, \exp}$ be the originally weighted spaces in \cite{Salm001}. We know that for $0< -\delta \ll 1$, 
	\begin{align*}
		\Omega^{-2} \Delta^g \colon C^{2, \alpha}_{\delta, \exp} (M_{\R \times T^2, n}) \oplus \R \phi \to C^{0, \alpha}_{\delta, \exp}(M_{\R \times T^2, n}),
	\end{align*}
	is an isomorphism. 
	Hence, 
	\begin{align*}
		\Omega^{-2} \Delta^g \colon C^{2, \alpha}_{\delta, \exp} (M_{\R \times T^2, n})  \to C^{0, \alpha}_{\delta, \exp}(M_{\R \times T^2, n}),
	\end{align*}
	is injective and has a 1-dimensional co-kernel.
	The formal adjoint of $\Omega^{-2} \Delta^g \colon C^{2, \alpha}_{\delta, \exp} (M_{\R \times T^2, n})  \to C^{0, \alpha}_{\delta, \exp}(M_{\R \times T^2, n})$ is $\Omega^{-2} \Delta^g \colon C^{2, \alpha}_{-\delta, \exp} (M_{\R \times T^2, n})  \to C^{0, \alpha}_{-\delta, \exp}(M_{\R \times T^2, n})$. Therefore, for $0 < \delta \ll 1$, the operator
	\begin{align*}
		\Omega^{-2} \Delta^g \colon C^{2, \alpha}_{\delta, \exp} (M_{\R \times T^2, n})  \to C^{0, \alpha}_{\delta, \exp}(M_{\R \times T^2, n}),
	\end{align*}
	is surjective and its kernel only consists of the constant functions.

	With this in mind, fix $\delta \in (-1, 0)$ and $0 < \delta' \ll 1$ and consider the formal adjoint of
	$$
	\Omega^{-2} \Delta^g \colon C^{2, \alpha}_\delta (M_{\R \times T^2, n}) \to C^{0, \alpha}_\delta(M_{\R \times T^2, n}).
	$$
	This formal adjoint is
	$$
	\Omega^{-2} \Delta^g \colon C^{2, \alpha}_{1- \delta} (M_{\R \times T^2, n}) \to C^{0, \alpha}_{1-\delta}(M_{\R \times T^2, n}).
	$$
	Hence, if $u$ lies in the kernel of this formal adjoint, then $u$ is polynomially growing. Therefore, $u \in C^{2, \alpha}_{\delta', \exp}(M_{\R \times T^2, n})$ and so $u$ is constant. We conclude that the operator $$
	\Omega^{-2} \Delta^g \colon C^{2, \alpha}_\delta (M_{\R \times T^2, n}) \to C^{0, \alpha}_\delta(M_{\R \times T^2, n})
	$$
	has no kernel and has a 1-dimensional co-kernel.

	Finally, we consider the extension
	$$
	\Omega^{-2} \Delta^g \colon C^{2, \alpha}_\delta (M_{\R \times T^2, n}) \oplus \R \phi \to C^{0, \alpha}_\delta(M_{\R \times T^2, n}).
	$$
	Either the dimension of the kernel increases by one, or the dimension of the co-kernel decreases by one.
	If the dimension of the kernel increases by one, then there exists a $u \in C^{2, \alpha}_\delta (M_{\R \times T^2, n})$ and $\lambda \in \R$ such that 
	$$
	\Delta (u + \lambda \phi) = 0.
	$$
	$u + \phi$ is a polynomially growing harmonic function, and so $u + \phi \in C^{2, \alpha}_{\delta', \exp}$. By the previous discussion, there exists a $c > 0$ such that
	$$
	u + \lambda \phi = c.
	$$
	Due to the growth rate of $u$ and $\phi$, this is a contradiction.
\end{proof}
\bibliography{sn-bibliography}
\end{document}

%% file: images/mayer-vietoris-r3/mayer-vietoris-r3.pdf_tex
\begingroup%
  \makeatletter%
  \providecommand\color[2][]{%
    \errmessage{(Inkscape) Color is used for the text in Inkscape, but the package 'color.sty' is not loaded}%
    \renewcommand\color[2][]{}%
  }%
  \providecommand\transparent[1]{%
    \errmessage{(Inkscape) Transparency is used (non-zero) for the text in Inkscape, but the package 'transparent.sty' is not loaded}%
    \renewcommand\transparent[1]{}%
  }%
  \providecommand\rotatebox[2]{#2}%
  \newcommand*\fsize{\dimexpr\f@size pt\relax}%
  \newcommand*\lineheight[1]{\fontsize{\fsize}{#1\fsize}\selectfont}%
  \ifx\svgwidth\undefined%
    \setlength{\unitlength}{1064.5737247bp}%
    \ifx\svgscale\undefined%
      \relax%
    \else%
      \setlength{\unitlength}{\unitlength * \real{\svgscale}}%
    \fi%
  \else%
    \setlength{\unitlength}{\svgwidth}%
  \fi%
  \global\let\svgwidth\undefined%
  \global\let\svgscale\undefined%
  \makeatother%
  \begin{picture}(1,0.24118202)%
    \lineheight{1}%
    \setlength\tabcolsep{0pt}%
    \put(0,0){\includegraphics[width=\unitlength,page=1]{images/mayer-vietoris-r3/mayer-vietoris-r3.pdf}}%
    \put(0.0281803,0.18115904){\color[rgb]{0,0,0}\makebox(0,0)[lt]{\lineheight{1.25}\smash{\begin{tabular}[t]{l}$Y$\end{tabular}}}}%
    \put(0.0281803,0.10567612){\color[rgb]{0,0,0}\makebox(0,0)[lt]{\lineheight{1.25}\smash{\begin{tabular}[t]{l}$X$\end{tabular}}}}%
    \put(0.02918674,0.03623181){\color[rgb]{0,0,0}\makebox(0,0)[lt]{\lineheight{1.25}\smash{\begin{tabular}[t]{l}$\Z_2 \cdot Y$\end{tabular}}}}%
    \put(0.46698772,0.19021698){\color[rgb]{0,0,0}\makebox(0,0)[lt]{\lineheight{1.25}\smash{\begin{tabular}[t]{l}$p_1$\end{tabular}}}}%
    \put(0.56399221,0.20436507){\color[rgb]{0,0,0}\makebox(0,0)[lt]{\lineheight{1.25}\smash{\begin{tabular}[t]{l}$p_2$\end{tabular}}}}%
    \put(0.4004258,0.0350241){\color[rgb]{0,0,0}\makebox(0,0)[lt]{\lineheight{1.25}\smash{\begin{tabular}[t]{l}$-p_2$\end{tabular}}}}%
    \put(0.56346891,0.0462881){\color[rgb]{0,0,0}\makebox(0,0)[lt]{\lineheight{1.25}\smash{\begin{tabular}[t]{l}$-p_1$\end{tabular}}}}%
    \put(0.5163032,0.13083706){\color[rgb]{0,0,0}\makebox(0,0)[lt]{\lineheight{1.25}\smash{\begin{tabular}[t]{l}$(0,0)$\end{tabular}}}}%
  \end{picture}%
\endgroup%